\documentclass[reqno,A4paper]{amsart}

\usepackage{amsmath}
\usepackage{amssymb}
\usepackage{amsthm}
\usepackage{enumerate}
\usepackage{mathrsfs} 
\usepackage{eqlist}
\usepackage{array}

\usepackage[english]{babel}                             
\usepackage[latin1]{inputenc}
\usepackage{comment}
\usepackage{ae}
\usepackage[leqno]{amsmath}  

\usepackage[dotinlabels]{titletoc}     

\usepackage{amssymb,latexsym,verbatim} 
\usepackage{amstext}
\usepackage{amsfonts}
\usepackage{amsbsy}
\usepackage{amsopn}
\usepackage{enumerate}
\usepackage{txfonts}
\usepackage{amsmath, amsthm, amssymb}

\usepackage{color}
\usepackage[usenames,dvipsnames,svgnames,table]{xcolor}
\usepackage[numbers,sort&compress]{natbib}
\allowdisplaybreaks
\newtheorem{theorem}{Theorem}
\newtheorem{definition}[theorem]{Definition}
\newtheorem{lemma}[theorem]{Lemma}
\newtheorem{corollary}[theorem]{Corollary}
\newtheorem{proposition}[theorem]{Proposition}

\numberwithin{equation}{section}
\numberwithin{theorem}{section}

\renewcommand{\epsilon}{\varepsilon}
\renewcommand{\rho}{\varrho}

\DeclareMathOperator{\supp}{supp}
\DeclareMathOperator*{\esssup}{ess\,sup}
\DeclareMathOperator*{\essinf}{ess\,inf}
\DeclareMathOperator*{\essosc}{ess\,osc}
\DeclareMathOperator*{\loc}{loc}

\DeclareMathOperator*{\diam}{diam}

\numberwithin{equation}{section}

\begin{document}

\title[continuity of solutions to singular parabolic equations]%
{On the continuity of solutions to doubly singular parabolic equations}
\author[Qifan Li]%
{Qifan Li*}

\newcommand{\acr}{\newline\indent}

\address{\llap{*\,}Department of Mathematics\acr
                   School of Sciences\acr
                   Wuhan University of Technology\acr
                   430070, 122 Luoshi Road,
                   Wuhan, Hubei\acr
                   P. R. China}
\email{qifan\_li@yahoo.com, qifan\_li@whut.edu.cn}




\begin{abstract}
This paper considers a certain
doubly singular parabolic equations with one singularity occurs in the time derivative, whose model is
\begin{equation*} \partial_t\beta(u)-\operatorname{div}|Du|^{p-2}Du\ni0,\qquad \text{in}\quad \Omega\times(0,T)\end{equation*}
where $\Omega\subset\mathbb{R}^N$ and $N\geq3$.
We show that the bounded weak solutions are
locally continuous in the range
$$2-\epsilon_0\leq p<2,$$
provided $\epsilon_0>0$ is
small enough, and the continuity is stable as
$p\to2$.
\end{abstract}
\maketitle
\section{Introduction}
The aim of this paper is to establish a continuity result for solutions to singular parabolic equations
 related to the two phase Stefan problems of the type
\begin{equation}\label{two phase stefan}\partial_t\beta(u)-\operatorname{div}A(x,t,u,Du)\ni 0\qquad\text{in}\quad\Omega_T.\end{equation}
Here, the set $\Omega_T$ denotes the cylinder $\Omega\times(0,T)$ over an open bounded domain $\Omega\subset \mathbb{R}^{N}$
 with dimension $N\geq3$, and the function $\beta(\cdot)$ is defined by
\begin{equation}\label{beta}
	\beta(r)=\begin{cases}
	r,&\quad r>0,\\
	[-\bar\nu,0],&\quad r=0,\\ r-\bar\nu,&\quad r<0,
	\end{cases}
\end{equation}
where $\bar\nu>0$ is a given constant.
The vector field  $A(x,t,u,\zeta)$ is measurable in $\Omega_T\times
 \mathbb{R}\times\mathbb{R}^{N}$ and satisfy the structure conditions:
 \begin{equation}\label{A}
	\begin{cases}
	A(x,t,u,\zeta)\cdot\zeta\geq C_0|\zeta|^{p},\\
	|A(x,t,u,\zeta)|\leq C_1|\zeta|^{p-1},
	\end{cases}
\end{equation}
where $C_0$ and $C_1$ are given positive constants. No attempt has been made here to discuss the equations with lower order terms.
A motivation for this study comes
from phase transitions for the material obeying a Fourier's law (see \cite{LSU, R, U1997}).
The continuity of weak solutions for the case $A(x,t,u,Du)=Du$ was first proved by Caffarelli and Evans \cite{CE}. This result was
extended by DiBenedetto \cite{Di82} to the quasilinear equations with general structure. The proof is based on the
De Giorgi's technique and a construction of logarithmic functions to prove the expansion of positivity in time direction.
Furthermore, Urbano \cite{U2000} established the continuity result for the degenerate case
$p\geq2$ by using the De Giorgi's technique together with the intrinsic scaling method. This idea goes back at least as far as \cite{Di82}.

The treatment of the singular case $p<2$ is much more difficult.
This problem was studied by Henriques and Urbano \cite{HU}. They introduced the intrinsic parabolic cylinders of the form
\begin{equation*}[(\bar x,\bar t)+Q(dR^p,d_*R)],\qquad d=\left(\frac{\omega}{2}\right)^{(1-p)(2-p)}\quad\text{and}
\quad d_*=\left(\frac{\omega}{2^{n_*+1}}\right)^{p-2}\end{equation*}
and use this kind of cylinders to perform alternative arguments.
However, we point out that there is a gap in the proof of Lemma 1 (\cite[page 930, line 8]{HU}).
The integral involving the term $|\nabla (\bar \theta-\bar k_n)_-\bar \xi_n|^p$ must be multiplied
by a factor $c_0^{-p}$ due to the chain rule. If we multiply such factor, the constant $\nu_0$ in \cite[page 927]{HU} will depend
upon $2^{-n_*}$ which is so small that we cannot determine the number $n_*$ in terms of $\nu_0$ (see \cite[Lemma 4, page 938, line -1]{HU}).
Of course, this is not meant to diminish the importance of Henriques and Urbano's pioneering work, but rather to point out that
it does not resolve the problem and a correct proof is also needed.

In this work, we try to find an idea to overcome the difficulty mentioned above. We shall use the following intrinsic
parabolic cylinders
\begin{equation*}[(\bar x,\bar t)+Q(dR^p,d_*R)],\quad d=\left(\frac{\omega}{2^{m_1}}\right)^{(1-p)(2-p)}\quad\text{and}
\quad d_*=\left(\frac{\omega}{2^{m_2}}\right)^{p-2}\quad\text{where}\quad m_1=\frac{pm_2}{p-1}\end{equation*}
for the alternative arguments. Since the length of the time interval is too large, we cannot simply use the idea from
\cite[Lemma 4]{HU}) to determine the quantity $m_2$ and the analysis of the second alternative is much more delicate.
In order to obtain a result of expansion positivity in such large time interval,
we shall apply an iterative argument. However,
our approach only works for the range
$2-\epsilon_0\leq p<2,$
where $\epsilon_0$ is
a small positive constant depending only upon the data. We have to impose this artificial condition, because the integrals involving the quantity $[\psi_{\pm}']^{2-p}$,
where $\psi_{\pm} $ is a logarithmic function, are too large in the singular range.
Moreover, we can also show that the continuity is stable as $p\to2$.
As for the proof of the full range $1<p<2$, we leave this problem open to the further study.

An outline of this paper is as follows. We provide some preliminary material and state the main result in \S 2, while \S 3 provides an exposition
of Caccioppoli estimates.
In \S 4, we introduce the
intrinsic parabolic cylinders for the alternative arguments and get started the proof of the main result. \S 5 is devoted to the analysis of
the first alternative. We obtain a
decay estimate for the oscillation of weak solution in a smaller cylinder. Subsequently, \S 6 is intended to prove a similar estimate
in the second alternative.
Finally, in \S 7, we finish the proof of the main result by a recursive argument.
\section{Preliminary material and main result}
We follow the notation of \cite{Di93}.
Throughout the paper, $\Omega$ denote a bounded open set in $\mathbb{R}^N$ with $N\geq 3$ and $\Omega_T=\Omega\times(0,T)$ is the associated
space-time cylinder.
We let $Du=(D_{x_1}u,D_{x_2}u,\cdots,D_{x_N}u)$ stand for the differentiation with respect to the space variables and let $\partial_tu$ denote the time derivative. The derivatives are teken in the weak sense. Points
in $\mathbb{R}^{N+1}$ will be denoted by $(x,t)$, where $x\in \mathbb{R}^{N}$ and $t\in \mathbb{R}$. We shall use cubes of the form
$$K_r=\{x=(x_1,x_2,\cdots,x_N)\in\mathbb{R}^N:|x_i|\leq r,\ i=1,2,\cdots ,N\}.$$
Concerning parabolic cylinders, we use the cylinders of the form $Q(r_1,r_2)=K_{r_2}\times(-r_1,0)$.
Let $\partial K_{r_2}$ be the boundary of $K_{r_2}$ and
$\partial_PQ(r_1,r_2)=[\partial K_{r_2}\times(-r_1,0)]\cup [K_{r_2}\times\{0\}]$
denotes the parabolic boundary of $Q(r_1,r_2)$.
Let $r'>r>0$, we construct the piecewise smooth function, with $\zeta\equiv 1$ in $K_r$, $\supp\zeta\subset K_{r'}$ and $|D\zeta|\leq (r'-r)^{-1}$, by
\begin{equation}\label{zeta}\zeta(x)=\prod_{i=1}^N\zeta_i(x_i)\end{equation}
where $x=(x_1,x_2,\cdots,x_N)$ and
\begin{equation}\label{zetai}
	\zeta_i(x_i)=\begin{cases}
0,\qquad &\text{if}\quad |x_i|>r',\\
\dfrac{x_i+r'}{r'-r},\qquad &\text{if}\quad -r'<x_i<-r,\\
1,\qquad &\text{if}\quad |x_i|\leq r,\\
\dfrac{r'-x_i}{r'-r},\qquad &\text{if}\quad r<x_i<r'.
	\end{cases}
\end{equation}
We set $W^{1,p}(\Omega)$ for the Sobolev space of weakly differentiable functions
$v:\Omega\to\mathbb{R}$ with $|v|,|Dv|\in L^p(\Omega)$. Similarly, a function $v\in W_{\loc}^{1,p}(\Omega)$ if $v\in W^{1,p}(K)$ for
all compact set $K\subset \Omega$.
We denote by $W_0^{1,p}(\Omega)$ the closure of $C_0^\infty(\Omega)$
with respect to the norm in $W^{1,p}(\Omega)$. For any $k\in \mathbb{R}$ and a function $v\in W^{1,p}(\Omega)$, the truncations
are defined by
\begin{equation*}\begin{split}(v-k)_+&=\max\{v-k,0\}
\\(v-k)_-&=\max\{k-v,0\}.\end{split}\end{equation*}
Let $B$ be a Banach space and $1\leq q<\infty$. We write $L^q(0,T,B)$
for the space of $L^q$ integrable functions from $[0,T]$ into $B$, which is a Banach space with the norm
$$\|v\|_{L^q(0,T;B)}=\left(\int_0^T\|v(t)\|_B^qdt\right)^{\frac{1}{q}}.$$
Moreover, we write $L_{\loc}^q(0,T,B)$ for space of functions $v\in L^q(t_1,t_2,B)$ for all intervals $[t_1,t_2]\subset(0,T]$.
Let us denote by $C_{\loc}(0,T;B)$ the space of locally continuous functions from $[0,T]$ into $B$. Finally, we denote by $W_p(\Omega_T)$ the set
of functions
$$v\in L_{\loc}^p(0,T;W_0^{1,p}(\Omega))\qquad\text{and}\qquad \partial_tv\in L_{\loc}^2(\Omega_T).$$
We are now ready to give the definition of the weak solution to the doubly singular parabolic equations
\eqref{two phase stefan}-\eqref{A}:
\begin{definition}A function $u:\Omega_T\to \mathbb{R}$ is said to be a local weak solution of \eqref{two phase stefan}-\eqref{A}, if
\begin{equation*}u\in C_{\loc}(0,T;L_{\loc}^2(\Omega))\cap L_{\loc}^p(0,T;W_{\loc}^{1,p}(\Omega))\end{equation*}
and there exists a function $w\subset\beta(u)$, with
the inclusion being intended in the sense of the graphs, such that the identity
\begin{equation}\label{weak form two phase stefan}\int_\Omega w(\cdot,t)\varphi(\cdot,t)dx\bigg|_{t=t_0}^{t_1}+\int_{t_0}^{t_1}\int_\Omega
\left[-w\partial_t \varphi+A(x,t,u,Du)\cdot D\varphi\right]dxdt=0 \end{equation}
holds for all testing functions $\varphi\in W_p(\Omega_T)$ and all intervals $[t_0,t_1]\subset(0,T]$.
\end{definition}
Throughout the paper, we assume that the weak solution is bounded and its time derivative is locally square integrable. To be
more precise, we assume that
\begin{equation}\label{boundedness}\esssup_{\Omega_T}|u|\leq 1,\end{equation} and the time derivative $\partial_t u$ exists and satisfies
\begin{equation}\label{integrability}\partial_t u
\in L_{\loc}^2(\Omega_T).\end{equation}
The
statement that a constant $C$ depends only upon the data, means that
it can be determined a priori only in terms of $\{\bar \nu,\ N,\ C_0,\ C_1\}$. Since we are concerned about the behaviour of constants as $p\to2$,
exponent $p$ should be excluded from the concept
of the data.
We can now state our main result:
\begin{theorem}\label{main theorem}
There exists a constant $\epsilon_0>0$, that can be determined
a priori only in terms of the data, such that the following holds:
Let $u$ be a weak solution to the singular parabolic equation \eqref{two phase stefan}-\eqref{A} with $p\in[2-\epsilon_0,2)$. Suppose that
the assumptions
\eqref{boundedness} and \eqref{integrability} are in force. Then the weak solution $u$ is locally continuous and the continuity is stable as
$p\to2$.
\end{theorem}
Finally, we collect here two results regarding Sobolev inequalities that will be of use later.
\begin{lemma}\label{Sobolev1}\cite[Chapter I, Proposition 2.1]{Di93} Let $\Omega\subset \mathbb{R}^N$ be a bounded convex set and let $\varphi\in C(\Omega)$ satisfy $0\leq\varphi\leq1$ and the sets $[\varphi\geq k]$ are convex for all $k\in(0,1)$. Let $v\in W^{1,p}(\Omega)$ with
 $1<p\leq 2$ and assume that
$|[v=0]\cap[\varphi=1]|>0$. Then there exists a constant $\gamma$ depending only upon $N$, such that
\begin{equation}\label{Sobolevf1}\left(\int_\Omega\varphi|v|^pdx\right)^{\frac{1}{p}}\leq \gamma\frac{(\diam\Omega)^N}{
|[v=0]\cap[\varphi=1]|^{\frac{N-1}{N}}}\left(\int_\Omega\varphi|Dv|^pdx\right)^{\frac{1}{p}}.\end{equation}
\end{lemma}
\begin{lemma}\label{Sobolev}\cite[Chapter I, Proposition 3.1]{Di93} Let $\Omega\subset \mathbb{R}^N$ be a bounded set and $N\geq 3$. Assume that
$v\in L^\infty(0,T;L^2(\Omega))\cap L^p(0,T;W_0^{1,p}(\Omega))$ and $1<p\leq2$. Then there exists a constant $\gamma$ depending only upon $N$ such that
\begin{equation}\label{Sobolevf}\iint_{\Omega_T}|v|^{p\frac{N+2}{N}}dxdt\leq \gamma\left(\iint_{\Omega_T}|Dv|^pdxdt\right)
\left(\esssup_{0<t<T}\int_\Omega |v|^2dx\right)^{\frac{p}{N}}.\end{equation}
\end{lemma}
\begin{proof}We only need to check that the constant $\gamma$ is independent of $p$. For any fixed $t\in(0,T)$, we apply \cite[Page 64 (2.13)]{LSU} to
obtain
$$\left\|v(\cdot,t)\right\|_{L^{\frac{Np}{N-p}}(\Omega)}\leq \frac{(N-1)p}{N-p}\|Dv(\cdot,t)\|_{L^p(\Omega)}\leq 2(N-1)\|Dv(\cdot,t)\|_{L^p(\Omega)},$$
since $N\geq 3$ and $1<p\leq 2$. Based on the argument of DiBenedetto \cite[page 8]{Di93}, we obtain
\begin{equation*}\begin{split}\iint_{\Omega_T}|v|^{p\frac{N+2}{N}}dxdt&\leq \int_0^T\left\|v(\cdot,t)\right\|_{L^{\frac{Np}{N-p}}(\Omega)}^p
\|v(\cdot,t)\|_{L^2(\Omega)}^{\frac{2p}{N}}dt\\&\leq4(N-1)^2\left(\iint_{\Omega_T}|Dv|^pdxdt\right)
\left(\esssup_{0<t<T}\int_\Omega |v|^2dx\right)^{\frac{p}{N}},
\end{split}\end{equation*}
where the constant on the right-hand side is independent of $p$.
\end{proof}
\section{Local energy estimates}
In this section we state two energy estimates for the cutoff functions $(u-k)_{\pm}$ where $k\in \mathbb{R}$ is a fixed number.
From the assumption \eqref{integrability}, the time derivative exists in the weak sense and square integrable.
This enables us to proceed similarly to the discussions in \cite[page 133]{Di82}. To this end, we introduce the auxiliary function
\begin{equation*}
	v(x,t)=\begin{cases}
	\bar\nu,&\quad \text{on}\quad[u<0],\\
	-w(x,t),&\quad  \text{on}\quad[u=0],
	\end{cases}
\end{equation*}
and the weak form \eqref{weak form two phase stefan} can be rewritten as
\begin{equation}\begin{split}\label{frequently use weak form two phase stefan}-\int_\Omega v(\cdot,t)\chi_{[u\leq 0]}&\varphi(\cdot,t)dx\bigg|_{t=t_0}^{t_1}+\int_{t_0}^{t_1}\int_\Omega v\chi_{[u\leq 0]}\partial_t \varphi dxdt
\\&+\int_{t_0}^{t_1}\int_\Omega
\left[\varphi\partial_t u+A(x,t,u,Du)\cdot D\varphi\right]dxdt=0 \end{split}\end{equation}
where $\varphi\in W_p(\Omega_T)$ and $[t_0,t_1]\subset(0,T]$. The appearance of the first two terms in \eqref{frequently use weak form two phase stefan} is due to the singularity in time derivative and this can cause an extra difficulty in the proof.
From now on, we denote the first two terms in \eqref{frequently use weak form two phase stefan} by the quantity
\begin{equation*}
U(\Omega,t_0,t_1,\varphi)=\int_\Omega v(\cdot,t)\chi_{[u\leq 0]}\varphi(\cdot,t)dx\bigg|_{t=t_0}^{t_1}-\int_{t_0}^{t_1}\int_\Omega v\chi_{[u\leq 0]} \partial_t\varphi dxdt.
\end{equation*}
Let $\bar x\in \Omega$ and $K$ be a cube centered at origin such that $[\bar x+K]\subset\Omega$. We denote by
$\zeta$ a piecewise smooth function in $[\bar x+K]\times [t_0,t_1]$ such that
\begin{equation}\label{def zeta}0\leq\zeta\leq1, \quad|D\zeta|<\infty\quad\text{and}\quad\zeta=0\quad\text{
on}\quad\partial [\bar x+K].\end{equation}
By substituting $\varphi=\pm(u-k)_{\pm}\zeta^p$ into \eqref{frequently use weak form two phase stefan}, we obtain the following
proposition, whose proof we omit (see \cite[Chapter II, Proposition 3.1]{Di93}).
\begin{proposition}
Let $u$ be a weak solution of \eqref{two phase stefan}-\eqref{A} in $\Omega_T$. There exists a constant $\gamma_1$ that can be determined
a priori only in terms of the data such that
\begin{equation}\begin{split}\label{Caccioppoli}
\esssup_{t_0<t<t_1}&\int_{[\bar x+K]\times\{t\}}(u-k)_{\pm}^2\zeta^p dx
+\int_{t_0}^{t_1}\int_{[\bar x+K]} |D(u-k)_{\pm}\zeta|^p dxdt\\
\leq &\int_{[\bar x+K]\times\{t_0\}}(u-k)_{\pm}^2\zeta^p dx+\gamma_1\int_{t_0}^{t_1}\int_{[\bar x+K]} (u-k)_{\pm}^p|D\zeta|^pdxdt\\&+\gamma_1
\int_{t_0}^{t_1}\int_{[\bar x+K]} (u-k)_{\pm}^2|\partial_t\zeta|dxdt+
U([\bar x+K],t_0,t_1,\pm(u-k)_{\pm}\zeta^p).\end{split}\end{equation}
\end{proposition}
We now turn to consider the local logarithmic estimates for weak solutions. To this end,
we introduce the logarithmic function
\begin{equation*}\psi^{\pm}(u)=\max\left\{\ln\frac{H_k^{\pm}}{H_k^{\pm}-(u-k)_{\pm}+c};0\right\},\quad 0<c<H_k^{\pm}\end{equation*}
where $H_k^{\pm}$ is a constant chosen such that
\begin{equation*}H_k^{\pm}\geq \esssup_{[\bar x+K]\times [t_0,t_1]}(u-k)_{\pm}.\end{equation*}
For simplicity of notation, we write $\psi^{\pm}$ instead of $\psi^{\pm}(u)$. We let $(\psi^{\pm})'$ stand for the derivative of $\psi^{\pm}(u)$
with respect to $u$. If we plug $\varphi=2\psi^{\pm}\left(\psi^{\pm}\right)^\prime\zeta^p$ into \eqref{frequently use weak form two phase stefan}, we obtain the following proposition.
\begin{proposition}
Let $u$ be a weak solution of \eqref{two phase stefan}-\eqref{A} in $\Omega_T$. There exists a constant $\gamma_2$ that can be determined
a priori only in terms of the data such that
\begin{equation}\begin{split}\label{logarithmic}\esssup_{t_0<t<t_1}&\int_{[\bar x+K]\times\{t\}}\left(\psi^{\pm}\right)^2\zeta^p dx\leq
\int_{[\bar x+K]\times\{t_0\}}\left(\psi^{\pm}\right)^2\zeta^p dx\\&+
\gamma_2\int_{t_0}^{t_1}\int_{[\bar x+K]}\psi^{\pm}|\left(\psi^{\pm}\right)^\prime|^{2-p}
|D\zeta|^pdxdt+
U([\bar x+K],t_0,t_1,2\psi^{\pm}\left(\psi^{\pm}\right)^\prime\zeta^p),\end{split}\end{equation}
where $\zeta=\zeta(x)$ is independent of $t\in[t_0,t_1]$ and satisfies \eqref{def zeta}.
\end{proposition}
This is a standard result which can be found in \cite[Chapter II, Proposition 3.2]{Di93} and no
proof will be given here.
Finally, we remark that our approach does not follow the idea in \cite{HU} and \cite{U2000} which concerns the approximate solutions to the
equations with regularization of maximal monotone graph. Instead, we follow the approach in \cite{Di82}, which is more convenient to deal with the second alternative.
\section{The intrinsic geometry}\label{The intrinsic geometry}
The continuity of $u$ will be a consequence of the following assertion. For any point $(x_0,t_0)\in\Omega_T$ there exists a family of nested and shrinking cylinders $[(x_0,t_0)+Q_n]$ such that the oscillation of $u$ in $[(x_0,t_0)+Q_n]$ tends to zero
as $n\to\infty$. To begin the proof, we introduce a certain intrinsic parabolic cylinder which plays an important role in alternative arguments.

Without loss of generality, we assume that
$(x_0,t_0)=(0,0)$. Let $R<1$ be a fixed number such that $Q(R,R^{\frac{1}{p}})\subset\Omega_T$ and
we set $\overline Q=Q(R,R^{\frac{1}{p}})$ as a reference parabolic cylinder. We write
\begin{equation*}\mu_-=\essinf_{\overline Q}u,\quad \mu_+=\esssup_{\overline Q}u\quad\text{and}\quad \essosc_{\overline Q}u=\mu_+-\mu_-.\end{equation*}
Here and subsequently, $\omega$ stands for a fixed number and satisfies
$\omega>\essosc_{\overline Q}u$. For $(\bar x,\bar t)\in \overline Q$, we introduce the intrinsic parabolic cylinders of the form
$(\bar x,\bar t)+Q(dR^p,d_*R)$
where
\begin{equation}\label{d}d=\left(\frac{\omega}{2^{m_1}}\right)^{(1-p)(2-p)},\quad d_*=\left(\frac{\omega}{2^{m_2}}\right)^{p-2}
\quad\text{and}\quad m_1=\frac{p}{p-1}m_2.\end{equation}
The quantities $m_1$ and $m_2$ depend upon $\omega$ and will be determined in \S \ref{m2}.
Moreover, the relation $m_1=pm_2/(m-1)$ implies
\begin{equation}\label{dd*}\frac{d_*^p}{d}=\frac{\left(\frac{\omega}{2^{m_2}}\right)^{p(p-2)}}
{\left(\frac{\omega}{2^{m_1}}\right)^{(1-p)(2-p)}}=\omega^{p-2}.\end{equation}
Let $L_1>1$ and $L_2>1$ denote the fixed constants which will be determined in \S 5. We first assume that
$Q(c_1R^p,c_2R)\subset \overline Q$ where $c_1=L_1d$ and $c_2=L_2d_*$. In the case $Q(c_1R^p,c_2R)\nsubseteq \overline Q$, we
shall derive a decay estimate for $\omega$ in terms of $R$ and
the discussion is
postponed until \S 6.2.1. Define $K_\omega=K_{(L_2-1)d_*R}$
and $I_\omega=[-(L_1-1)dR^p,0]$. It is easy to check that $K_{\frac{1}{2}c_2R}\subset K_\omega$.

The main ingredient of the proof is to establish an estimate for the essential oscillation of $u$ in a smaller cylinder $Q\left(d\left(\frac{R}{2}\right)^p,d_*\left(\frac{R}{2}\right)\right)$,
\begin{equation}\label{osc}\essosc_{Q\left(d\left(\frac{R}{2}\right)^p,d_*\left(\frac{R}{2}\right)\right)}u\leq \sigma(\omega)\omega\end{equation}
where $\sigma(\omega)<1$.
We now assume that $\mu_+-\mu_-\geq \omega/2$, otherwise \eqref{osc} follows immediately.
Moreover, for technical reasons, we have to consider the
two cases
$\mu_+\geq |\mu_-|$ and $\mu_+< |\mu_-|$ separately.

Motivated by the work of Henriques and Urbano \cite{HU}, we consider the four complementary cases described below.
In the case $\mu_+\geq |\mu_-|$. For a constant $\nu_0\in(0,1)$, that will be determined in \S 5, we have two possible alternatives.

 \textbf{The first alternative}. There exists $\bar t\in I(\omega)$ such that for all $\bar x\in K_\omega$,
\begin{equation}\label{1st}\big|\{(x,t)\in (\bar x,\bar t)+Q(dR^p,d_*R):u<\mu_-+\frac{\omega}{4}\}\big|\leq \nu_0|Q(dR^p,d_*R)|.\end{equation}

 \textbf{The second alternative}. For any $\bar t\in I(\omega)$ there exists $\bar x\in K_\omega$ such that
\begin{equation}\label{2nd}\big|\{(x,t)\in (\bar x,\bar t)+Q(dR^p,d_*R):u>\mu_+-\frac{\omega}{4}\}\big|\leq (1-\nu_0)|Q(dR^p,d_*R)|.\end{equation}
In the case $\mu_+< |\mu_-|$. For a constant $\nu_0\in(0,1)$, that can be determined by $\omega$, we introduce the following two alternatives.

 \textbf{The third alternative}. There exists $\bar t\in I(\omega)$ such that for all $\bar x\in K_\omega$,
\begin{equation*}\big|\{(x,t)\in (\bar x,\bar t)+Q(dR^p,d_*R):u>\mu_+-\frac{\omega}{4}\}\big|\leq \nu_0|Q(dR^p,d_*R)|.\end{equation*}

 \textbf{The fourth alternative}. For any $\bar t\in I(\omega)$ there exists $\bar x\in K_\omega$ such that
\begin{equation*}\big|\{(x,t)\in (\bar x,\bar t)+Q(dR^p,d_*R):u<\mu_-+\frac{\omega}{4}\}\big|\leq (1-\nu_0)|Q(dR^p,d_*R)|.\end{equation*}
For simplicity, we concentrate in the next two sections only the case $\mu_+\geq|\mu_-|$, since
the treatment of the case $\mu_+< |\mu_-|$ is similar. The proof of \eqref{osc} in the case $\mu_+< |\mu_-|$ is left to the reader.
\section{The first alternative}
The aim in this section is to establish the estimate \eqref{osc} for the first alternative. To start with,
we establish the following DeGiorgi type lemma and determine the constant $\nu_0$ in terms of data and $\omega$.
\begin{lemma}\label{DeGiorgi}There exists a constant $\nu_0\in (0,1)$, depending only upon data and $\omega$, such that if \eqref{1st} holds
for some $(\bar x, \bar t)\in K_\omega\times I_\omega$ then
\begin{equation}\label{DeGiorgi1}u(x,t)\geq\mu_-+\frac{\omega}{8}\quad\quad\text{a.e.}\quad(x,t)\in (\bar x,\bar t)+Q\left(d\left(\frac{R}{2}\right)^p,d_*\frac{R}{2}\right).\end{equation}
\end{lemma}
\begin{proof}Without loss of generality we may assume $(\bar x,\bar t)=(0,0)$. Define two decreasing sequences of numbers
$$R_n=\frac{R}{2}+\frac{R}{2^{n+1}},\quad k_n=\mu_-+\frac{\omega}{8}+\frac{\omega}{2^{n+3}},\quad n=0,1,2,\cdots$$
and construct the family of nested and shrinking cylinders $Q_n=Q(dR_n^p,d_*R_n)$. We construct smooth cutoff functions $0\leq \zeta_n\leq 1$, such that
$\zeta_n\equiv 1$ in $Q_{n+1}$, $\zeta_n\equiv 0$ on $\partial_PQ_{n}$, $|D\zeta_n|\leq 2^{n+2}/(d_*R)$ and $0<\partial_t\zeta_n\leq 2^{2+np}/(dR^p)$. Write the energy estimate \eqref{Caccioppoli} over the cylinder $Q_n$ for the functions $(u-k_n)_-$.
We first observe from \cite[page 145]{Di82} that
$$U(K_{d_*R_n},-dR_n^p,0,-(u-k_n)_{-}\zeta_n^p)
\leq \bar\nu\iint_{Q_n}(u-k_n)_-\partial_t \zeta_n^pdxdt
\leq 2\bar\nu\iint_{Q_n}(u-k_n)_-\partial_t \zeta_ndxdt.$$
Since $(u-k_n)_-\leq \frac{\omega}{4}<1$ and $\partial_t\zeta_n\geq 0$, we obtain
\begin{equation*}\begin{split}\esssup_{-dR_n^p<t<0}&\int_{K_{d_*R_n}\times\{t\}}(u-k_n)_{-}^2\zeta_n^p dx
+\iint_{Q_n}|D(u-k_n)_{-}\zeta_n|^p dxdt\\
\leq &\gamma_1\iint_{Q_n}(u-k_n)_{-}^p|D\zeta_n|^pdxdt+(1+2\bar\nu)\gamma_1
\iint_{Q_n}(u-k_n)_-\partial_t\zeta_ndxdt
\\
\leq &\bar \gamma\left(\left(\frac{\omega}{4}\right)^p\frac{2^{2n}}{(d_*R)^p}+\left(\frac{\omega}{4}\right)\frac{2^{2n}}{dR^p}\right)
\iint_{Q_n}\chi_{[(u-k_n)_->0]}dxdt,\end{split}\end{equation*}
where
$\bar \gamma=(1+2\bar\nu)\gamma_1$.
Moreover, the relation \eqref{dd*} implies
\begin{equation*}\left(\frac{\omega}{4}\right)^p\frac{2^{2n}}{(d_*R)^p}=
\frac{1}{4^p}\frac{\omega ^22^{2n}}{d R^p}\leq \left(\frac{\omega}{4}\right)\frac{2^{2n}}{dR^p},\end{equation*}
since $\omega<1$.
Then we conclude that
\begin{equation*}\begin{split}\esssup_{-dR_n^p<t<0}\int_{K_{d_*R_n}\times\{t\}}(u-k_n)_{-}^2\zeta_n^p dx
+\iint_{Q_n}|D(u-k_n)_{-}\zeta_n|^p dxdt
\leq &\bar \gamma\left(\frac{\omega}{4}\right)\frac{2^{2n}}{dR^p}
\iint_{Q_n}\chi_{[(u-k_n)_->0]}dxdt.\end{split}\end{equation*}
Applying parabolic Sobolev's inequality \eqref{Sobolevf}, we deduce
\begin{equation}\begin{split}\label{S1}\iint_{Q_n}&|(u-k_n)_-\zeta|^{p\frac{N+2}{N}}dxdt\\&\leq \gamma
\left(\esssup_{-dR_n^p<t<0}\int_{K_{d_*R_n}\times\{t\}}(u-k_n)_{-}^2\zeta_n^{2} dx\right)^{\frac{p}{N}}
\iint_{Q_n}|D(u-k_n)_{-}\zeta_n|^p dxdt\\
&\leq \gamma\bar \gamma\left(\frac{\omega}{4}\right)^{1+\frac{p}{N}}\left(\frac{2^{2n}}{dR^p}
\iint_{Q_n}\chi_{[(u-k_n)_->0]}dxdt\right)^{1+\frac{p}{N}}.\end{split}\end{equation}
At this point, we set
$$A_n=\iint_{Q_n}\chi_{[(u-k_n)_->0]}dxdt\quad\text{and}\quad Y_n=\frac{A_n}{|Q_n|}.$$
The left-hand side of \eqref{S1} is estimated below by
\begin{equation}\begin{split}\label{S2}\iint_{Q_n}&|(u-k_n)_-\zeta|^{p\frac{N+2}{N}}dxdt
\geq \iint_{Q_{n+1}}|(u-k_n)_-|^{p\frac{N+2}{N}}\chi_{[(u-k_{n+1})_->0]}dxdt\\&\geq (k_n-k_{n+1})^{p\frac{N+2}{N}}A_{n+1}
\geq \left(\frac{\omega}{2^{n+4}}\right)^{p\frac{N+2}{N}}A_{n+1} .\end{split}\end{equation}
Combining \eqref{S2} with \eqref{S1} and noting that $|Q_n|=dd_*^NR_n^{N+p}$, we obtain
\begin{equation*}Y_{n+1}\leq \gamma\bar\gamma 2^{10+N+\frac{16}{N}}\frac{\omega^{1-p-\frac{p}{N}}2^{(4+\frac{8}{N})n}}{dd_*^{-p}}Y_n^{1+\frac{p}{N}}
=\gamma\bar\gamma 2^{10+N+\frac{16}{N}}\omega^{-1-\frac{p}{N}}2^{(4+\frac{8}{N})n}Y_n^{1+\frac{p}{N}},\end{equation*}
where we used \eqref{dd*} in the last step. Next, we set
\begin{equation}\label{nu0} \nu_0=\left(\gamma\bar\gamma 2^{10+N+\frac{16}{N}}\right)^{-N}2^{-N^2(4+\frac{8}{N})}\omega^{N+1}.\end{equation}
We observe that if $Y_0\leq \nu_0$, then
\begin{equation*}Y_0\leq \left(\gamma\bar\gamma 2^{10+N+\frac{16}{N}}\omega^{-1-\frac{p}{N}}\omega^{-1-\frac{p}{N}}\right)^{-\frac{N}{p}}2^{-\frac{N^2}{p^2}(4+\frac{8}{N})}.
\end{equation*}
Using a lemma on fast geometric convergence of sequences (cf. \cite[Chapter I, Lemma 4.1]{Di93}),
we conclude that $Y_n\to0$ as $n\to\infty$. This completes the proof of Lemma \ref{DeGiorgi}.
\end{proof}

From the definition of the first alternative, we can extend the estimate \eqref{DeGiorgi1} to a larger domain by using a covering argument. To be more precise, we obtain the following corollary whose proof we omit.
\begin{corollary}\label{cor5.2}Let $\nu_0$ be a constant chosen
according to \eqref{nu0} and assume that \eqref{1st} holds for some $\bar t\in I_\omega$ and for all $\bar x\in K_{\omega}$. Then there holds
\begin{equation}\label{1st alt DeGiorgi}u(x,t)\geq\mu_-+\frac{\omega}{8}\quad\quad\text{a.e.}\quad(x,t)\in K_{c_2\frac{R}{2}}\times\left(
\bar t-d\left(\frac{R}{2}\right)^p,\bar t\right].\end{equation}
\end{corollary}
 Next, we set $t_*=\bar t-d(R/2)^p$ and establish the following result regarding the expansion of the positivity in time direction.
\begin{lemma}\label{1st expand time}Let $\nu_0$ be a constant chosen
according to \eqref{nu0} and assume that \eqref{1st} holds for some $\bar t\in I_\omega$ and for all $\bar x\in K_{\omega}$. For any fixed $\nu_1\in(0,1)$, the constant $L_2$ can be chosen in dependence on $N$, $C_0$, $C_1$, $\nu_1$ and $L_1$ such that, with $c_2=L_2d_*$, there holds
\begin{equation}\label{1st expand time formula}\big|\{x\in K_{c_2\frac{R}{4}}:u<\mu_-+\frac{\omega}{2^6}\}\big|<\nu_1|K_{c_2\frac{R}{4}}|\end{equation}
for all $t\in (t_*,0)$.
\end{lemma}
\begin{proof}
We first recall that
\begin{equation}\begin{split}\label{SS0}\psi^-&=\ln^+\left(\frac{H_k^-}{H_k^--(u-k)_-+c}\right)
=\begin{cases}
	\ln\left(\dfrac{H_k^-}{H_k^-+u-k+c}\right),&\quad k-H_k^-\leq u<k-c,\\
	0,&\quad u\geq k-c,
	\end{cases}\end{split}\end{equation}
and
\begin{equation}\begin{split}\label{SS00}\left(\psi^-\right)^\prime
=\begin{cases}
-\dfrac{1}{H_k^-+u-k+c},&\quad k-H_k^-\leq u<k-c,\\
	0,&\quad u\geq k-c,
\end{cases}\end{split}\end{equation}
where
\begin{equation*} H_k^-\geq \esssup_{K_{c_2\frac{R}{2}}\times[t_*,0]}(u-k)_-\qquad\text{and}\qquad 0<c<H_k^-.\end{equation*}
Next we let $k=\mu_-+\frac{\omega}{8}$ and choose $H_k^-=\frac{\omega}{8}$ which is admissible since $(u-k)_-\leq \frac{\omega}{8}$.
We observe from \eqref{1st alt DeGiorgi} that $\psi^-(x,t_*)= 0$ for all $x\in K_{c_2\frac{R}{2}}$.
The following proof will be divided into two steps.

Step 1: \emph{Let $\zeta\equiv\zeta(x)$ be a piecewise smooth function. We establish the estimate}: \begin{equation}\label{claim}U(K_{c_2\frac{R}{2}},t_*,0,2\psi^-\left(\psi^-\right)^\prime\zeta^p)\leq 0.\end{equation}
In the case $k\leq0$, we have $u<k-c<0$ and this yields
\begin{equation*}\begin{split}
U&(K_{c_2\frac{R}{2}},t_*,0,2\psi^-\left(\psi^-\right)^\prime\zeta^p)\\&=2\bar\nu\int_{K_{c_2\frac{R}{2}}} \psi^-\left(\psi^-\right)^\prime\zeta^p dx\bigg|_{t=t_*}^{0}-2\bar\nu\int_{t_*}^{0}\int_{K_{c_2\frac{R}{2}}} \frac{\partial \left[\psi^-\left(\psi^-\right)^\prime\zeta^p\right]}{\partial t} dxdt
=0,\end{split}
\end{equation*}
which is our claim.
We now turn to the case $k>0$. If $\mu_-\geq c$, then there holds $u\geq \mu_-\geq c>0$ and we obtain again the identity $U(K_{c_2\frac{R}{2}},t_*,0,2\psi^-\left(\psi^-\right)^\prime\zeta^p)=0$. It now remains to consider the case $\mu_-<c$.
Since $\zeta\equiv \zeta(x)$, then $\partial_t\zeta\equiv 0$. We write
\begin{equation*}\begin{split}
U&(K_{c_2\frac{R}{2}},t_*,0,2\psi^-\left(\psi^-\right)^\prime\zeta^p)\\&=2\int_{K_{c_2\frac{R}{2}}} v(\cdot,t)\chi_{[u\leq 0]}\psi^-\left(\psi^-\right)^\prime\zeta^p dx\bigg|_{t=t_*}^{0}-2\int_{t_*}^{0}\int_{K_{c_2\frac{R}{2}}} v\chi_{[u\leq 0]}\frac{\partial \left[\psi^-\left(\psi^-\right)^\prime\right]}{\partial t}\zeta^p dxdt
\\&=:S_1+S_2,\end{split}
\end{equation*}
with the obvious meaning of $S_1$ and $S_2$. We begin with the estimate for $S_1$.
Noting that $-\left(\psi^-\right)^\prime=(u-\mu_-+c)^{-1}\leq c^{-1}$, we obtain
\begin{equation}\begin{split}\label{SS1}S_1&=2\int_{K_{c_2\frac{R}{2}}} v(\cdot,t)\chi_{[u\leq 0]}\psi^-\left(\psi^-\right)^\prime\zeta^p dx\bigg|_{t=0}+2\int_{K_{c_2\frac{R}{2}}} v(\cdot,t)\chi_{[u\leq 0]}\psi^-(-\left(\psi^-\right)^\prime)\zeta^p dx\bigg|_{t=t_*}
\\&\leq 2\bar\nu\int_{K_{c_2\frac{R}{2}}\cap [u<0]} \psi^-\left(\psi^-\right)'\zeta^p dx\bigg|_{t=0}+
\frac{2\bar\nu}{c}\int_{K_{c_2\frac{R}{2}}} \psi^-\zeta^p dx\bigg|_{t=t_*}.\end{split}
\end{equation}
To estimate $S_2$, we note that $\partial u/\partial t \equiv 0$ on $[u=0]$ and there holds
\begin{equation*}\begin{split}S_2&=2\bar\nu\int_{t_*}^{0}\int_{K_{c_2\frac{R}{2}}} \left[\psi^-\left(\psi^-\right)'\right]'\frac{\partial u_-}{\partial t}\zeta^p dxdt\\&=-2\bar\nu\int_{t_*}^{0}\int_{K_{c_2\frac{R}{2}}} \chi_{[u< 0]}\frac{\partial \left[\psi^-\left(\psi^-\right)^\prime\right]}{\partial t}\zeta^p dxdt\end{split}
\end{equation*}
At this stage, we introduce an auxiliary function $\Phi_c(u)$, defined by
\begin{equation*}\begin{split}\Phi_c(u)&=\begin{cases}
	-\psi^-(u)\left(\psi^-\right)'(u)+\psi^-(0)\left(\psi^-\right)'(0),&\quad \mu_-\leq u<0,\\
	0,&\quad 0\leq u<\mu_-+\frac{\omega}{8}-c,
	\end{cases}
\\&=\begin{cases}
	\ln\left(\dfrac{\omega}{8(u-\mu_-+c)}\right)\dfrac{1}{u-\mu_-+c}-\ln\left(\dfrac{\omega}{8(c-\mu_-)}\right)\dfrac{1}{c-\mu_-},&\quad \mu_-\leq u<0,\\
	0,&\quad 0\leq u<\mu_-+\dfrac{\omega}{8}-c.
	\end{cases}
\end{split}\end{equation*}
We observe that $0\leq \Phi_c(u)\leq -\psi^-\left(\psi^-\right)'$, $\Phi_c(u)$ is Lipschitz with respect to $u$
and
$$\frac{\partial \Phi_c(u)}{\partial t}=-\frac{\partial \left[\psi^-\left(\psi^-\right)^\prime\right]}{\partial t}\chi_{[u<0]}.$$
Integrating by parts, we deduce
\begin{equation}\begin{split}\label{SS2}S_2&=2\bar\nu\int_{t_*}^{0}\int_{K_{c_2\frac{R}{2}}} \frac{\partial \Phi_c}{\partial t}\zeta^p dxdt\\
&=2\bar\nu\int_{K_{c_2\frac{R}{2}}} \Phi_c\zeta^p dx\bigg|_{t=0}-
2\bar\nu\int_{K_{c_2\frac{R}{2}}} \Phi_c\zeta^p dx\bigg|_{t=t_*}\\&\leq 2\bar\nu\int_{K_{c_2\frac{R}{2}}} \Phi_c\zeta^p dx\bigg|_{t=0}\\&\leq -2\bar\nu\int_{K_{c_2\frac{R}{2}}\cap [u<0]} \psi^-\left(\psi^-\right)'\zeta^p dx\bigg|_{t=0}.\end{split}
\end{equation}
Combining \eqref{SS1} and \eqref{SS2}, we conclude that
\begin{equation*}\begin{split}U(K_{c_2\frac{R}{2}},t_*,0,2\psi^-\left(\psi^-\right)^\prime\zeta^p)=S_1+S_2\leq
\frac{2\bar\nu}{c}\int_{K_{c_2\frac{R}{2}}} \psi^-\zeta^p dx\bigg|_{t=t_*}=0,\end{split}
\end{equation*}
since $\psi^-(x,t_*)= 0$ for all $x\in K_{c_2\frac{R}{2}}$.
This proves the claim \eqref{claim}.

Step 2: \emph{Proof of \eqref{1st expand time}}.
Since $\psi^-(x,t_*)= 0$ for all $x\in K_{c_2\frac{R}{2}}$, we find that
$$\int_{K_{c_2\frac{R}{2}}\times\{-t_*\}}\left(\psi^-\right)^2\zeta^p dx= 0.$$
Plugging this into \eqref{logarithmic} and taking into account \eqref{claim}, we obtain
\begin{equation*}\begin{split}\esssup_{t_*<t<0}\int_{K_{c_2\frac{R}{2}}\times\{t\}}\left(\psi^-\right)^2\zeta^p dx\leq
\gamma_2\iint_{Q(|t_*|,c_2\frac{R}{2})}\psi^-|\left(\psi^-\right)^\prime|^{2-p}
|D\zeta|^pdxdt.\end{split}\end{equation*}
We take cutoff function $\zeta=\zeta(x)$ which satisfies $0\leq\zeta\leq1$ in $K_{c_2\frac{R}{2}}$, $\zeta\equiv 1$
in $K_{c_2\frac{R}{4}}$ and $|D\zeta|\leq 4(c_2R)^{-1}$. Now we take $c=2^{-6}\omega$ and deduce from \eqref{SS0} and \eqref{SS00} the estimate
 \begin{equation*}\begin{split}\psi^-|\left(\psi^-\right)^\prime|^{2-p}
|D\zeta|^p&\leq \left(\ln\frac{\omega}{8c}\right)c^{p-2}\left(\frac{4}{c_2R}\right)^p=2^{11}(\ln 2)\omega^{p-2}
L_2^{-p}d_*^{-p}R^{-p}.
\end{split}\end{equation*}
Keeping in mind $|t_*|\leq L_1dR^p$, we deduce from \eqref{dd*} the estimate
\begin{equation*}\begin{split}\esssup_{-t_*<t<0}\int_{K_{c_2\frac{R}{2}}\times\{t\}}\left(\psi^-\right)^2\zeta^p dx&\leq
C|t_*|\omega^{p-2}
L_2^{-p}d_*^{-p}R^{-p}|K_{c_2\frac{R}{4}}|
\leq CL_1L_2^{-p}|K_{c_2\frac{R}{4}}|,\end{split}\end{equation*}
where $C=2^{N+12}\gamma_2>1$. At this point, we choose
\begin{equation*}L_2\geq (2^{N+12}\gamma_2)^{\frac{1}{p}}\nu_1^{-\frac{1}{p}}L_1^{\frac{1}{p}}.\end{equation*}
For such a choice of $L_2$, the above estimate yields
\begin{equation}\begin{split}\label{SS3}\esssup_{t_*<t<0}\int_{K_{c_2\frac{R}{2}}\times\{t\}}\left(\psi^-\right)^2\zeta^p dx\leq \nu_1|K_{c_2\frac{R}{4}}|.\end{split}\end{equation}
The left-hand side of \eqref{SS3} is estimated below by integrating
over the smaller set
$$\{x\in K_{c_2\frac{R}{4}}:u<\mu_-+\frac{\omega}{2^6}\}\subset K_{c_2\frac{R}{2}}.$$
On such a set, $\zeta\equiv 1$ and
\begin{equation*}\begin{split}\psi^-&=\ln^+\left(\frac{\frac{\omega}{8}}{\frac{\omega}{8}-(u-k)_-+\frac{\omega}{2^6}}\right)
\geq 2\ln2>1.\end{split}\end{equation*}
This gives
\begin{equation}\begin{split}\label{SS4}\int_{K_{c_2\frac{R}{2}}\times\{t\}}\left(\psi^-\right)^2\zeta^p dx\geq \big|\{x\in K_{c_2\frac{R}{4}}:u<\mu_-+\frac{\omega}{2^6} \}\big|,\end{split}\end{equation}
for all $t\in(t_*,0)$.
Combining \eqref{SS4} with \eqref{SS3}, we conclude that for any $t\in(t_*,0)$ there holds
\begin{equation*}\begin{split}\big|\{x\in K_{c_2\frac{R}{4}}:u<\mu_-+\frac{\omega}{2^6} \}\big|\leq \nu_1
|K_{c_2\frac{R}{4}}|,\end{split}\end{equation*}
which proves the lemma.
\end{proof}
With the help of the preceding two lemmas we can now prove the following proposition which is the main result in this section.
\begin{proposition}\label{1st main result}Suppose that $Q(c_1R^p,c_2R)\subset \overline Q\subset\Omega_T$.
Let $\nu_0$ be a constant chosen
according to \eqref{nu0} and assume that \eqref{1st} holds for some $\bar t\in I_\omega$ and for all $\bar x\in K_{\omega}$. Then
there exists a constant $\nu_1\in(0,1)$ depending only upon the data such that the following holds. For the constant $L_2>0$ which is determined
a priori only in terms of $\{N,C_0,C_1,L_1\}$ such that
\begin{equation}\label{L2leq}L_2\geq (2^{N+12}\gamma_2)^{\frac{1}{p}}\nu_1^{-\frac{1}{p}}L_1^{\frac{1}{p}},\end{equation}
and
$c_2=L_2d_*$, there holds
\begin{equation}\label{DeGiorgi2}u(x,t)>\mu_-+\frac{\omega}{2^{7}}\qquad\text{a.e.}\quad\text{in}\quad Q\left(-t_*,c_2\frac{R}{8}\right).\end{equation}
\end{proposition}
\begin{proof}
Let $\gamma$ and $\gamma_1$ be the constants determined by Lemma \ref{Sobolev} and \ref{Caccioppoli}, respectively, depending only upon the data
$\{N,C_0,C_1\}$.
We now choose
\begin{equation}\label{nu1}\nu_1=2^{-3N^2(1+\frac{2}{N})}2^{-22(N+2)-N^2}
\gamma^{-N}\gamma_1^{-N-2}\end{equation}
and choose $L_2$ satisfying \eqref{L2leq}. By Lemma \ref{1st expand time}, we conclude from \eqref{1st expand time formula} that
\begin{equation}\label{SS6}\int_{t_*}^0\big|\{x\in K_{c_2\frac{R}{4}}:u<\mu_-+\frac{\omega}{2^6}\}\big|
dt<\nu_1|K_{c_2\frac{R}{4}}\times(t_*,0)|.\end{equation}
Next, we define two decreasing sequences of numbers
\begin{equation*}R_n=\frac{R}{8}+\frac{R}{2^{n+3}},\qquad k_n=\mu_-+\frac{\omega}{2^7}+\frac{\omega}{2^{n+7}},\qquad n=0,1,2,\cdots.\end{equation*}
We set $Q_n=Q(-t_*,c_2R_n)$
and choose piecewise smooth cutoff functions $\zeta_n(x)$ defined in $K_{c_2R_n}$ and satisfying $0\leq\zeta_n\leq 1$ in $K_{c_2R_n}$, $\zeta_n\equiv 1$ in $K_{c_2R_{n+1}}$
and $|D\zeta_n|\leq 2^{n+4}(c_2R)^{-1}$. We recall that from \eqref{DeGiorgi1} there holds $(u-k_n)_-(\cdot,t_*)= 0$ in $K_{c_2R_n}$.
Write the energy estimate \eqref{Caccioppoli} for $(u-k_n)_-$ over the cylinder $Q_n$, we obtain
\begin{equation*}\begin{split}\esssup_{t_*<t<0}&\int_{K_{c_2R_n}\times\{t\}}(u-k_n)_-^2\zeta_n^p dx
+\iint_{Q_n}|D(u-k_n)_-\zeta_n|^p dxdt\\
\leq &\gamma_1\iint_{Q_n}(u-k_n)_-^p|D\zeta_n|^pdxdt
+U(K_{c_2R_n},t_*,0,-(u-k_n)_-\zeta_n^p),
\end{split}\end{equation*}
since $\partial_t\zeta_n\equiv 0$. Next, we claim that
\begin{equation}\label{claim2}U(K_{c_2R_n},t_*,0,-(u-k_n)_-\zeta_n^p)\leq 0.\end{equation}
To prove \eqref{claim2}, we first consider the case $k_n< 0$. In this case, either $(u-k_n)_-=0$ or $u\leq k_n<0$. We obtain
\begin{equation*}\begin{split}U&(K_{c_2R_n},t_*,0,-(u-k_n)_-\zeta_n^p)\\&=-\bar\nu\int_{K_{c_2R_n}}(u-k_n)_-\zeta_n^p dx\bigg|_{t=t_*}^{0}+\bar\nu\int_{t_*}^{0}\int_{K_{c_2R_n}} \frac{\partial (u-k_n)_-}{\partial t}\zeta_n^p dxdt
\\&=0.\end{split}\end{equation*}
In the case $k_n\geq 0$, we note that
$(u-k_n)_-\chi_{[u\leq0]}=k_n\chi_{[u\leq0]}+u_-$ and $\partial_t(u-k_n)_-\chi_{[u\leq 0]}=\partial_tu_-$.
Taking into account $(u-k_n)_-(\cdot,t_*)= 0$ in $K_{c_2R_n}$, we get
\begin{equation*}\begin{split}U&(K_{c_2R_n},t_*,0,-(u-k_n)_-\zeta_n^p)\\&=-\int_{K_{c_2R_n}} v(\cdot,t)\chi_{[u\leq 0]}(u-k_n)_-\zeta_n^p dx\bigg|_{t=t_*}^{0}+\int_{t_*}^{0}\int_{K_{c_2R_n}} v\chi_{[u\leq 0]}\frac{\partial (u-k_n)_-}{\partial t}\zeta_n^p dxdt
\\&\leq -\int_{K_{c_2R_n}} v(\cdot,t)\chi_{[u\leq 0]}u_-\zeta_n^p dx\bigg|_{t=0}+\bar\nu\int_{K_{c_2R_n}}u_-\zeta_n^p dx\bigg|_{t=t_*}^{0}\leq 0,\end{split}\end{equation*}
which proves the claim.
At this stage, we arrive at
\begin{equation*}\begin{split}\esssup_{t_*<t<0}&\int_{K_{c_2R_n}\times\{t\}}(u-k_n)_-^2\zeta_n^p dx
+\iint_{Q_n}|D(u-k_n)_-\zeta_n|^p dxdt
\leq \gamma_1\iint_{Q_n}(u-k_n)_-^p|D\zeta_n|^pdxdt.
\end{split}\end{equation*}
Set
$$A_n=\iint_{Q_n}\chi_{[(u-k_n)_->0]}dxdt\quad\text{and}\quad Y_n=\frac{A_n}{|Q_n|}.$$
With this notation, the estimate \eqref{SS6} reads
\begin{equation}\label{Y00}Y_0\leq\nu_1.\end{equation}
Keeping in mind $(u-k_n)_-\leq 2^{-6}\omega$ and $|D\zeta_n|\leq 2^{n+4}(c_2R_n)^{-1}$, we deduce
\begin{equation*}\begin{split}\esssup_{t_*<t<0}&\int_{K_{c_2R_n}\times\{t\}}(u-k_n)_-^2\zeta_n^p dx
+\iint_{Q_n}|D(u-k_n)_-\zeta_n|^p dxdt
\leq \gamma_12^{-2}\omega^p2^{2n}c_2^{-p}R_n^{-p}A_n.
\end{split}\end{equation*}
Applying the parabolic Sobolev's inequality \eqref{Sobolevf}, we obtain
\begin{equation*}\begin{split}\iint_{Q_n}&|(u-k_n)_-\zeta|^{p\frac{N+2}{N}}dxdt\\&\leq \gamma
\left(\esssup_{t_*<t<0}\int_{K_{c_2R_n}\times\{t\}}(u-k_n)_{-}^2\zeta_n^2  dx\right)^{\frac{p}{N}}
\iint_{Q_n}|D(u-k_n)_{-}\zeta_n|^p dxdt\\
&\leq 8^{2(1+\frac{2}{N})}\gamma\gamma_1^{1+\frac{2}{N}}2^{n(2+\frac{4}{N})}\omega^{p(1+\frac{p}{N})}c_2^{-p(1+\frac{p}{N})}R^{-p
(1+\frac{p}{N})}A_n^{1+\frac{p}{N}}.\end{split}\end{equation*}
On the other hand, we estimate below
\begin{equation*}\begin{split}\iint_{Q_n}&|(u-k_n)_-\zeta|^{p\frac{N+2}{N}}dxdt
\geq \iint_{Q_{n+1}}|(u-k_n)_-|^{p\frac{N+2}{N}}\chi_{[(u-k_{n+1})_->0]}dxdt\\&\geq (k_n-k_{n+1})^{p\frac{N+2}{N}}A_{n+1}
\geq \left(\frac{\omega}{2^{n+8}}\right)^{p\frac{N+2}{N}}A_{n+1} .\end{split}\end{equation*}
Combining the estimates above we infer that
\begin{equation*}\begin{split}Y_{n+1}\leq 2^{3n(1+\frac{2}{N})}2^{22(1+\frac{2}{N})}
\gamma\gamma_1^{1+\frac{2}{N}}c_2^{-p(1+\frac{p}{N})}\omega^{(p-2)\frac{p}{N}}R^{-p(1+\frac{p}{N})}\frac{|Q_n|}{|Q_{n+1}|}
|Q_n|^{\frac{p}{N}}Y_n^{1+\frac{p}{N}}.\end{split}\end{equation*}
Taking into account $|Q_n|=L_2^Nd_*^NR_n^{N}t_*$, $8^{-1}R\leq R_n\leq 4^{-1}R$, $t_*\leq L_1dR^p$ and the relation \eqref{dd*}, we obtain
\begin{equation*}\begin{split}Y_{n+1}&\leq 2^{3n(1+\frac{2}{N})}2^{22(1+\frac{2}{N})+N}
\gamma\gamma_1^{1+\frac{2}{N}}\omega^{(p-2)\frac{p}{N}}\left(\frac{L_1}{L_2^{p}}\right)^{\frac{p}{N}}
\left(\frac{d}{d_*^p}\right)^{\frac{p}{N}}Y_n^{1+\frac{p}{N}}\\&\leq
2^{3n(1+\frac{2}{N})}2^{22(1+\frac{2}{N})+N}
\gamma\gamma_1^{1+\frac{2}{N}}Y_n^{1+\frac{p}{N}}
\end{split}\end{equation*}
since
\begin{equation*}\begin{split}|Q_n|\leq
L_1L_2^Nd_*^Nd R^{N+p},\qquad \frac{|Q_n|}{|Q_{n+1}|}\leq 2^N\qquad\text{and}\qquad\frac{L_1}{L_2^{p}}\leq \frac{\nu_1}{2^{N+12}\gamma_2}<1.\end{split}\end{equation*}
Recalling from \eqref{nu1} and \eqref{Y00} that
\begin{equation*}Y_0\leq \nu_1<\left(2^{3(1+\frac{2}{N})}\right)^{-\frac{N^2}{p^2}}\left(2^{22(1+\frac{2}{N})+N}
\gamma\gamma_1^{1+\frac{2}{N}}\right)^{-\frac{N}{p}}.\end{equation*}
By the lemma on fast geometric convergence of sequences (cf. \cite[Chapter I, Lemma 4.1]{Di93}), we conclude that
$Y_n\to0$ as $n\to\infty$, which proves the proposition.
\end{proof}
From now on, we choose $L_1=2$. For such a choice of $L_1$, we determine $L_2$ by
\begin{equation}\label{conditioin for L2 final}L_2= 8(2^{N+12}\gamma_2)^{\frac{1}{p}}\nu_1^{-\frac{1}{p}},\end{equation}
where $\nu_1$ is the constant chosen according to \eqref{nu1}.
By these choices, taking \eqref{DeGiorgi2} into account, we obtain the desired estimate \eqref{osc} under the assumption of first alternative.
\section{The second alternative}
In this section we will establish the estimate \eqref{osc} for the second alternative. We start with the following lemma whose proof we omit.
\begin{lemma}Let $\nu_0$ be a constant chosen
according to \eqref{nu0} and assume that \eqref{2nd} holds for some fixed $\bar t\in I_\omega$ and $\bar x\in K_\omega$. There exists a time level $t_0\in [\bar t-dR^p,\bar t-\frac{1}{2}\nu_0 dR^p]$ such that
\begin{equation}\label{t0}\big|\{x\in \bar x+K_{d_*R}:u>\mu_+-\frac{\omega}{4}\}\big|\leq\left(\frac{1-\nu_0}{1-\frac{1}{2}\nu_0}\right)|K_{d_*R}|.\end{equation}
\end{lemma}
The estimate \eqref{t0} is the starting point for the analysis of the second alternative and we can now ignore the inequality \eqref{2nd}.
Next, since $\mu_+\geq |\mu_-|$ and $\mu_+-\mu_->\frac{\omega}{2}$, then $\mu_+\geq0$ and $4\mu_+>\omega$. For $k\geq\mu_+-\frac{\omega}{4}$, we observe that if $(u-k)_+>0$ then $u>k>0$. So we can always discard the term involving $U$ in the energy estimates.
Let $m_2\geq 2$ be the quantity which will be determined in \S \ref{m2}. We give the following lemma regarding the expansion of positivity
in time direction.
\begin{lemma}\label{2nd expand time}There exists $\tilde n>1$ depending only upon the data and $\nu_0$, and independent of $m_2$ such that
\begin{equation}\label{tilde t}\big|\{x\in \bar x+K_{d_*R}:u>\mu_+-\frac{\omega}{2^{m_2+\tilde n}}\}\big|\leq\left(1-\left(\frac{\nu_0}{2}\right)^2\right)|K_{d_*R}|\end{equation}
for all $t\in(t_0,\tilde t)$ where
\begin{equation}\label{level tilde t}
\tilde t=t_0+\frac{\nu_0}{2}2^{-\tilde np(2-p)}\left(\frac{\omega}{2^{m_2}}\right)^{(1-p)(2-p)}R^p.\end{equation}
\end{lemma}
\begin{proof}Without loss of generality, we assume that $\bar x=0$. From \eqref{level tilde t}, we observe that
$$K_{d_*R}\times (t_0, \tilde t]\subset (0,\bar t)+Q(dR^p,d_*R).$$
This enable us to use the logarithmic estimate \eqref{logarithmic} over the cylinder $K_{d_*R}\times (t_0, \tilde t]$.
Set $k=\mu_+-\frac{\omega}{2^{m_2}}$ and $c=\frac{\omega}{2^{m_2+\tilde n}}$
where $\tilde n>1$ is to be determined.
We consider the logarithmic function
\begin{equation*}\begin{split}\psi^+=\ln^+\left(\frac{\frac{\omega}{2^{m_2}}}{\frac{\omega}{2^{m_2}}-(u-k)_++c}\right).\end{split}\end{equation*}
Then we have $\psi^+\leq \tilde n\ln2$ and
\begin{equation*}\begin{split}|(\psi^+)'|^{2-p}\leq \left(\frac{\omega}{2^{m_2-1}}\right)^{(p-1)(2-p)}
\left(\frac{2^{m_2+\tilde n}}{\omega}\right)^{p(2-p)}=2^{(p-1)(2-p)}
\left(\frac{\omega}{2^{m_2}}\right)^{(p-1)(2-p)}d_*^p
2^{\tilde np(2-p)}.\end{split}\end{equation*}
Choose a piecewise smooth cutoff function $\zeta(x)$, defined in $K_{d_*R}$, and satisfying $0\leq \zeta\leq1$ in $K_{d_*R}$,
$\zeta\equiv 1$ in $K_{(1-\sigma)d_*R}$ and $|D\zeta|\leq (\sigma d_*R)^{-1}$, where $\sigma\in (0,1)$ is to be determined.
Since $k>\mu_+-\frac{\omega}{4}$, then $U(K_{d_*R},t_0,\tilde t,2\psi^{+}\left(\psi^{+}\right)^\prime\zeta^p)=0$.
Then we deduce from \eqref{logarithmic} the logarithmic estimate
\begin{equation*}\begin{split}\esssup_{t_0<t<\tilde t}&\int_{K_{d_*R}\times\{t\}}\left(\psi^{+}\right)^2\zeta^p dx\leq
\int_{K_{d_*R}\times\{t_0\}}\left(\psi^{+}\right)^2\zeta^p dx+\gamma_2\int_{t_0}^{\tilde t}\int_{
K_{d_*R}}\psi^{+}|\left(\psi^{+}\right)^\prime|^{2-p}
|D\zeta|^pdxdt.\end{split}\end{equation*}
Keeping in mind the definition of $\tilde t$ and \eqref{t0}, we obtain
\begin{equation*}\begin{split}\esssup_{t_0<t<\tilde t}&\int_{K_{d_*R}\times\{t\}}\left(\psi^{+}\right)^2\zeta^p dx\\&\leq
\tilde n^2(\ln2)^2\big|\{x\in K_{d_*R}:u>k\}\big|+\gamma_2\frac{2\tilde n\ln 2}{\sigma^pd_*^pR^p}
\left(\frac{\omega}{2^{m_2}}\right)^{(p-1)(2-p)}d_*^p
2^{\tilde np(2-p)}(\tilde t-t_0)|K_{d_*R}|
\\&\leq \left(\tilde n^2(\ln2)^2\left(\frac{1-\nu_0}{1-\frac{1}{2}\nu_0}\right)+\nu_0(\ln2)\gamma_2\frac{\tilde n}{\sigma^p}\right)
|K_{d_*R}|.\end{split}\end{equation*}
At this stage, we proceed similarly as in \cite[Page 938-939]{HU}. To estimate below the integral on the left-hand side, we consider a smaller set defined by
$$S=\{x\in K_{(1-\sigma)d_*R}:u>\mu_+-\frac{\omega}{2^{m_2+\tilde n}}\}\subset K_{(1-\sigma)d_*R}.$$
On such a set, $\zeta\equiv 1$ and
\begin{equation*}\begin{split}\psi^+&=\ln^+\left(\frac{\frac{\omega}{2^{m_2}}}{\frac{\omega}{2^{m_2}}-(u-k)_++\frac{\omega}{2^{m_2+\tilde n}}}\right)
\geq (\tilde n-1)\ln2.\end{split}\end{equation*}
It follows that
\begin{equation*}\begin{split}\big|\{x\in \bar x+K_{d_*R}:u>\mu_+-\frac{\omega}{2^{m_2+\tilde n}}\}\big|\leq
\left(\left(\frac{\tilde n}{\tilde n-1}\right)^2\left(\frac{1-\nu_0}{1-\frac{1}{2}\nu_0}\right)+\frac{8\gamma_2}{\tilde n\sigma^p}+N\sigma\right)|K_{d_*R}|\end{split}\end{equation*}
for all $t\in(t_0,\tilde t)$.
To prove the lemma choose
\begin{equation}\label{tilde n}\sigma=\frac{3}{8N}\nu_0^2\qquad\text{and}\qquad\tilde n=\frac{8\gamma_2}{\nu_0^{6}}.\end{equation}
With this choice we see that $\tilde n$ is independent of $m_2$.
\end{proof}
\subsection{Expansion of positivity in space variable}
In this subsection we will establish an expansion of positivity result for the weak solutions in a larger cylinder.
We have to work with an equation in dimensionless form and our proof is in the spirit of \cite{AZ, CD, Di93, HU}.
To
start with, we introduce the change of variables
\begin{equation*}\begin{split}x'=\frac{x-\bar x}{2c_2R}\qquad\text{and}\qquad t'=4^p\left(\frac{t-\tilde t}{\tilde t-t_0}\right).\end{split}\end{equation*}
This transformation maps $[\bar x+K_{d_*R}]\times (t_0,\tilde t)\to K_{\frac{1}{2L_2}}\times (-4^p,0)$ and $[\bar x+K_{8c_2R}]
\times (t_0,\tilde t)\to K_{4}\times (-4^p,0)$.
Next, we set the new functions
\begin{equation*}\begin{split}\tilde u(x',t')=u(x,t)\qquad\text{and}\qquad\bar u(x',t')=\left(\tilde u(x',t')-\mu_+\right)
\left(\frac{2^{m_2+\tilde n-1}}{\omega}\right). \end{split}\end{equation*}
With these notations, the estimate \eqref{tilde t} can be rewritten as
\begin{equation}\label{positive small cube}\big|\{x'\in K_{\frac{1}{2L_2}}:\bar u(x',t')<-\frac{1}{2}\}\big|\geq \left(\frac{\nu_0}{2}\right)^2 |K_{\frac{1}{2L_2}}|\end{equation}
for all $t'\in(-4^p,0]$. On the other hand, we
set $\tilde w(x',t')=w(x,t)$. The function $v(x,t)$ in \eqref{frequently use weak form two phase stefan} can be written in the new variable as
\begin{equation}\label{tilde v}
	\tilde v(x',t')=\begin{cases}
	\bar\nu,&\qquad \text{on}\quad\big\{\bar u< -\mu_+\frac{2^{m_2+\tilde n-1}}{\omega}\big\},\\
	-\tilde w(x',t'),&\qquad \text{on}\quad\big\{\bar u=-\mu_+\frac{2^{m_2+\tilde n-1}}{\omega}\big\}.
	\end{cases}
\end{equation}
For any $\varphi\in W_p(Q_4)$ and $[t_1,t_2]\subset (-4^p,0]$, we rewrite the weak form \eqref{frequently use weak form two phase stefan} in terms of the new variables and new functions
as follows
\begin{equation}\begin{split}\label{dimensonal less weak form two phase stefan}-\frac{2^{m_2+\tilde n-1}}{\omega}&\int_{K_4} \tilde v(\cdot,t')\chi_{\left\{\bar u\leq -\mu_+\frac{2^{m_2+\tilde n-1}}{\omega}\right\}}\varphi(\cdot,t')dx'\bigg|_{t'=t_1}^{t_2}
\\&+
\frac{2^{m_2+\tilde n-1}}{\omega}\int_{t_1}^{t_2}\int_{K_4} \tilde v\chi_{\left\{\bar u\leq -\mu_+\frac{2^{m_2+\tilde n-1}}{\omega}\right\}}\frac{\partial \varphi}{\partial t'}dx'dt'
\\&\quad+\int_{t_1}^{t_2}\int_{K_4}
\left(\varphi\frac{\partial \bar u}{\partial t'}+\bar A(x',t',\bar u,D\bar u)\cdot D_{x'}\varphi\right)dx'dt'=0\end{split}\end{equation}
where
\begin{equation*}\begin{split}\bar A&(x',t',\bar u,D\bar u)=\frac{2^{m_2+\tilde n}(\tilde t-t_0)}{4^{p+1} c_2\omega R}
A\left(\bar x+2c_2Rx',\tilde t+\frac{\tilde t-t_0}{4^p}t',\frac{\omega}{2^{m_2+\tilde n-1}}\bar u+\mu_+,\frac{\omega}{2^{m_2+\tilde n}c_2R}D_{x'}\bar u\right).\end{split}\end{equation*}
From \eqref{A}, we obtain the structure conditions for $\bar A$ as follows
\begin{equation}\label{bar A}
	\begin{cases}
	\bar A(x',t',\bar u,D \bar u)\cdot D \bar u\geq C_0\Lambda(\omega)|D \bar u|^{p},\\
	|\bar A(x',t',\bar u,D \bar u)|\leq C_1\Lambda(\omega)|D \bar u|^{p-1},
	\end{cases}
\end{equation}
where
\begin{equation*}\Lambda(\omega)=\frac{2^{(m_2+\tilde n)(2-p)}\omega^{p-2}(\tilde t-t_0)}{4^{p+1}c_2^pR^p}.
\end{equation*}
According to \eqref{level tilde t}, the above identity reads
\begin{equation}\label{structure}\Lambda(\omega)=4^{-p-\frac{3}{2}}\nu_0L_2^{-p}2^{\tilde n(1-p)(2-p)},
\end{equation}
which is independent of $m_2$.
By abuse of the notation, we write $(x,t)$ instead of the new variables. The next lemma is an analogue of \cite[Chapter II, Lemma 1.1]{Di93}.
\begin{lemma}\label{subsolution lemma}Let $k\in (-1,0)$. Then there holds
\begin{equation}\begin{split}\label{subsolution}\int_{t_1}^{t_2}\int_{K_4}&
\varphi \frac{\partial (\bar u-k)_+}{\partial t}+\bar A(x,t,k+(\bar u-k)_+,D(\bar u-k)_+)\cdot D\varphi dxdt
\leq 0\end{split}\end{equation}
for all $\varphi\in W_p(Q_4)$ and $\varphi\geq0$.
\end{lemma}
\begin{proof}We proceed similarly as in \cite[page 18-19]{Di93}. In \eqref{dimensonal less weak form two phase stefan}
choose the testing function
\begin{equation*}\varphi_\epsilon= \varphi \frac{(\bar u-k)_+}{(\bar u-k)_++\epsilon},\quad\epsilon>0.\end{equation*}
In view of $\omega<4\mu_+$, we find that the first two terms in \eqref{dimensonal less weak form two phase stefan} vanish and there holds
\begin{equation*}\begin{split}\int_{t_1}^{t_2}\int_{K_4}&
\varphi_\epsilon\frac{\partial \bar u}{\partial t}+\frac{(\bar u-k)_+}{(\bar u-k)_++\epsilon}\bar A(x,t,\bar u,D\bar u)\cdot D\varphi dxdt
\\&\leq -C_0\Lambda(\omega)\int_{t_1}^{t_2}\int_{K_4}\frac{\epsilon \varphi}{((\bar u-k)_++\epsilon)^2} |D(\bar u-k)_+|^pdxdt\\&\leq 0.\end{split}\end{equation*}
Consequently, the inequality \eqref{subsolution} follows by passing to the limit $\epsilon\downarrow 0$.
\end{proof}
With the previous result at hand, we can now give the following lemma which concerns the expansion of positivity in space variable.
Since the constants in the proof should be uniformly bounded in $p$ when $p\to2$, we employ the idea from Alkhutov and Zhikov \cite{AZ}. Moreover, this approach enables us to obtain explicit estimates for the parameters which will
play a crucial role in the next subsection.
\begin{lemma}\label{dimension less lemma}Suppose that $\frac{3}{2}\leq p<2$. For any $\nu\in(0,1)$, there exists $\tilde m>1$ depending only upon the data, $\nu$ and $\nu_0$, such that
\begin{equation}\label{pro6.4formula}\big|\{x\in K_2:\bar u(x,t)\geq -2^{-\tilde m}\}\big|\leq\nu|K_2|\end{equation}
for all $t\in [-2^p,0]$.
\end{lemma}
\begin{proof}Before proceeding to the proof,
we introduce the auxiliary functions
\begin{equation*}\phi_k(\bar u)=\int_0^{(\bar u-k)_+}\frac{d\tau}{(-(1-\delta)k-\tau+V)^{p-1}}\end{equation*}
and
\begin{equation*}\psi_k(\bar u)=\ln\left(\frac{-(1-\delta)k+V}{-k(1-\delta)-(\bar u-k)_++V}\right)\end{equation*}
where $V=(-k)(-\delta)^{\frac{p}{p-1}}$ and $k$, $\delta\in (-\frac{1}{8},0)$ are to be determined.
Recalling the definition of $\bar u$, we have $\bar u\leq 0$ and $(\bar u-k)_+\leq -k$. Moreover, we obtain the inequality
\begin{equation*}\begin{split}(-k(1-\delta)-\tau +V)^{p-1}&\leq (-k(1-\delta)-\tau )^{p-1}+V^{p-1}\leq
(1-\delta)(-k(1-\delta)-\tau )^{p-1}\end{split}\end{equation*}
for any $0\leq\tau\leq-k$.
Taking into account
\begin{equation*}\begin{split}\phi_k(\bar u)&=\frac{1}{2-p}(-k(1-\delta)+V)^{2-p}\left\{1-\left(1-\frac{(\bar u-k)_+}{-k(1-\delta)+V}\right)^{2-p}\right\}\end{split}\end{equation*}
and $1-(1-z)^{2-p}\leq -(2-p)\ln(1-z)$, we conclude that
\begin{equation}\begin{split}\label{phi1}\phi_k(\bar u)\leq \ln\frac{-k(1-\delta)+V}{-k(1-\delta)-(\bar u-k)_++V}
\leq \ln\frac{-k(1-\delta)+V}{-k(-\delta)+V}\leq \ln\frac{1-\delta}{-\delta}.\end{split}\end{equation}
Let $\zeta=\zeta_1(x)\zeta_2(t)$ be a piecewise smooth cutoff function defined in $Q_4$. Suppose that $0\leq\zeta\leq1$ in $Q_4$,
$\zeta\equiv 1$ in $Q_2$, $\zeta\equiv 0$ on $\partial_PQ_4$, $|D\zeta_1|\leq 1$, $0\leq\partial_t\zeta_2\leq1$ and the sets $\{x\in K_4:\zeta_1(x)>-k\}$ are convex for all $k\in (-\frac{1}{8},0)$. Testing the weak formulation \eqref{subsolution} with the function
\begin{equation*}\varphi(x,t)=\frac{\zeta_1(x)^p\zeta_2(t)^p}{(-(1-\delta)k-(\bar u(x,t)-k)_++V)^{p-1}}\end{equation*}
and taking \eqref{bar A} into account, we have
\begin{equation*}\begin{split}\int_{t_1}^{t_2}\int_{K_4}&
\zeta^p\frac{\partial\phi_k(\bar u)}{\partial t}+C_0\Lambda(\omega)\frac{(p-1)\zeta^p |D(\bar u-k)_+|^p}{(-k(1-\delta)-(\bar u-k)_++V)^p} dxdt
\\&
\leq C_1\Lambda(\omega)\int_{t_1}^{t_2}\int_{K_4}\frac{p \zeta^{p-1}|D\zeta||D(\bar u-k)_+|^{p-1}}{(-k(1-\delta)-(\bar u-k)_++V)^{p-1}}dxdt\end{split}\end{equation*}
for any $[t_1,t_2]\subset (-4^p,0]$. In view of $p\geq \frac{3}{2}$, $\Lambda(\omega)\leq1$, and therefore, Young's inequality and \eqref{phi1}
allow us
to conclude that
\begin{equation*}\begin{split}\int_{K_4}&
\zeta^p \phi_k(\bar u)dx\bigg|_{t=t_1}^{t_2}+\frac{C_0\Lambda(\omega)}{4}\int_{t_1}^{t_2}\int_{K_4}\zeta^p |D\psi_k(\bar u)|^pdxdt
\\&
\leq \kappa_p\Lambda(\omega)\int_{t_1}^{t_2}\int_{K_4} \zeta^p|D\zeta|^pdxdt+
\int_{t_1}^{t_2}\int_{K_4}
p\zeta^{p-1}\phi_k(\bar u)\frac{\partial\zeta}{\partial t}dxdt\\&
\leq C_2 (t_2-t_1)\ln\frac{1-\delta}{-\delta},
\end{split}\end{equation*}
where
\begin{equation}\label{C2}\kappa_p=pC_1\left(\frac{2pC_1}{(p-1)C_0}\right)^{p-1}\qquad\text{and}\qquad
C_2=\frac{16C_1^2}{C_0}.\end{equation}
Next, we observe that
$$\big\{x\in K_4:\bar u(x,t)<-\tfrac{1}{2}\big\}\cap K_{\frac{1}{2L_2}}\subset\big\{x\in K_4:\phi_k(\bar u)=0\big\}\cap
\big\{x\in K_4:\zeta_1(x)=1\big\} $$
for any $t\in(-4^p,0]$. Applying Lemma \ref{Sobolev1} on each of the time slice
and keeping in mind \eqref{positive small cube}, we deduce
\begin{equation*}\int_{K_4}\zeta_1^p |\psi_k(\bar u)|^pdx\leq 4^{3N}\gamma^2L_2^{2(N-1)}\nu_0^{-2\frac{N-1}{N}}\int_{K_4}\zeta_1^p |D\psi_k(\bar u)|^pdx.\end{equation*}
Multiplying both sides of the above inequality by $\zeta_2(t)^p$ and integrating over $[t_1,t_2]$, we obtain
\begin{equation*}\begin{split}\int_{K_4}&
\zeta^p \phi_k(\bar u)dx\bigg|_{t=t_1}^{t_2}+\bar \Lambda(\omega)\int_{t_1}^{t_2}\int_{K_4}\zeta^p \psi_k(\bar u)^pdxdt
\leq C_2
(t_2-t_1)\ln\frac{1-\delta}{-\delta},\end{split}\end{equation*}
where
\begin{equation}\label{lambda bar}\bar \Lambda(\omega)=\gamma_3\nu_0^{2\frac{N-1}{N}}\Lambda(\omega)\qquad\text{and}\qquad
\gamma_3=4^{-3N-1}\gamma^{-2}L_2^{-2(N-1)}C_0^{-1}.\end{equation}
Recalling that $u\in C_{\loc}(0,T;L_{\loc}^2(\Omega))$, we obtain $\phi_k(\bar u),\ \psi_k(\bar u)^p\in  C_{\loc}(0,T;L_{\loc}^1(\Omega))$
and for any $t\in(-4^p,0]$ there holds
\begin{equation}\begin{split}\label{Dini}\frac{d^-}{dt}\int_{K_4\times\{t\}}&
\zeta^p \phi_k(\bar u)dx+\bar \Lambda(\omega)\int_{K_4\times\{t\}}\zeta^p \psi_k(\bar u)^pdx
\leq C_2
\ln\frac{1-\delta}{-\delta},\end{split}\end{equation}
where
$$\frac{d^-}{dt}\int_{K_4\times\{t\}}
\zeta^p \phi_k(\bar u)dx=\lim\sup_{h\to0+}\frac{1}{h}\left(
\int_{K_4\times\{t\}}
\zeta^p \phi_k(\bar u)dx-\int_{K_4\times\{t-h\}}
\zeta^p \phi_k(\bar u)dx\right).$$
At this stage, we introduce the quantities
\begin{equation*}Y_i=\sup_{-4^p\leq t\leq 0}\int_{K_4\cap [\bar u>-|\delta|^i]}\zeta^p (x,t)dx,\quad i=1,2,\cdots.\end{equation*}
To prove this lemma, it suffices to determine constants $\delta$ and $i_*$ depending only upon the data, $\nu$ and $\nu_0$, such that
$Y_{i_*}\leq\nu$.

Fix $i\in\mathbb{N}$,
we choose $k=-|\delta|^i$ and $V=|\delta|^{i+\frac{p}{p-1}}$.
For any $\epsilon>0$, there exists $t_0\in (-4^p,0]$ such that
\begin{equation*}\int_{K_4\cap [\bar u>-|\delta|^{i+1}]}\zeta^p (\cdot,t_0)dx\geq Y_{i+1}-\epsilon.\end{equation*}
Let
$$\mathcal{C}^+=\left\{t\in(-4^p,0]:\frac{d^-}{dt}\int_{K_4\times\{t\}}\zeta^p \phi_k(\bar u)dx\geq 0\right\}.$$
In the case $t_0\in \mathcal{C}^+$, we choose $\epsilon\in(0,\nu/2]$. In \eqref{Dini} take $t=t_0$, $k=-|\delta|^i$ and there holds
\begin{equation*}\begin{split}\bar \Lambda(\omega)\int_{K_4}\zeta^p (\cdot,t_0)\psi_{-|\delta|^i}(\bar u)^p(\cdot,t_0)dx
\leq C_2
\ln\frac{1-\delta}{-\delta}.\end{split}\end{equation*}
To estimate below the integral on the left-hand side, take into account
the domain of integration $K_4\cap [\bar u>-|\delta|^{i+1}]$. On such a set, $(\bar u+|\delta|^i)_+>|\delta|^i-|\delta|^{i+1}$.
Taking into account $V\leq |\delta|^{i+1}$, we deduce
\begin{equation*}\begin{split}\psi_k(\bar u)=\ln\left(\frac{(1-\delta)|\delta|^i+V}{(1-\delta)|\delta|^i-(\bar u+|\delta|^i)_++V}\right)\geq\ln\frac{1+|\delta|}{3|\delta|}.\end{split}\end{equation*}
Since $|\delta|<\frac{1}{8}$, then there holds
$$\frac{1+|\delta|}{|\delta|}\leq\left(\frac{1+|\delta|}{3|\delta|}\right)^2.$$
We use these estimates to conclude that
\begin{equation*}\begin{split}\bar \Lambda(\omega)\int_{K_4\cap [\bar u>-|\delta|^{i+1}]}\zeta^p (\cdot,t_0)dx
\leq 2C_2
\left(\ln\frac{1+|\delta|}{3|\delta|}\right)^{1-p}\end{split}\end{equation*}
and therefore
\begin{equation*}Y_{i+1}\leq \frac{\nu}{2}+\frac{2C_2}{\bar \Lambda(\omega)}
\left(\ln\frac{1+|\delta|}{3|\delta|}\right)^{1-p}.\end{equation*}
Then $Y_{i+1}\leq\nu$, provided
$$|\delta|\leq\frac{1}{3}\left(\exp\left(\frac{\nu\bar\Lambda(\omega)}{2C_2}\right)^{\frac{1}{1-p}}-\frac{1}{3}\right)^{-1}.$$
To this end, we choose
\begin{equation}\label{delta1}|\delta|= \frac{1}{30}\exp\left\{-64\left(\frac{4C_2}{\nu\bar\Lambda(\omega)}\right)^2\right\}.\end{equation}
We now turn our attention to the case $t_0\notin \mathcal{C}^+$. Denote by $t_*$ the least upper bound of the set
$$\mathcal{C}_{t_0}^+=\{t\in \mathcal{C}^+:t< t_0 \}.$$
In the case $t_0=t_*$, there exists a sequence $t_k'\in \mathcal{C}_{t_0}^+$ such that
$t_k'\to t_0$ as $k\to\infty$, and
\begin{equation*}\begin{split}\bar \Lambda(\omega)\int_{K_4}\zeta^p (\cdot,t_k')\psi_{-|\delta|^i}(\bar u)^p(\cdot,t_k')dx
\leq C_2
\ln\frac{1-\delta}{-\delta}.\end{split}\end{equation*}
Passing to the limit $t_k'\to t_0$ and taking into account that $\psi_{-|\delta|^i}(\bar u)^p\in  C_{\loc}(0,T;L_{\loc}^1(\Omega))$, we deduce
\begin{equation*}\begin{split}\bar \Lambda(\omega)\int_{K_4}\zeta^p (\cdot,t_0)\psi_{-|\delta|^i}(\bar u)^p(\cdot,t_0)dx
\leq C_2
\ln\frac{1-\delta}{-\delta}.\end{split}\end{equation*}
We may now repeat the same
arguments as in the previous proof and obtain $Y_{i+1}\leq\nu$, provided $\delta$ satisfies \eqref{delta1}.

Finally, we come to the case $t_0>t_*$. Let $j$ be an integer that can
be determined a priori only in terms of the data, $\nu$ and $\nu_0$.
Let us initially assume $Y_i>\nu$ for all $i=1,2,\cdots,j$.
Next, we claim that
\begin{equation}\label{claim3}Y_{i+1}\leq (1-|\delta|)Y_i\qquad i=1,2,\cdots,j.\end{equation}
Fix $i\in\{1,2,\cdots,j\}$. Since $t_0>t_*$, then there holds
$$\frac{d^-}{dt}\int_{K_4}\zeta^p(\cdot,t) \phi_{-|\delta|^i}(\bar u)(\cdot,t)dx\leq 0$$
for any $t\in (t_*,t_0)$.
It follows that
\begin{equation}\begin{split}\label{t0t*}\int_{K_4}\zeta^p(\cdot,t_0) \phi_{-|\delta|^i}(\bar u)(\cdot,t_0)dx\leq \int_{K_4}\zeta^p(\cdot,t_*) \phi_
{-|\delta|^i}(\bar u)(\cdot,t_*)dx.\end{split}\end{equation}
Recalling the definition of $t_*$, we have
\begin{equation}\begin{split}\label{Cdelta}\bar \Lambda(\omega)\int_{K_4}\zeta^p (\cdot,t_*)\psi_{-|\delta|^i}(\bar u)^p(\cdot,t_*)dx
\leq C_2
\ln\frac{1-\delta}{-\delta}=:C_\delta,\end{split}\end{equation}
with the obvious meaning of $C_\delta$.
Consider the set $K_4\cap [(\bar u+|\delta|^i)_+>\tau |\delta|^i]$ where $\tau\in [0,1]$. On this set
\begin{equation*}\begin{split}\psi_{-|\delta|^i}(\bar u)=\ln\left(\frac{(1-\delta)|\delta|^i+V}{(1-\delta)|\delta|^i-(\bar u+|\delta|^i)_++V}\right)\geq\ln\frac{1+2|\delta|}{1+2|\delta|-\tau},\end{split}\end{equation*}
since $V\leq|\delta|^{i+1}$. Combining this estimate with \eqref{Cdelta}, we obtain
\begin{equation*}\int_{K_4\cap [(\bar u+|\delta|^i)_+>\tau |\delta|^i]}\zeta^p(\cdot,t_*)dx\leq \frac{C_\delta}{\bar\Lambda(\omega)}\left(\ln\frac
{1+2|\delta|}{1+2|\delta|-\tau}\right)^{-p}.\end{equation*}
At this stage, we set
\begin{equation*}\tau_*=\frac{\exp\left(\frac{C_\delta}{\bar\Lambda(\omega)Y_i}\right)^{\frac{1}{p}}-1}{
\exp\left(\frac{C_\delta}{\bar\Lambda(\omega)Y_i}\right)^{\frac{1}{p}}}(1+2|\delta|).\end{equation*}
Since $Y_i>\nu$, then we have
\begin{equation*}\tau_*<\frac{\exp\left(\frac{C_\delta}{\bar\Lambda(\omega)\nu}\right)^{\frac{1}{p}}-1}{
\exp\left(\frac{C_\delta}{\bar\Lambda(\omega)\nu}\right)^{\frac{1}{p}}}(1+2|\delta|)=:\sigma(1+2|\delta|),\end{equation*}
with the obvious meaning of $\sigma$. For a technical reason, we introduce a constant $\sigma'>\sigma$ defined by
\begin{equation*}\sigma'=\frac{\exp\left(\frac{2C_\delta}{\bar\Lambda(\omega)\nu}\right)^{\frac{1}{p}}-1}{
\exp\left(\frac{2C_\delta}{\bar\Lambda(\omega)\nu}\right)^{\frac{1}{p}}}.\end{equation*}
At this point, we claim that
\begin{equation}\label{claim4}\sigma'(1+2|\delta|)<1-\sqrt{|\delta|}.\end{equation}
For the value of $|\delta|$ given by \eqref{delta1}, we have
\begin{equation*}|\delta|< \frac{1}{2}\exp\left\{-2^{\frac{p}{p-1}}\left(\frac{4C_2}{\nu\bar\Lambda(\omega)}\right)^\frac{1}{p-1}\right\}\end{equation*}
for all $\frac{3}{2}\leq p<2$. Then there holds
\begin{equation}\label{tau*}\left(\ln\left(\frac{1}{2|\delta|}\right)^{\frac{4C_2}{\nu\bar\Lambda(\omega)}}\right)^{\frac{1}{p}}
\leq\ln\left(\frac{1}{2|\delta|}\right)^{\frac{1}{2}}.\end{equation}
From
\eqref{delta1},
we observe that
$|\delta|<(2+2\sqrt{2})^{-1}$
and this implies
$|\delta|^{-1}(1+|\delta|)<(2|\delta|)^{-2}.$
Combining this inequality with \eqref{tau*}, we obtain
\begin{equation*}\begin{split}\exp\left(\frac{2C_\delta}{\bar\Lambda(\omega)\nu}\right)^{\frac{1}{p}}
&=\exp\left(\frac{2C_2}{\bar\Lambda(\omega)\nu}\ln\frac{1+|\delta|}{|\delta|}\right)^{\frac{1}{p}}
\leq \exp\left(\frac{4C_2}{\bar\Lambda(\omega)\nu}\ln\frac{1}{2|\delta|}\right)^{\frac{1}{p}}\leq
\frac{1}{\sqrt{2|\delta|}}.\end{split}\end{equation*}
By \eqref{delta1}, this choice of $\delta$ yields $|\delta|<[4(3+2\sqrt{2})]^{-1}$ and there holds $(1-\sqrt{2|\delta|})(1+2|\delta|)\leq 1-\sqrt{|\delta|}$. Then we have
\begin{equation*}\sigma'(1+2|\delta|)<\left(1-\exp\left(\frac{2C_\delta}{\bar\Lambda(\omega)\nu}\right)^{-\frac{1}{p}}
\right)(1+2|\delta|)\leq 1-\sqrt{|\delta|}.\end{equation*}
This implies the claimed estimate \eqref{claim4}.
We now proceed to estimate the right-hand side of \eqref{t0t*}. Using a change of variable $\tau'=|\delta|^{-i}\tau$, we obtain
\begin{equation*}\begin{split}\int_{K_4}&\zeta^p(\cdot,t_*) \phi_{-|\delta|^i}(\bar u)(\cdot,t_*)dx\\&=\int_{K_4}\zeta^p(\cdot,t_*)
\int_0^{(\bar u+|\delta|^i)_+}\frac{d\tau}{((1-\delta)|\delta|^i-\tau+V)^{p-1}}dx\\&
\leq \int_0^1\frac{|\delta|^{(2-p)i}}{(1+|\delta|-\tau)^{p-1}}
\left(\int_{K_4\cap [\bar u+|\delta|^i)_+>\tau |\delta|^i]}\zeta^p(\cdot,t_*)dx\right)d\tau.
\end{split}\end{equation*}
and this yields
\begin{equation*}\begin{split}\int_{K_4}&\zeta^p(\cdot,t_0) \phi_{-|\delta|^i}(\bar u)(\cdot,t_0)dx\\&\leq
\int_0^{\tau_*}\frac{|\delta|^{(2-p)i}}{(1+|\delta|-\tau)^{p-1}}Y_id\tau+\frac{C_\delta}{\bar\Lambda(\omega)}
\int_{\tau_*}^1\frac{|\delta|^{(2-p)i}}{(1+|\delta|-\tau)^{p-1}}
\left(\ln\frac
{1+2|\delta|}{1+2|\delta|-\tau}\right)^{-p}d\tau.\end{split}\end{equation*}
Taking into account that $Y_i>\nu$, we obtain
\begin{equation*}\begin{split}\int_{K_4}\zeta^p(\cdot,t_0) \phi_{-|\delta|^i}(\bar u)(\cdot,t_0)&dx\leq
|\delta|^{(2-p)i}Y_iG(Y_i,\delta)\end{split}\end{equation*}
where
\begin{equation*}G(Y_i,\delta)=\int_0^1\frac{d\tau}{(1+|\delta|-\tau)^{p-1}}-\int_{\sigma(1+2|\delta|)}^1
\left(1-\frac{C_\delta}{\nu\bar\Lambda(\omega)}\left(\ln\frac
{1+2|\delta|}{1+2|\delta|-\tau}\right)^{-p}\right)\frac{d\tau}{(1+|\delta|-\tau)^{p-1}}.\end{equation*}
Moreover, we rewrite the above estimate as
\begin{equation}\begin{split}\label{Yi}\int_{K_4}\zeta^p(\cdot,t_0) \phi_{-|\delta|^i}(\bar u)(\cdot,t_0)&dx\leq
Y_i(1-f(\delta))\int_0^{1-|\delta|}\frac{|\delta|^{(2-p)i}}{(1+|\delta|-\tau)^{p-1}}d\tau
\end{split}\end{equation}
where the function $f(\delta)$ satisfies
\begin{equation*}\begin{split}f(\delta)&\int_0^{1-|\delta|}\frac{d\tau}{(1+|\delta|-\tau)^{p-1}}
\\&=\int_{\sigma(1+2|\delta|)}^1
\left(1-\frac{C_\delta}{\nu\bar\Lambda(\omega)}\left(\ln\frac
{1+2|\delta|}{1+2|\delta|-\tau}\right)^{-p}\right)\frac{d\tau}{(1+|\delta|-\tau)^{p-1}}
-\int_{1-|\delta|}^{1}\frac{d\tau}{(1+|\delta|-\tau)^{p-1}}.\end{split}\end{equation*}
To estimate below the integral on the left-hand side of \eqref{t0t*}, take into account
the domain of integration $K_4\cap [\bar u>-|\delta|^{i+1}]$.
Recalling that $V=|\delta|^{i+\frac{p}{p-1}}$, we deduce
\begin{equation*}\begin{split}((1-\delta)|\delta|^i-\tau+V)^{p-1}\leq
((1-\delta)|\delta|^i-\tau)^{p-1}+|\delta||\delta|^{(i+1)(p-1)}\leq (1-\delta)((1-\delta)|\delta|^i-\tau)^{p-1},\end{split}\end{equation*}
for any $\tau<|\delta|^i$.
Then we find that
\begin{equation*}\begin{split}\int_{K_4}&\zeta^p(\cdot,t_0) \phi_{-|\delta|^i}(\bar u)(\cdot,t_0)dx
\\&=\int_{K_4}\zeta^p(\cdot,t_0)
\int_0^{(\bar u+|\delta|^i)_+}\frac{d\tau}{((1-\delta)|\delta|^i-\tau+V)^{p-1}}dx
\\&\geq \int_{K_4\cap [\bar u>-|\delta|^i]}\zeta^p(\cdot,t_0)
\int_0^{(\bar u+|\delta|^i)_+}\frac{d\tau }{(1-\delta)((1-\delta)|\delta|^i-\tau)^{p-1}}dx
\\&\geq \int_{K_4\cap [\bar u>-|\delta|^{i+1}]}\zeta^p(\cdot,t_0)
\int_0^{|\delta|^i-|\delta|^{i+1}}\frac{d\tau }{(1-\delta)((1-\delta)|\delta|^i-\tau)^{p-1}}dx.
\end{split}\end{equation*}
Applying a change of variable $\tau'=|\delta|^{-i}\tau$, we obtain
\begin{equation*}\begin{split}\int_{K_4}\zeta^p(\cdot,t_0) \phi_{-|\delta|^i}(\bar u)(\cdot,t_0)&dx\geq
(1-\delta)^{-1}(Y_{i+1}-\epsilon)\int_0^{1-|\delta|}\frac{|\delta|^{(2-p)i}}{(1+|\delta|-\tau)^{p-1}}d\tau.
\end{split}\end{equation*}
Combining this estimate with \eqref{Yi}, we conclude that
\begin{equation}\label{Yi+1}Y_{i+1}\leq (1-f(\delta))(1-\delta)^{-1}Y_i+\epsilon.\end{equation}
Finally, we need to show that $f(\delta)>\delta^2$. The strategy of proof is exactly the same as in \cite[Page 378-379]{AZ} and
we include the proof here for the sake of completeness.
Since $\sigma'>\sigma$ and
\begin{equation*}\frac{C_\delta}{\nu \bar \Lambda(\omega)}=\frac{1}{2}\left(\ln\frac{1}{1-\sigma'}\right)^p.\end{equation*}
It follows that
\begin{equation*}\frac{C_\delta}{\nu \bar \Lambda(\omega)}\left(\ln\frac{1+2|\delta|}{1+2|\delta|-\tau}\right)^{-p}\leq\frac{1}{2}\quad\text{for\ \ all}\quad\tau\in(\sigma'(1-2\delta),1).
\end{equation*}
Then we arrive at
\begin{equation*}f(\delta)\int_0^{1-|\delta|}\frac{d\tau}{(1+|\delta|-\tau)^{p-1}}
\geq \frac{1}{2}\int_{\sigma'(1+2|\delta|)}^1\frac{d\tau}{(1+|\delta|-\tau)^{p-1}}-\int_{1-|\delta|}^1
\frac{d\tau}{(1+|\delta|-\tau)^{p-1}}.\end{equation*}
Taking into account that $\sigma'(1+2|\delta|)<1-\sqrt{|\delta|}$, we obtain
\begin{equation*}\begin{split}f(\delta)&>\frac{(|\delta|+\sqrt{|\delta|})^{2-p}-|\delta|^{2-p}-2((2|\delta|)^{2-p}-|\delta|^{2-p})}
{2((1+|\delta|)^{2-p}-(2|\delta|)^{2-p})}
\\&>\frac{(|\delta|(1+|\delta|))^{\frac{2-p}{2}}-|\delta|^{2-p}-2((2|\delta|)^{2-p}-|\delta|^{2-p})}
{2((1+|\delta|)^{2-p}-(2|\delta|)^{2-p})}.\end{split}\end{equation*}
By \eqref{delta1}, this choice of $\delta$ yields $|\delta|<e^{-64\ln2}$ and therefore
$$|\delta|\leq \frac{\delta'}{1-\delta'},\qquad \text{where}\qquad\delta'=\inf_{z\in(1,2)}(1+32(2-z)\ln2)^{\frac{2}{z-2}}.$$
From this estimate, we find that
\begin{equation}\label{relationdelta}(|\delta|(1+|\delta|))^{\frac{2-p}{2}}-|\delta|^{2-p}\geq 8
((2|\delta|)^{2-p}-|\delta|^{2-p})\geq 8(2-p)|\delta|^{2-p}\ln2,\end{equation}
since $2^{2-p}-1\geq (2-p)\ln2$. Taking into account that $\ln(1+|\delta|^{-1})\leq |\delta|^{-1}$, we obtain
\begin{equation}\begin{split}\label{relationdelta1}
((1+|\delta|)^{2-p}&-(2|\delta|)^{2-p})=|\delta|^{2-p}\int_{2-p}^{(2-p)\ln(1+|\delta|^{-1})}e^\xi d\xi\\&\leq
(2-p)|\delta|^{2-p}(1+|\delta|^{-1})\ln(1+|\delta|^{-1})\leq 2(2-p)|\delta|^{-p},
\end{split}\end{equation}
From \eqref{relationdelta} and \eqref{relationdelta1}, we obtain
\begin{equation*}\begin{split}f(\delta)&
>\frac{3}{8}\frac{(|\delta|(1+|\delta|))^{\frac{2-p}{2}}-|\delta|^{2-p}}
{(1+|\delta|)^{2-p}-(2|\delta|)^{2-p}}\geq \frac{3\ln2}{2}|\delta|^2>\delta^2.\end{split}\end{equation*}
Therefore, we conclude from \eqref{Yi+1} that $Y_{i+1}\leq (1-|\delta|)Y_i$ and the claim \eqref{claim3} follows.

From \eqref{claim3}, by iteration
\begin{equation*}Y_{j+1}\leq (1-|\delta|)^jY_1\leq (1-|\delta|)^j|K_4|.\end{equation*}
Having fixed $\nu\in(0,1)$, one can choose
\begin{equation*}j= 5+\left[\frac{\ln\dfrac{4^N}{\nu}}{\ln\dfrac{1}{1-|\delta|}}\right],\end{equation*}
where $[\cdot]$ denotes the integer portion of the number. For such a choice, $Y_{j+1}\leq\nu$ and hence
\begin{equation*}\big|\{x\in K_2: \bar u(x,t)>-|\delta|^{j+1}\}\big|\leq \nu|K_2|\quad\text{for\ \ all}\quad t\in[-2^p,0].\end{equation*}
Moreover, by \eqref{delta1}, there exists $\gamma_3>2^{10}$, depending only upon the data $\{N,C_0,C_1\}$, such that
\begin{equation*}\exp\left\{\left(\frac{\gamma_4}{\nu\bar\Lambda(\omega)}\right)^2\right\}
\geq\left(\frac{1}{|\delta|}\right)^2\ln\left(\dfrac{4^N}{\nu}\right).\end{equation*}
We choose
\begin{equation}\label{tilde m}
\tilde m=\tilde m(\nu)=\exp\left\{\left(\frac{\gamma_4}{\nu\bar\Lambda(\omega)}\right)^2\right\}.\end{equation}
Taking into account that
\begin{equation*}|\delta|\leq\ln\frac{1}{1-|\delta|}\leq2|\delta|\qquad\text{and}\qquad \ln\frac{1}{|\delta|}\leq
\frac{1}{\sqrt{|\delta|}},\end{equation*}
the choice \eqref{tilde m} guarantees that $2^{-\tilde m}\leq|\delta|^{j+1}$. For such a choice of $\tilde m$, we obtain the desired estimate \eqref{pro6.4formula}. Finally, if $Y_{i_0}\leq\nu$ for some $i_0\in\{1,2,\cdots,j\}$, then
the estimate \eqref{pro6.4formula} holds as well, for the same choice of $\tilde m$ as in \eqref{tilde m}. This concludes the
proof of the lemma.
\end{proof}
From \eqref{structure}, \eqref{lambda bar} and \eqref{tilde m}, we remark that $\tilde m$ can be chosen independent of $m_2$.
Transforming back to the original function $u$ and original variables $(x,t)$, we obtain an estimate for the measure of the level sets
\begin{equation}\label{transform back}\big|\{x\in \bar x+K_{4c_2R}:u(x,t)>\mu_+-\frac{\omega}{2^{m_2+\tilde n+\tilde m-1}}\}\big|
<\nu |K_{4c_2R}|,\end{equation}
for all $\tilde t-2^{-p}(\tilde t-t_0)<t<\tilde t$.
With the help of this estimate we can now prove a DeGiorgi-type lemma.
\begin{lemma}\label{DeGiorgi4}Let $\tilde n>1$ be the constant chosen according to \eqref{tilde n}. Then there exist a constant $\tilde m
>1$ depending only upon the data and $\nu_0$, and a time level $t_\omega\in (t_0,\tilde t)$ such that
\begin{equation}\label{space expanding}u(x,t)<\mu_+-\frac{\omega}{2^{m_2+\tilde n+\tilde m}}\qquad\text{a.e.}\quad\text{in}\quad(\bar x,\tilde t)+ Q\left( \frac{\tilde t-t_\omega}{2^p},2c_2R\right).\end{equation}
\end{lemma}
\begin{proof} Without loss of generality, we may assume $(\bar x,\tilde t)=(0,0)$. For $\nu\in(0,1)$ to be determined later we take $\tilde m
=\tilde m(\nu)>1$ according to \eqref{tilde m}. For $n=0,1,2,\cdots$, set
\begin{equation*}R_n=2c_2R+\frac{2c_2R}{2^{n}},\quad k_n=\mu_+-\frac{\omega}{2^{m_2+\tilde n+\tilde m}}
-\frac{\omega}{2^{m_2+\tilde n+\tilde m+n}}\quad\text{and}\quad
Q_n= Q\left(\frac{\tilde t-t_\omega}{2^{p}}+\frac{\tilde t-t_\omega}{2^{p+pn}},R_n\right)\end{equation*}
where the time level $t_*$ is taken as
\begin{equation*}t_\omega=\tilde t-\left(\frac{\nu_0}{4}\right)\left(\frac{1}{2^{\tilde n+\tilde m}}\right)^{2-p}\left(\frac{\omega}{2^{m_2}}\right)^{(p-1)(p-2)}
R^p.\end{equation*}
From \eqref{tilde n}, \eqref{structure}, \eqref{lambda bar} and \eqref{tilde m}, we observe that
$$\tilde m(\nu)\geq \tilde m(1)>\tilde n>(p-1)\tilde n$$
for any $\nu\in(0,1)$. So, we conclude from \eqref{level tilde t} that $t_\omega\in (t_0,\tilde t)$.
Take piecewise smooth cutoff functions $\zeta_n$ in $Q_n$, such that  $0\leq\zeta_n\leq1$, $\zeta_n\equiv1$ in $Q_{n+1}$, $\zeta_n\equiv0$ on $\partial_PQ_n$,
$|D\zeta_n|\leq 2^n/(c_2R)$ and $0<\partial_t\zeta_n\leq 2^{p+1}2^{pn}/(\tilde t-t_\omega)$.
Write down the energy estimates \eqref{Caccioppoli} for the
truncated functions $(u-k_n)_+\zeta_n^p$ over the cylinders $Q_n$.
Taking into account that
$$U(K_{R_n},-2^{-p}(\tilde t-t_\omega)-2^{-p-pn}(\tilde t-t_\omega),0,(u-k_n)_{+}\zeta_n^p)=0,$$
we deduce
\begin{equation*}\begin{split}\esssup_{-2^{-p}(\tilde
 t-t_\omega)-2^{-p-pn}(\tilde t-t_\omega)<t<0}&\int_{K_{R_n}\times\{t\}}(u-k_n)_{+}^2\zeta_n^p dx
+\iint_{Q_n}|D(u-k_n)_{+}\zeta_n|^p dxdt\\
&\leq 2^{2n+6}\gamma_1\nu_0^{-1}R^{-p}\left(\frac{\omega}{2^{m_2}}\right)^{(3-p)p}\left(\frac{1}{2^{\tilde m+\tilde n}}\right)^p
\iint_{Q_n}\chi_{[(u-k_n)_+>0]}dxdt.
\end{split}\end{equation*}
At this point, we set
\begin{equation*}A_n=\iint_{Q_n}\chi_{[(u-k_n)_+>0]}dxdt\qquad\text{and}\qquad Y_n=\frac{A_n}{|Q_n|}.\end{equation*}
Applying parabolic Sobolev's inequality \eqref{Sobolevf}, we obtain
\begin{equation*}\begin{split}
\iint_{Q_n}&|(u-k_n)_{+}|^{p(1+\frac{2}{N})}\zeta_n^{p(1+\frac{2}{N})} dxdt\\&\leq \gamma
\left(\esssup_{-2^{-p}(\tilde
 t-t_\omega)-2^{-p-pn}(\tilde t-t_\omega)<t<0}\int_{K_{R_n}\times\{t\}}(u-k_n)_{+}^2\zeta_n^2 dx\right)^{\frac{p}{N}}
\iint_{Q_n}|D(u-k_n)_{+}\zeta_n|^p dxdt\\&
\leq \gamma_5\nu_0^{-(1+\frac{p}{N})}2^{2n(1+\frac{p}{N})}R^{-p(1+\frac{p}{N})}\left(\frac{\omega}{2^{m_2}}
\right)^{(3-p)p(1+\frac{p}{N})}\left(\frac{1}{2^{\tilde m+\tilde n}}\right)^{p(1+\frac{p}{N})}
A_n^{1+\frac{p}{N}},\end{split}\end{equation*}
for a constant $\gamma_5$ depending only upon the data.
The integral on the left-hand side is estimated below by
\begin{equation*}\begin{split}\iint_{Q_n}&|(\bar u-k_n)_{+}|^{p(1+\frac{2}{N})}\zeta_n^{p(1+\frac{2}{N})} dxdt
\geq (k_{n+1}-k_n)^{p\frac{N+2}{N}}A_{n+1}
\geq \left(\frac{\omega}{2^{m_2+\tilde n+\tilde m+n+1}}\right)^{p\frac{N+2}{N}}A_{n+1} .\end{split}\end{equation*}
Combining the estimates above and keeping in mind
\begin{equation*}|Q_n|\leq 4^NL_2^N\nu_0\left(\frac{\omega}{2^{m_2}}\right)^{(p-2)N+(p-1)(p-2)}\left(\frac{1}{2^{\tilde n+\tilde m}}\right)^{2-p}R^{N+p},
\end{equation*}
we infer that
\begin{equation*}Y_{n+1}\leq \gamma_6\nu_0^{-1}2^{4n(1+\frac{2}{N})}Y_n^{1+\frac{p}{N}}
\leq \gamma_6\nu_0^{-1}2^{4n(1+\frac{2}{N})}Y_n^{1+\frac{2}{N}},\end{equation*}
for a constant $\gamma_6$ depending only upon the data. At this point, we set
\begin{equation}\label{nu2}\nu_2=2^{-N(N+2)}\gamma_6^{-\frac{N}{2}}\nu_0^{\frac{N}{2}}.\end{equation}
Choose $\nu=\nu_2$ in Lemma \ref{dimension less lemma}, and hence $\tilde m=\tilde m(\nu_2)$, from this and \eqref{tilde m}. Moreover, we conclude from \eqref{transform back} that
$Y_0\leq\nu_2.$
By the lemma on fast geometric convergence of sequences, we infer that $Y_n\to0$, as $n\to\infty$, which proves the lemma.
\end{proof}
\subsection{Expansion of positivity in time variable}\label{m2}
The aim of this subsection is to establish a DeGiorgi-type result similar to that of Lemma \ref{DeGiorgi4}.
To start with, we determine the constant $m_2$ in terms of the data and $\omega$.
Let $\tilde n$ and $\nu_2$ be the constants determined by \eqref{tilde n} and \eqref{nu2}, respectively. In \eqref{tilde m}, take $\nu=\nu_2$
and choose $\tilde m=\tilde m(\nu_2)$.
At this point, we
choose $m_2>2$ be such that
\begin{equation}\label{m2f}m_2p=m_2+\tilde n+\tilde m+10\qquad \text{i.\ e.,}\qquad m_2=\frac{\tilde n+\tilde m+10}{p-1},\end{equation}
and hence $m_1=m_2p/(p-1)$. This determines the precise formulation of the intrinsic parabolic cylinders. Next, we consider some
geometric properties of these cylinders.

From now on, we assume that $p>\frac{1+\sqrt{5}}{2}$ and set $\kappa=\frac{1}{2}(p-1-\frac{1}{p})>0$. Moreover, we will need the following
assumption:
\begin{equation}\label{R} R\leq\min\left\{\left(\frac{1}{8L_2}\right)^{\frac{1}{\kappa}},\left(\frac{1}{16}\right)^{\frac{1}{(p-1)^2}}\right\}.\end{equation}
In the case $2^{-m_2}\omega\geq R$. Keeping in mind $L_1=2$, we conclude from \eqref{R} that $Q(8c_1R^p,8c_2R)\subset \overline Q$.
While in the case $2^{-m_2}\omega\leq R$ we conclude from \eqref{nu0}, \eqref{tilde n}, \eqref{structure}, \eqref{lambda bar},
\eqref{tilde m}, \eqref{nu2} and \eqref{m2f} that
\begin{equation*}\exp\left(-\exp\left(\exp \frac{\gamma_8}{\omega^{\gamma_7}}\right)\right)\leq R\end{equation*}
for some constants $\gamma_7$ and $\gamma_8$ depending only upon the data. From this inequality, we obtain a decay estimate for $\omega$ as follows
\begin{equation}\omega\leq \left(\frac{\gamma_8}{\ln\ln\ln\frac{1}{R}}\right)^{\frac{1}{\gamma_7}}.\end{equation}
We now turn our attention to the case $Q(8c_1R^p,8c_2R)\subset \overline Q$. In this case we have already established the
estimate \eqref{space expanding}.
Before proceeding further, let us remark that this estimate implies
\begin{equation}\label{starting}u(x,\tilde t)<\mu_+-\frac{\omega}{2^{m_2+\tilde n+\tilde m}}
\qquad\text{a.e.}\quad x\in K_{c_2R}.\end{equation}
The next lemma deals with the expansion of positivity in time, starting from $\tilde t$.
\begin{lemma}\label{3rd expand time 6.6}For any $\nu\in(0,1)$, there exist a constant
$\gamma_9$ depending only upon the data, and a time level $t^{(0)}=\tilde t+\gamma_9\nu\omega^{(1-p)(2-p)}R^p$
such that
\begin{equation}\label{expand level}\big|\{x\in K_{\frac{1}{2}c_2R}:u(x,t)>\mu_+-\frac{\omega}{2^{m_2p-1}}\}\big|
<\nu |K_{\frac{1}{2}c_2R}|\end{equation}
for any $t\in (\tilde t, t^{(0)})$.
\end{lemma}
\begin{proof}Set $k=\mu_+-\frac{\omega}{2^{m_2p-4}}$ and $c=\frac{\omega}{2^{m_2p-1}}$.
We consider the logarithmic function
\begin{equation*}\begin{split}\psi^+=\ln^+\left(\frac{\frac{\omega}{2^{m_2p-4}}}{\frac{\omega}{2^{m_2p-4}}-(u-k)_++c}\right).
\end{split}\end{equation*}
Then we have $\psi^+\leq 8\ln2$ and
\begin{equation*}\begin{split}[(\psi^+)']^{2-p}\leq \left(\frac{1}{c}\right)^{2-p}
\leq \left(\frac{\omega}{2^{m_2p-1}}\right)^{p-2}.\end{split}\end{equation*}
Choose a piecewise smooth cutoff function $\zeta(x)$, defined in $K_{c_2R}$, and satisfying $0\leq \zeta\leq1$ in $K_{c_2R}$,
$\zeta\equiv 1$ in $K_{\frac{1}{2}c_2R}$ and $|D\zeta|\leq 2(c_2R)^{-1}$.
Since $k>\mu_+-\frac{\omega}{4}$, it is easy to check that $U(K_{c_2R},\tilde t,t^{(0)},2\psi^{+}\left(\psi^{+}\right)^\prime\zeta^p)=0$.
Then we obtain from \eqref{logarithmic} the logarithmic estimate
\begin{equation*}\begin{split}\esssup_{\tilde t<t<t^{(0)}}&\int_{K_{c_2R}\times\{t\}}\left(\psi^{+}\right)^2\zeta^p dx\leq
\int_{K_{c_2R}\times\{\tilde t\}}\left(\psi^{+}\right)^2\zeta^p dx+\gamma_2\int_{\tilde t}^{t^{(0)}}\int_{
K_{c_2R}}\psi^{+}|\left(\psi^{+}\right)^\prime|^{2-p}
|D\zeta|^pdxdt\end{split}\end{equation*}
From \eqref{starting}, we observe that
$\int_{K_{c_2R}\times\{\tilde t\}}\left(\psi^{+}\right)^2\zeta^p dx=0$
and there holds
\begin{equation}\begin{split}\label{FFF}\esssup_{\tilde t<t<t^{(0)}}\int_{K_{c_2R}\times\{t\}}&\left(\psi^{+}\right)^2\zeta^p dx\leq\gamma_2\int_{\tilde t}^{t^{(0)}}\int_{
K_{c_2R}}\psi^{+}|\left(\psi^{+}\right)^\prime|^{2-p}
|D\zeta|^pdxdt\\&
\leq
2^{N+5}(\ln2)\gamma_2L_2^{-1}
\omega^{(1-p)(p-2)}R^{-p} (t^{(0)}-\tilde t)|K_{\frac{1}{2}c_2R}|\leq\nu |K_{\frac{1}{2}c_2R}|,\end{split}\end{equation}
provided
\begin{equation*}\gamma_9=2^{-N-5}\gamma_2^{-1}L_2.\end{equation*}
To estimate below the integral
on the left-hand side, we consider a smaller set defined by
$$S=\big\{x\in K_{\frac{1}{2}c_2R}:u(x,t)>\mu_+-\frac{\omega}{2^{m_2p-1}}\big\}\subset K_{\frac{1}{2}c_2R}.$$
On such a set, $\zeta\equiv 1$ and $\psi^+\geq 2\ln2>1$. Combining this inequality with \eqref{FFF}, we obtain
\eqref{expand level}, which proves the lemma.
\end{proof}
With the help of this lemma we can now prove the following DeGiorgi-type result.
\begin{lemma}\label{DeGiorgi initial}There exists a constant $\nu_3\in(0,1)$ depending only upon the data
such that
\begin{equation}\label{DeGiorgi5}u(x,t)<\mu_+-\frac{\omega}{2^{m_2p}}\qquad\text{a.e.}\quad\text{in}\quad K_{\frac{1}{4}c_2R}
\times(\tilde t, t^{(0)}],\end{equation}
where $t^{(0)}=\tilde t+\gamma_9\nu_3\omega^{(1-p)(2-p)}R^p$.
\end{lemma}
\begin{proof}Let $t^{(0)}=\tilde t+\gamma_9\nu\omega^{(1-p)(2-p)}R^p$ where $\nu\in(0,1)$ is to be determined.
Consider two decreasing sequences of real numbers
\begin{equation*}R_n=c_2\frac{R}{4}+c_2\frac{R}{4^{n+1}}\quad \text{and}\quad k_n=\mu_+-\frac{\omega}{2^{m_2p}}-\frac{\omega}{2^{m_2p+n}}\qquad n=0,1,2,\cdots.\end{equation*}
We set $Q_n=K_{R_n}\times(\tilde t, t^{(0)}]$.
Take piecewise smooth cutoff
functions $\zeta_n(x)$ in $K_{R_n}$, such that $0\leq\zeta_n\leq 1$ in $K_{R_n}$, $\zeta_n\equiv 1$ in $K_{R_{n+1}}$
and $|D\zeta_n|\leq 4^{n+2}(c_2R)^{-1}$. Write down the energy estimates \eqref{Caccioppoli} for the truncated
functions $(u-k_n)_+$ over the cylinders $Q_n$.
Taking into account that
\begin{equation*}\partial_t\zeta_n\equiv 0\qquad\text{and}\qquad U(K_{R_n},\tilde t,t^{(0)},(u-k_n)_+\zeta_n^p)=0,\end{equation*}
we obtain
\begin{equation*}\begin{split}\esssup_{\tilde t<t<t^{(0)}}&\int_{K_{R_n}\times\{t\}}(u-k_n)_+^2\zeta_n^p dx
+\iint_{Q_n}|D(u-k_n)_+\zeta_n|^p dxdt
\\&\leq \gamma_1\iint_{Q_n}(u-k_n)_+^p|D\zeta_n|^pdxdt\leq  \frac{4^{n+3}\gamma_1
\omega^{p(3-p)}}{L_2^p2^{2m_2p}R^p}\iint_{Q_n}\chi_{[(u-k_n)_+>0]}dxdt.
\end{split}\end{equation*}
Set
\begin{equation*}A_n=\iint_{Q_n}\chi_{[(u-k_n)_+>0]}dxdt\qquad\text{and}\qquad Y_n=\frac{A_n}{|Q_n|}.\end{equation*}
Applying the parabolic Sobolev's inequality \eqref{Sobolevf}, we get
\begin{equation}\begin{split}\label{FFF1}\iint_{Q_n}&|(u-k_n)_+\zeta_n|^{p\frac{N+2}{N}}dxdt
\\&\leq \gamma \left(\esssup_{\tilde t<t<t^{(0)}}\int_{K_{R_n}\times\{t\}}(u-k_n)_+^2\zeta_n^2 dx\right)^{\frac{p}{N}}
\iint_{Q_n}|D(u-k_n)_+\zeta_n|^p dxdt
\\&\leq \gamma_{10} 4^{2n} \left(\omega^{(3-p)p}2^{-2pm_2}\right)^{1+\frac{p}{N}}R^{-p(1+\frac{p}{N})}A_n^{1+\frac{p}{N}},\end{split}\end{equation}
for a constant $\gamma_{10}$ depending only upon the data. The integral on the left-hand side is
estimated below by
\begin{equation}\begin{split}\label{FFF2}\iint_{Q_n}&|(u-k_n)_+\zeta_n|^{p\frac{N+2}{N}}dxdt
\geq (k_{n+1}-k_n)^{p\frac{N+2}{N}}A_{n+1}
\geq \left(\frac{\omega}{2^{m_2p+n+1}}\right)^{p\frac{N+2}{N}}A_{n+1}
\end{split}\end{equation}
Combining \eqref{FFF1} and \eqref{FFF2}, we have
\begin{equation*}\begin{split}A_{n+1}\leq \gamma_{10}2^{\frac{2(N+2)}{N}}
4^{6n}\omega^{-\frac{p^3}{N}+(\frac{3}{N}-1)p^2+2(1-\frac{1}{N})p}2^{m_2p(p-2)}R^{-p(1+\frac{p}{N})}A_n^{1+\frac{p}{N}}.\end{split}\end{equation*}
Taking into account that
$$|Q_n|=\gamma_9\nu L_2^N\left(\frac{\omega}{2^{m_2}}\right)^{N(p-2)}\left(\frac{1}{4}+\frac{1}{4^{n+1}}\right)^N
\omega^{(1-p)(2-p)}R^{N+p}$$
and $\nu<1$,
we obtain
\begin{equation*}Y_{n+1}\leq \gamma_{11}4^{6n}\nu^{\frac{p}{N}}Y_n^{1+\frac{p}{N}}\leq \gamma_{11}4^{6n}Y_n^{1+\frac{2}{N}},\end{equation*}
for a constant $\gamma_{11}$ depending only upon the data. At this point, we set
\begin{equation}\label{nu3}\nu_3=\gamma_{11}^{-\frac{N}{2}}4^{-\frac{3N^2}{2}}.\end{equation}
We now choose $\nu=\nu_3$. By Lemma \ref{3rd expand time 6.6},
$Y_0\leq \nu_3$. Applying the lemma on fast geometric convergence of
sequences, we deduce $Y_n\to 0$ as $n\to\infty$, which proves the lemma.
\end{proof}

\subsection{Iterative arguments: time propagation of positivity from $t^{(0)}$ to $t^{(1)}$}\label{t0t1 1st iter}
In this subsection we first remark that the time level $t^{(0)}$ could be lower than $\bar t$. So the estimate \eqref{DeGiorgi5}
is insufficient for the proof. We have to use an iterative argument to obtain the estimate similar to \eqref{DeGiorgi5}
in the cylinder with a larger time interval.

The starting point is a space propagation of positivity similar to Lemma \ref{dimension less lemma}.
Starting from \eqref{DeGiorgi5}, we set $R^{(0)}=\frac{1}{4}c_2R$ and introduce the change of variables
\begin{equation*}\begin{split}x'=\frac{x}{20R^{(0)}}\qquad\text{and}\qquad t'=\frac{t-t^{(0)}}{t^{(0)}-\tilde t}.\end{split}\end{equation*}
This transformation maps $K_{R^{(0)}}\times (\tilde t,t^{(0)}]\to K_{\frac{1}{20}}\times (-1,0]$ and
$K_{20R^{(0)}}\times (\tilde t,t^{(0)}]\to K_1\times (-1,0]$.
Moreover, we set the new functions
\begin{equation*}\begin{split}\tilde u(x',t')=u(x,t)\qquad\text{and}\qquad\bar u(x',t')=\left(\tilde u(x',t')-\mu_+\right)
\left(\frac{2^{m_2p}}{\omega}\right). \end{split}\end{equation*}
With these notations, the estimate \eqref{DeGiorgi5} implies
\begin{equation}\label{K11}\big|\{x'\in K_1:\bar u(x',t')<-1\}\big|\geq |K_{\frac{1}{20}}|=20^{-N}\end{equation}
for all $t'\in(-1,0]$.
On the other hand, we set $\tilde w(x',t')=w(x,t)$ and $v(x,t)$ in \eqref{frequently use weak form two phase stefan} can be written in the new variable as
\begin{equation*}
	\tilde v(x',t')=\begin{cases}
	\bar\nu,&\quad \text{on}\quad\{\bar u< -\mu_+\frac{2^{m_2p}}{\omega}\},\\
	-\tilde w(x',t'),&\quad \text{on}\quad\{\bar u=-\mu_+\frac{2^{m_2p}}{\omega}\}.
	\end{cases}
\end{equation*}
We rewrite the weak form \eqref{frequently use weak form two phase stefan} in terms of the new variables and new functions
\begin{equation}\begin{split}\label{dimensonal less weak form two phase stefan I1}-\frac{2^{m_2p}}{\omega}\int_{K_1} &\tilde v(\cdot,t')\chi_{\left\{\bar u\leq -\mu_+\frac{2^{m_2p}}{\omega}\right\}}\varphi(\cdot,t')dx'\bigg|_{t'=t_1}^{t_2}+
\frac{2^{m_2p}}{\omega}\int_{t_1}^{t_2}\int_{K_1} \tilde v\chi_{\left\{\bar u\leq -\mu_+\frac{2^{m_2p}}{\omega}\right\}}\frac{\partial \varphi}{\partial t'}dx'dt'
\\&+\int_{t_1}^{t_2}\int_{K_1}
\left(\varphi\frac{\partial \bar u}{\partial t'}+\bar A(x',t',\bar u,D\bar u)\cdot D\varphi\right)dx'dt'=0 \end{split}\end{equation}
for any $\varphi\in W_p(Q_1)$ and $[t_1,t_2]\subset (-1,0]$. We observe that the vector field $\bar A$ satisfies the structure condition
\begin{equation}\label{bar A I1}
	\begin{cases}
	\bar A(x',t',\bar u,D \bar u)\cdot D \bar u\geq C_0\Lambda_1|D \bar u|^{p},\\
	|\bar A(x',t',\bar u,D \bar u)|\leq C_1\Lambda_1|D \bar u|^{p-1},
	\end{cases}
\end{equation}
where
\begin{equation*}\Lambda_1=\left(\frac{\omega}{2^{m_2p}}\right)^{p-2}\frac{t^{(0)}-\tilde t}{(20R^{(0)})^p}.\end{equation*}
Recalling that
$t^{(0)}=\tilde t+\gamma_9\nu_3\omega^{(1-p)(2-p)}R^p$ and
$R^{(0)}=\frac{1}{4}c_2R$, we deduce
\begin{equation}\label{Lambda1}\Lambda_1=\frac{\gamma_9\nu_3}{(5L_2)^p}.\end{equation}
The constant $\Lambda_1$ depends only upon the data.
To simplify notation, we continue to write $(x,t)$ for the new
variables. In the same fashion as in the proof of Lemma \ref{subsolution lemma}, we conclude that
the truncated functions $(\bar u-k)_+$ are subsolutions to parabolic equations. In a precise way we have
\begin{equation}\begin{split}\label{subsolution I1}\int_{t_1}^{t_2}\int_{K_1}&
\varphi \frac{\partial (\bar u-k)_+}{\partial t}+\bar A(x,t,k+(\bar u-k)_+,D(\bar u-k)_+)\cdot D\varphi dxdt
\leq 0\end{split}\end{equation}
for all $k\in (-1,0)$, $t_1,t_2\in (-1,0)$ and all nonnegative $\varphi\in W_p(Q_1)$.
\begin{lemma}\label{dimension less lemma I1}Suppose that $\frac{3}{2}\leq p<2$. For any $\nu\in(0,1)$,
there exists $\bar m>1$ depending
only upon the data and $\nu$, such that
\begin{equation}\label{pro6.8formula}\big|\{x\in K_{\frac{3}{4}}:\bar u(x,t)\geq -2^{-\bar m}\}\big|\leq\nu|K_{\frac{3}{4}}|\end{equation}
for all $t\in \left(-\left(\frac{3}{4}\right)^p,0\right]$.
\end{lemma}
\begin{proof}We proceed similarly as in the proof of Lemma \ref{dimension less lemma}. To this end,
we introduce the auxiliary functions
\begin{equation*}\phi_k(\bar u)=\int_0^{(\bar u-k)_+}\frac{d\tau}{(-(1-\delta)k-\tau+V)^{p-1}}\end{equation*}
and
\begin{equation*}\psi_k(\bar u)=\ln\left(\frac{-(1-\delta)k+V}{-k(1-\delta)-(\bar u-k)_++V}\right)\end{equation*}
where $k,\ \delta\in (-\frac{1}{8},0)$ and $V= (-k)(-\delta)^{\frac{p}{p-1}}$.
Take a piecewise smooth, cutoff function $\zeta=\zeta_1(x)\zeta_2(t)$ in $Q_1$, such that $0\leq\zeta\leq1$ in $Q_1$,
$\zeta\equiv 1$ in $Q_{\frac{3}{4}}$, $\zeta\equiv 0$ on $\partial_PQ_1$, $|D\zeta_1|\leq 4$, $0\leq\partial_t\zeta_2\leq 4^p
$ and the sets $\{x\in K_1:\zeta_1(x)>-k\}$ are convex for all $k\in (-\frac{1}{8},0)$. In the weak formulation \eqref{subsolution I1}
take the test function
\begin{equation*}\varphi=\frac{\zeta^p}{(-(1-\delta)k-(\bar u-k)_++V)^{p-1}}.\end{equation*}
This gives
\begin{equation*}\begin{split}\int_{K_1}&
\zeta^p \phi_k(\bar u)dx\bigg|_{t=t_1}^{t_2}+\frac{1}{4}C_0\Lambda_1\int_{t_1}^{t_2}\int_{K_1}\zeta^p |D\psi_k(\bar u)|^pdxdt
\leq 16C_2(t_2-t_1)\ln\frac{1-\delta}{-\delta}\end{split}\end{equation*}
for any $[t_1,t_2]\in(-1,0]$.
Next, we observe from \eqref{K11} that
\begin{equation*}K_{\frac{1}{20}}\subset \{x\in K_1:\bar u(x,t)<-1\}\subset\{x\in K_1:\phi_k(\bar u)=0\}\cap\{x\in K_1:
\zeta_1(x)=1\}, \end{equation*}
for any $t\in(-1,0]$. Applying Sobolev's inequality \eqref{Sobolevf1} slicewise, we obtain
\begin{equation*}\int_{K_1}\zeta_1^p \psi_k(\bar u)^pdx\leq 20^{N-1}\gamma\int_{K_1}\zeta_1^p |D\psi_k(\bar u)|^pdx,\quad
\forall t\in(-1,0].\end{equation*}
With the same argument as in the proof of Lemma \ref{dimension less lemma}, we derive the estimate
\begin{equation*}\begin{split}\frac{d^-}{dt}\int_{K_1}&
\zeta^p \phi_k(\bar u)dx+ \Lambda_2\int_{K_1}\zeta^p \psi_k(\bar u)^pdx
\leq 16C_2
\ln\frac{1-\delta}{-\delta},\quad
\forall t\in(-1,0],\end{split}\end{equation*}
where the constant $\Lambda_2$ depends only upon the data. At this stage, we follow the
proof of Lemma \ref{dimension less lemma}
to conclude that there exists a constant $\gamma'$ depending only on the data, such that the following holds. If we choose
\begin{equation}\label{bar m I1}
\bar m=\bar m(\nu)=\exp\left\{\left(\frac{\gamma'}{\nu\Lambda_2}\right)^2\right\},\end{equation}
then the estimate \eqref{pro6.8formula} follows. Moreover, we observe that the constant $\bar m$ depends only upon the data.
\end{proof}
\begin{lemma}\label{DeGiorgi bar u I1}
There exist a constant $\bar m>1$ depending only upon the data and $p$, and a time level $t_\omega'\in\left(-\left(\frac{3}{4}\right)^p,0\right)
$ such that
\begin{equation}\label{space expanding bar u I1}\bar u(x,t)<-\frac{1}{2^{\bar m+1}}\qquad\text{a.e.}\quad\text{in}\quad
Q\left(-\frac{t_\omega'}{2^p},\frac{1}{2}\right),\end{equation}
and the constant $\bar m$ is stable as $p\to2$.
\end{lemma}
\begin{proof} For $\nu\in(0,1)$ to be determined later we take $\bar m>1$ according to \eqref{bar m I1}.
For $n=0,1,2,\cdots$, set
\begin{equation*}R_n=\frac{1}{2}+\frac{1}{2^{n+2}},\quad k_n=-\frac{1}{2^{\bar m+1}}
-\frac{1}{2^{\bar m+n+1}}\quad\text{and}\quad
Q_n= Q\left(\frac{-t_\omega'}{2^{p}}+\frac{-t_\omega'}{2^{p+pn}},R_n\right),\end{equation*}
where $t_\omega'=-2^{(\bar m+1)(p-2)}\left(\frac{3}{4}\right)^p$.
Take piecewise smooth cutoff functions $\zeta_n$ in $Q_n$, such that $0\leq\zeta_n\leq1$, $\zeta_n\equiv1$ in $Q_{n+1}$, $\zeta_n\equiv0$ on $\partial_PQ_n$,
$|D\zeta_n|\leq 2^{n+3}$ and $0<\partial_t\zeta_n\leq 2^{p+p(n+1)}(-t_\omega')^{-1}$.

In the weak formulation \eqref{dimensonal less weak form two phase stefan I1} take the
test function $\varphi=(\bar u-k_n)_+\zeta_n^p$ . Observe that the first
two terms vanish. By a standard argument, we derive a Caccioppoli's estimate for $(\bar u-k_n)_+$ over $Q_n$ as follows
\begin{equation*}\begin{split}\esssup_{2^{-p}t_\omega'+2^{-p-pn}t_\omega'<t<0}&\frac{1}{2}\int_{K_{R_n}\times\{t\}}(
\bar u-k_n)_{+}^2\zeta_n^p dx
+\frac{1}{2}C_0\Lambda_1\iint_{Q_n}|D(\bar u-k_n)_{+}|^p\zeta_n^p dxdt\\
&\leq 8C_0\left(\frac{C_1}{C_0}\right)^2\Lambda_1\iint_{Q_n}(\bar u-k_n)_+^p|D\zeta_n|^pdxdt+\frac{1}{2}
\iint_{Q_n}(\bar u-k_n)_{+}^2\frac{\partial\zeta_n^p}{\partial t}dxdt.
\end{split}\end{equation*}
Applying parabolic Sobolev's inequality \eqref{Sobolevf}, we obtain
\begin{equation}\begin{split}\label{SSSSS}
\iint_{Q_n}&|(\bar u-k_n)_{+}|^{p(1+\frac{2}{N})}\zeta_n^{p(1+\frac{2}{N})} dxdt\\&\leq \gamma
\left(\esssup_{2^{-p}t_\omega'+2^{-p-pn}t_\omega'<t<0}\int_{K_{R_n}\times\{t\}}(\bar u-k_n)_{+}^2\zeta_n^{2} dx\right)^{\frac{p}{N}}
\iint_{Q_n}|D(\bar u-k_n)_{+}\zeta_n|^p dxdt\\&
\leq C_32^{pn(1+\frac{p}{N})}\left(\frac{1}{2^{\tilde m}}\right)^{p+\frac{p^2}{N}}
\left(\iint_{Q_n}\chi_{[(\bar u-k_n)_+>0]}dxdt\right)^{1+\frac{p}{N}},\end{split}\end{equation}
for a constant $C_3$ depending only upon the data.
At this point, we set
\begin{equation*}A_n=\iint_{Q_n}\chi_{[(\bar u-k_n)_+>0]}dxdt\qquad\text{and}\qquad Y_n=\frac{A_n}{|Q_n|}.\end{equation*}
The left-hand side of \eqref{SSSSS} is estimated below by
\begin{equation*}\begin{split}\iint_{Q_n}&|(\bar u-k_n)_{+}|^{p(1+\frac{2}{N})}\zeta_n^{p(1+\frac{2}{N})} dxdt
\geq (k_{n+1}-k_n)^{p\frac{N+2}{N}}A_{n+1}
\geq \left(\frac{1}{2^{\bar m+n+2}}\right)^{p\frac{N+2}{N}}A_{n+1} .\end{split}\end{equation*}
Combining this estimate with \eqref{SSSSS}, we obtain
\begin{equation*}Y_{n+1}\leq C_44^{pn(1+\frac{2}{N})}Y_n^{1+\frac{p}{N}},\end{equation*}
for a constant $C_4$ depending only upon the data. At this point, we set
\begin{equation}\label{nu4}\nu_4=4^{-\frac{N(N+2)}{p}}C_4^{-\frac{N}{p}}.\end{equation}
Choose $\nu=\nu_4$, and hence $\bar m=\bar m(\nu_4)$, from this and \eqref{bar m I1}. From Lemma \ref{dimension less lemma I1},
$Y_0\leq\nu_4$. By the lemma on fast geometric convergence of sequences, we infer that $Y_n\to0$, as $n\to\infty$. This proves the lemma.
\end{proof}
Transforming back to the original function $u$ and original variables $(x,t)$, we obtain the following DeGiorgi-type result.
\begin{lemma}\label{DeGiorgiI1}Let $\bar m>1$ be the constant as in
 Lemma \ref{DeGiorgi bar u I1}. Then there exists a time level $t_\omega^{(0)}\in (\tilde t, t^{(0)})$ such that
\begin{equation}\label{space expanding I1}u(x,t)<\mu_+-\frac{\omega}{2^{m_2p+\bar m+1}}\qquad\text{a.e.}\quad\text{in}\quad
(0,t^{(0)})+ Q\left( \frac{t^{(0)}-t_\omega^{(0)}}{2^p},10R^{(0)}\right),\end{equation}
where
$t_\omega^{(0)}=t^{(0)}+(t^{(0)}-\tilde t)t_\omega'.$
\end{lemma}
In order to obtain an expansion of positivity result at a higher time level $t^{(1)}$, we need the following lemma
that is similar to Lemma \ref{3rd expand time 6.6}.
\begin{lemma}\label{3rd expand time I1}For any $\nu\in(0,1)$, there exist a constant $\delta_1$ depending only upon the data and $p$
, and a time level
\begin{equation}\label{t1}t^{(1)}=t^{(0)}+\delta_1\nu\omega^{(1-p)(2-p)}R^p\end{equation}
such that
\begin{equation}\label{3rd expand time I1f}\big|\{x\in K_{6R^{(0)}}:u(x,t)>\mu_+-\frac{\omega}{2^{m_2p+\bar m+4}}\}\big|
<\nu |K_{6R^{(0)}}|\end{equation}
for any $t\in (t^{(0)}, t^{(1)})$, and the constant $\delta_1$ is stable as $p\to2$.
\end{lemma}
\begin{proof}Let $k=\mu_+-\frac{\omega}{2^{m_2p+\bar m+1}}$ and $c=\frac{\omega}{2^{m_2p+\bar m+4}}$.
We consider the logarithmic function
\begin{equation*}\begin{split}\psi^+=\ln^+\left(\frac{\frac{\omega}{2^{m_2p+\bar m+1}}}{\frac{\omega}{2^{m_2p+\bar m+1}}-(u-k)_++c}\right).
\end{split}\end{equation*}
Then we get $\psi^+\leq 8\ln2$ and
\begin{equation*}\begin{split}[(\psi^+)']^{2-p}\leq \left(\frac{1}{c}\right)^{2-p}
\leq \left(\frac{\omega}{2^{m_2p}}\right)^{p-2}2^{(\bar m+4)(2-p)}.\end{split}\end{equation*}
Choose a piecewise smooth cutoff function $\zeta(x)$, defined in $K_{10R^{(0)}}$, and satisfying $0\leq \zeta\leq1$ in $K_{10R^{(0)}}$,
$\zeta\equiv 1$ in $K_{6R^{(0)}}$ and $|D\zeta|\leq (4R^{(0)})^{-1}$. Such kind of cutoff functions can be chosen explicitly via the formulas
\eqref{zeta}-\eqref{zetai}.
Since $k>\mu_+-\frac{\omega}{4}$, we check at once that $U(K_{10R^{(0)}},t^{(0)},t^{(1)},2\psi^{+}\left(\psi^{+}\right)^\prime\zeta^p)=0$.
Then we obtain from \eqref{logarithmic} the logarithmic estimate
\begin{equation*}\begin{split}\esssup_{t^{(0)}<t<t^{(1)}}&\int_{K_{10R^{(0)}}\times\{t\}}\left(\psi^{+}\right)^2\zeta^p dx\leq
\int_{K_{10R^{(0)}}\times\{t^{(0)}\}}\left(\psi^{+}\right)^2\zeta^p dx+\gamma_2\int_{t^{(0)}}^{t^{(1)}}\int_{
K_{10R^{(0)}}}\psi^{+}|\left(\psi^{+}\right)^\prime|^{2-p}
|D\zeta|^pdxdt.\end{split}\end{equation*}
By Lemma \ref{DeGiorgiI1}, we see that
$\int_{K_{10R^{(0)}}\times\{t^{(0)}\}}\left(\psi^{+}\right)^2\zeta^p dx=0$
and therefore
\begin{equation*}\begin{split}\esssup_{t^{(0)}<t<t^{(1)}}\int_{K_{10R^{(0)}}\times\{t\}}\left(\psi^{+}\right)^2\zeta^p dx&\leq
8\gamma_2\left(\frac{\omega}{2^{m_2p}}\right)^{p-2}2^{(\bar m+4)(2-p)}(4R^{(0)})^{-p}(t^{(1)}-t^{(0)})
(\ln2)|K_{10R^{(0)}}|
\\&=8\frac{10^N\ln2}{6^N}L_2^{-p}\gamma_2\omega^{(1-p)(p-2)}2^{(\bar m+4)(2-p)}R^{-p}(t^{(1)}-t^{(0)})|K_{6R^{(0)}}|.\end{split}\end{equation*}
At this stage, we choose the time level
\begin{equation}\label{t1delta1}t^{(1)}=t^{(0)}+\delta_1\nu\omega^{(1-p)(2-p)}R^p\quad\text{where}\quad \delta_1=\frac{1}{8}\frac{6^NL_2^p}{10^N
\gamma_2 2^{(\bar m+4)(2-p)}
\ln2}.
\end{equation}
Combining the estimates above we infer that
\begin{equation*}\begin{split}\esssup_{t^{(0)}<t<t^{(1)}}&\int_{K_{10R^{(0)}}\times\{t\}}\left(\psi^{+}\right)^2\zeta^p dx\leq
\nu|K_{6R^{(0)}}|.\end{split}\end{equation*}
On the other hand, we introduce the smaller set
$$S=\{x\in K_{6R^{(0)}}:u(x,t)>\mu_+-\frac{\omega}{2^{m_2p+\bar m+4}}\}\subset K_{6R^{(0)}}.$$
In this set, $\zeta\equiv 1$ and $\psi^+\geq 2\ln2>1$. Then we conclude that
\begin{equation*}\begin{split}\int_{K_{10R^{(0)}}\times\{t\}}\left(\psi^{+}\right)^2\zeta^p dx\geq
\int_{S\times\{t\}}\left(\psi^{+}\right)^2\zeta^p dx\geq
\big|\{x\in K_{6R^{(0)}}:u(x,t)>\mu_+-\frac{\omega}{2^{m_2p+\bar m+4}}\}\big|
\end{split}\end{equation*}
for all $t\in (t^{(0)},t^{(1)})$ and the lemma follows with the choice of $\delta_1$ in \eqref{t1delta1}.
\end{proof}
With the help of Lemma \ref{3rd expand time I1}, we establish a DeGiorgi-type lemma and determine the value of $\nu$
and the time level $t^{(1)}$ in \eqref{t1}.
This result extends \eqref{DeGiorgi5}
to a larger time interval.
\begin{lemma}\label{DeGiorgi I1I1}
Let $\delta_1>0$ be the constant chosen according to \eqref{t1delta1}. Then there exists a
constant $\nu_5$ depending only upon the data and $p$ such that the following holds.
For the choice
$$t^{(1)}=t^{(0)}+\delta_1\nu_5\omega^{(1-p)(2-p)}R^p,$$
there holds
\begin{equation}\label{DeGiorgiI17}u(x,t)<\mu_+-\frac{\omega}{2^{m_2p+\bar m+5}}\qquad\text{a.e.}\quad\text{in}\quad K_{2R^{(0)}}\times(t^{(0)},t^{(1)}]\end{equation}
and the constant $\nu_5$ is stable as $p\to2$.
\end{lemma}
\begin{proof}For $\nu\in(0,1)$ to be determined later we take $\bar m=\bar m(\nu)>1$ according to \eqref{bar m I1}.
For $n=0,1,2,\cdots$, set
\begin{equation*}R_n=2R^{(0)}+\frac{4R^{(0)}}{2^{n}},\quad k_n=\mu_+-\frac{\omega}{2^{m_2p+\bar m+5}}-\frac{\omega}{2^{m_2p+\bar m+5+n}}\quad \text{and}\quad Q_n=K_{R_n}\times(t^{(0)},t^{(1)}].\end{equation*}
Let $\zeta_n=\zeta_n(x)$ be defined via \eqref{zeta}-\eqref{zetai} with $r$ and $r'$ replaced by $R_{n+1}$ and $R_n$.
Then the cutoff function $\zeta_n$ satisfies $0\leq\zeta_n\leq 1$ in $K_{R_n}$, $\zeta_n\equiv 1$ in $K_{R_{n+1}}$
and $|D\zeta_n|\leq 2^n(2R^{(0)})^{-1}$. Taking into account that $U(K_{R_n},t^{(0)},t^{(1)},(u-k_n)_+\zeta_n^p)=0$,
we proceed similarly as in the proof of Proposition \ref{DeGiorgi initial}. Write down the energy
estimates \eqref{Caccioppoli} for $(u-k_n)_+$ over $Q_n$, we obtain
\begin{equation*}\begin{split}\esssup_{t^{(0)}<t<t^{(1)}}\int_{K_{R_n}\times\{t\}}(u-k_n)_+^2\zeta_n^p dx
+\iint_{Q_n}|D(u-k_n)_+\zeta_n|^p dxdt
\leq \gamma_1\iint_{Q_n}(u-k_n)_+^p|D\zeta_n|^pdxdt.
\end{split}\end{equation*}
Set $A_n=\iint_{Q_n}\chi_{[(u-k_n)_+>0]}dxdt$.
Applying the parabolic Sobolev's inequality \eqref{Sobolevf}, we deduce
\begin{equation*}\begin{split}\iint_{Q_n}|(u-k_n)_+\zeta_n|^{p\frac{N+2}{N}}dxdt&\leq \gamma
\iint_{Q_n}|D(u-k_n)_+\zeta_n|^p dxdt\left(\esssup_{t^{(0)}<t<t^{(1)}}\int_{K_{R_n}\times\{t\}}(u-k_n)_+^2\zeta_n^p dx\right)^{\frac{p}{N}}
\\&\leq \gamma\gamma_1^{1+\frac{p}{N}} \frac{2^{np(1+\frac{p}{N})}}{L_2^{p(1+\frac{p}{N})}} \left(\frac{\omega^{(3-p)p}}{2^{2pm_2}}\right)^{1+\frac{p}{N}}\left(\frac{1}{2^{\bar m+3}}\right)^{p(1+\frac{p}{N})}R^{-p(1+\frac{p}{N})}A_n^{1+\frac{p}{N}}.\end{split}\end{equation*}
The integral on the left-hand side is
estimated below by
\begin{equation*}\begin{split}\iint_{Q_n}|(u-k_n)_+\zeta_n|^{p\frac{N+2}{N}}dxdt\geq (k_{n+1}-k_n)^{p\frac{N+2}{N}}A_{n+1}
=\left(\frac{\omega}{2^{m_2p+\bar m+n+6}}\right)^{p\frac{N+2}{N}}A_{n+1}.\end{split}\end{equation*}
Combining the estimates above we see that
\begin{equation*}\begin{split}A_{n+1}\leq \gamma\gamma_1^{1+\frac{p}{N}}
2^{n(2p+\frac{2p}{N}+\frac{p^2}{N})}\omega^{-\frac{p^3}{N}+(\frac{3}{N}-1)p^2+2(1-\frac{1}{N})p}2^{m_2p(p-2)}R^{-p(1+\frac{p}{N})}
\frac{2^{p(1+\frac{p}{N})}4^{p\frac{N+2}{N}}}{L_2^{p(1+\frac{p}{N})}}\left(\frac{1}{2^{\bar m+4}}\right)^{p(\frac{p-2}{N})}A_n^{1+\frac{p}{N}}.\end{split}\end{equation*}
Set $Y_n=A_n/|Q_n|$ and we shall derive a recursive inequality. Keeping in mind that
\begin{equation*}\begin{split}|Q_n|=\left(2R^{(0)}+\frac{4R^{(0)}}{2^{n}}\right)^N(t^{(1)}-t^{(0)})=(2+2^{2-n})^N\frac{1}{4^N}
\left(\frac{\omega}{2^{m_2}}\right)^{(p-2)N}\delta_1\nu
\omega^{(1-p)(2-p)}R^{N+p},\end{split}\end{equation*}
we deduce
\begin{equation*}\begin{split}\frac{1}{2^N}
\left(\frac{\omega}{2^{m_2}}\right)^{(p-2)N}\delta_1\nu
\omega^{(1-p)(2-p)}R^{N+p}\leq|Q_n|\leq \frac{3^N}{2^N}
\left(\frac{\omega}{2^{m_2}}\right)^{(p-2)N}\delta_1\nu
\omega^{(1-p)(2-p)}R^{N+p}.\end{split}\end{equation*}
Moreover, this estimate implies
\begin{equation*}\frac{|Q_{n+1}|}{|Q_n|}\geq\frac{1}{3^N}.\end{equation*}
Since $0<\nu<1$, then we conclude that
\begin{equation*}Y_{n+1}\leq \frac{3^{N+p}}{2^p}\gamma\gamma_1^{1+\frac{p}{N}}(\delta_1\nu)^{\frac{p}{N}}2^{n(2p+\frac{2p}{N}+\frac{p^2}{N})}
\frac{2^{p(1+\frac{p}{N})}4^{p\frac{N+2}{N}}}{L_2^{p(1+\frac{p}{N})}}\left(\frac{1}{2^{\bar m+4}}\right)^{p(\frac{p-2}{N})}
Y_n^{1+\frac{p}{N}}= \delta_2b^{n}Y_n^{1+\frac{p}{N}},\end{equation*}
where
\begin{equation}\label{delta2b}b=2^{2p+\frac{2p}{N}+\frac{p^2}{N}}\quad\text{and}\quad
\delta_2:=\frac{3^{p+N}}{2^p}\gamma\gamma_1^{1+\frac{p}{N}}\delta_1^{\frac{p}{N}}
\frac{2^{p(1+\frac{p}{N})}4^{p\frac{N+2}{N}}}{L_2^{p(1+\frac{p}{N})}}\left(\frac{1}{2^{\bar m+4}}\right)^{p(\frac{p-2}{N})}.\end{equation}
At this point, we set $\nu_5=\delta_2^{-\frac{N}{p}}b^{-\frac{N^2}{p^2}}$. Recalling the definition of $\delta_1$ from \eqref{t1delta1}, we see that
\begin{equation*}\delta_2=\frac{3^{p+N}}{2^p}\gamma\gamma_1^{1+\frac{p}{N}}\left(\frac{6^N}{8^{3-p}10^N\gamma_1
\ln2}\right)^{\frac{p}{N}}
\frac{2^{p(1+\frac{p}{N})}4^{p\frac{N+2}{N}}}{L_2^{p}}\left(\frac{1}{2^{3}}\right)^{p(\frac{p-2}{N})},\end{equation*}
which is even independent of $\bar m$.
Therefore, we conclude that if
$Y_0\leq \delta_2^{-\frac{N}{p}}b^{-\frac{N^2}{p^2}}$
then $Y_n\to 0$ as $n\to\infty$. To this end, we choose $\nu=\nu_5$ and the lemma follows.
\end{proof}
\subsection{Iterative arguments: time propagation of positivity from $t^{(1)}$ to $t^{(2)}$}\label{t1t2}
Starting from \eqref{DeGiorgiI17}, we will repeat the argument of \S\ref{t0t1 1st iter} to obtain an estimate
similar to \eqref{DeGiorgiI17} in a space-time cylinder containing a higher
time level $t^{(2)}$. The argument is divided into four steps.
To start with, we set $R^{(1)}=2R^{(0)}$.

{\emph Step 1:}
We introduce the change of variables
\begin{equation*}\begin{split}x'=\frac{x}{20R^{(1)}}\qquad\text{and}\qquad t'=\frac{t-t^{(1)}}{t^{(1)}-t^{(0)}}.\end{split}\end{equation*}
This transformation maps $K_{R^{(1)}}\times (t^{(0)},t^{(1)})\to K_{\frac{1}{20}}\times (-1,0)$ and
$K_{20R^{(1)}}\times (t^{(0)},t^{(1)})\to Q_1$.
We introduce the new functions
\begin{equation*}\begin{split}\tilde u(x',t')=u(x,t)\qquad\text{and}\qquad\bar u(x',t')=\left(\tilde u(x',t')-\mu_+\right)
\left(\frac{2^{m_2p+\bar m+5}}{\omega}\right). \end{split}\end{equation*}
From \eqref{DeGiorgiI17}, we see that
$\big|\{x'\in K_1:\bar u(x',t')<-1\}\big|\geq 20^{-N}$
for all $t'\in(-1,0)$.
Set $\tilde w(x',t')=w(x,t)$ and $v(x,t)$ in \eqref{frequently use weak form two phase stefan} can be written in the new variable as
\begin{equation*}
	\tilde v(x',t')=\begin{cases}
	\bar\nu,&\quad \text{on}\quad\{\bar u< -\mu_+\frac{2^{m_2p+\bar m+5}}{\omega}\},\\
	-\tilde w(x',t'),&\quad \text{on}\quad\{\bar u=-\mu_+\frac{2^{m_2p+\bar m+5}}{\omega}\}.
	\end{cases}
\end{equation*}
We rewrite the weak form \eqref{frequently use weak form two phase stefan} in terms of the new variables and new functions
\begin{equation}\begin{split}\label{dimensonal less weak form two phase stefan I2}-\frac{2^{m_2p+\bar m+5}}{\omega}&\int_{K_1} \tilde v(\cdot,t')\chi_{\left\{\bar u\leq -\mu_+\frac{2^{m_2p+\bar m+5}}{\omega}\right\}}\varphi(\cdot,t')dx'\bigg|_{t'=t_1}^{t_2}\\&+
\frac{2^{m_2p+\bar m+5}}{\omega}\int_{t_1}^{t_2}\int_{K_1} \tilde v\chi_{\left\{\bar u\leq -\mu_+\frac{2^{m_2p+\bar m+5}}{\omega}\right\}}\frac{\partial \varphi}{\partial t'}dx'dt'
\\&\quad+\int_{t_1}^{t_2}\int_{K_1}
\left(\varphi\frac{\partial \bar u}{\partial t'}+\bar A(x',t',\bar u,D\bar u)\cdot D\varphi\right)dx'dt'=0 \end{split}\end{equation}
for any $\varphi\in W_p(Q_1)$ and $[t_1,t_2]\subset (-1,0]$. It is easy to check that
 the vector field $\bar A$ satisfies the structure condition
\begin{equation}\label{bar A I2}
	\begin{cases}
	\bar A(x',t',\bar u,D \bar u)\cdot D \bar u\geq C_0\Lambda_2|D \bar u|^{p},\\
	|\bar A(x',t',\bar u,D \bar u)|\leq C_1\Lambda_2|D \bar u|^{p-1},
	\end{cases}
\end{equation}
where
\begin{equation*}\Lambda_2=\left(\frac{\omega}{2^{m_2p+\bar m+5}}\right)^{p-2}\frac{t^{(1)}-t^{(0)}}{(20R^{(1)})^p}.\end{equation*}
According to \eqref{t1delta1}-\eqref{DeGiorgiI17}, we deduce that
\begin{equation}\label{Lambda2}\Lambda_2=\frac{2^{5(2-p)}\delta_1\nu_5}{10^p2^{\bar m(p-2)}L_2^p}=\frac{2^{-1-p}6^N\nu_5}{10^{N+p}\gamma_2
\ln2},\end{equation}
which depends only upon the data and $p$, and is stable as $p\to2$. For any $\nu\in(0,1)$, we proceed similarly as in the proof of Lemma
\ref{dimension less lemma I1} to conclude that there exists a constant $\gamma''$, depending only upon the data,
such that the following holds. For the constant $m'$ with the expression
\begin{equation}\begin{split}\label{mprime} m'=\exp\left\{\left(\frac{\gamma''}{\nu\Lambda_2}\right)^2\right\},
\end{split}\end{equation}
there holds
\begin{equation}\label{step1}\big|\{x'\in K_{\frac{3}{4}}: \bar u(x',t')>-2^{-m'}\}\big|\leq \nu|K_{\frac{3}{4}}|\quad\text{for\ \ all}\quad t'\in
\left(-\left(\tfrac{3}{4}\right)^{p},0\right].\end{equation}

{\emph Step 2:} Based on the estimate \eqref{step1}, we are now in a position to obtain an upper bound for $u$ in a space-time cylinder
with a larger cube in space.
By the proof of Lemma \ref{DeGiorgi bar u I1} and Lemma \ref{DeGiorgiI1}, we conclude that
there exists a time level $t_\omega^{(1)}\in (t^{(0)}, t^{(1)})$, such that
\begin{equation}\label{space expanding I2}u(x,t)<\mu_+-\frac{\omega}{2^{m_2p+\bar m+5+m'
+1}}\qquad\text{a.e.}\quad\text{in}\quad(0,t^{(1)})+ Q\left( \frac{t^{(1)}-t_\omega^{(1)}}{2^p},10R^{(1)}\right),\end{equation}
where
\begin{equation}\begin{split}\label{prime m I2} m'=\exp\left\{\left(\frac{\gamma''}{\nu_4\Lambda_2}\right)^2\right\}
\end{split}\end{equation}
and the constant $\nu_4$ is chosen according to \eqref{nu4}.

{\emph Step 3:} Let $\delta_1>0$ be the constant chosen according to \eqref{t1delta1}. For any $\nu\in(0,1)$, we claim that
there exist a constant $M_1>0$ depending only upon the data and $p$, and a time level
\begin{equation*}t^{(2)}=t^{(1)}+\delta_1\nu\omega^{(1-p)(2-p)}M_1R^p\end{equation*}
such that
\begin{equation}\label{I2 expand time}\big|\{x\in K_{6R^{(1)}}:u(x,t)>\mu_+-\frac{\omega}{2^{m_2p+\bar m+5+ m'+4}}\}\big|
<\nu|K_{6R^{(1)}}|\end{equation}
for any $t\in (t^{(1)}, t^{(2)}]$.
\begin{proof}[Proof of the claim] Let $k=\mu_+-\frac{\omega}{2^{m_2p+\bar m+5+ m'+1}}$ and $c=\frac{\omega}{2^{m_2p+\bar m+5+ m'+4}}$.
We consider the logarithmic function
\begin{equation*}\begin{split}\psi^+=\ln^+\left(\frac{\frac{\omega}{2^{m_2p+\bar m+5+ m'+1}}}{\frac{\omega}{2^{m_2p+\bar m+5+ m'+1}}-(u-k)_++c}\right).
\end{split}\end{equation*}
Then we have $\psi^+\leq 8\ln2$ and
\begin{equation*}\begin{split}[(\psi^+)']^{2-p}\leq \left(\frac{1}{c}\right)^{2-p}
\leq \left(\frac{\omega}{2^{m_2p}}\right)^{p-2}2^{(\bar m+5+m'+4)(2-p)}.\end{split}\end{equation*}
Let $\zeta=\zeta(x)$ be defined via \eqref{zeta}-\eqref{zetai} with $r$ and $r'$ replaced by $6R^{(1)}$ and $10R^{(1)}$.
It follows that $0\leq \zeta\leq1$ in $K_{10R^{(1)}}$,
$\zeta\equiv 1$ in $K_{6R^{(1)}}$ and $|D\zeta|\leq (4R^{(1)})^{-1}$. From \eqref{space expanding I2}, we deduce
$\int_{K_{10R^{(1)}}\times\{t^{(1)}\}}\left(\psi^{+}\right)^2\zeta^p dx=0$.
Taking into account that $U(K_{10R^{(1)}},t^{(1)},t^{(2)},2\psi^{+}\left(\psi^{+}\right)^\prime\zeta^p)=0$,
we obtain from \eqref{logarithmic} the logarithmic estimate
\begin{equation*}\begin{split}\esssup_{t^{(1)}<t<t^{(2)}}&\int_{K_{10R^{(1)}}\times\{t\}}\left(\psi^{+}\right)^2\zeta^p dx\leq
\gamma_2\int_{t^{(1)}}^{t^{(2)}}\int_{
K_{10R^{(1)}}}\psi^{+}|\left(\psi^{+}\right)^\prime|^{2-p}
|D\zeta|^pdxdt.\end{split}\end{equation*}
Recalling the definition of $\delta_1$ from \eqref{t1delta1}, we deduce
\begin{equation*}\begin{split}\esssup_{t^{(1)}<t<t^{(2)}}\int_{K_{10R^{(1)}}\times\{t\}}\left(\psi^{+}\right)^2\zeta^p dx&\leq
8\gamma_2\left(\frac{\omega}{2^{m_2p}}\right)^{p-2}2^{(\bar m+5+m'+4)(2-p)}(4R^{(1)})^{-p}(t^{(2)}-t^{(1)})
(\ln2)|K_{10R^{(1)}}|
\\&=\nu M_12^{(m'+5)(2-p)}2^{-p}|K_{6R^{(1)}}|.\end{split}\end{equation*}
At this stage, we set
\begin{equation}\label{M1}M_1=\frac{2^p}{2^{(m'+5)(2-p)}}.\end{equation}
Then we find that
\begin{equation*}\begin{split}\esssup_{t^{(1)}<t<t^{(2)}}\int_{K_{10R^{(1)}}\times\{t\}}\left(\psi^{+}\right)^2\zeta^p dx\leq
\nu|K_{6R^{(1)}}|.\end{split}\end{equation*}
On the other hand, we introduce the smaller set
$$\{x\in K_{6R^{(1)}}:u(x,t)>\mu_+-\frac{\omega}{2^{m_2p+\bar m+5+m'+4}}\}\subset K_{6R^{(1)}}.$$
In this set, $\zeta\equiv 1$ and $\psi^+\geq 2\ln2>1$. Then we conclude that
\begin{equation*}\begin{split}\int_{K_{10R^{(1)}}\times\{t\}}\left(\psi^{+}\right)^2\zeta^p dx
\geq
\big|\{x\in K_{6R^{(1)}}:u(x,t)>\mu_+-\frac{\omega}{2^{m_2p+\bar m+5+m'+4}}\}\big|
\end{split}\end{equation*}
for all $t\in (t^{(1)},t^{(2)})$ and the estimate \eqref{I2 expand time} follows.
\end{proof}

{\emph Step 4:} Take $\nu_5=\delta_2^{-\frac{N}{p}}b^{-\frac{N^2}{p^2}}$ be the constant as in Lemma \ref{DeGiorgi I1I1}.
We claim that for the choice
\begin{equation}\label{t2}t^{(2)}=t^{(1)}+\delta_1\nu_5\omega^{(1-p)(2-p)}M_1R^p,\end{equation}
there holds
\begin{equation}\label{I2 DeGiorgi}u(x,t)<\mu_+-\frac{\omega}{2^{m_2p+\bar m+5+m'+5}}
\qquad\text{a.e.}\quad\text{in}\quad K_{2R^{(1)}}\times(t^{(1)},t^{(2)}].\end{equation}
\begin{proof}[Proof of the claim] Proceed similarly as in the proof of Lemma \ref{DeGiorgi I1I1}, we
consider two decreasing sequences of real numbers
\begin{equation*}R_n=2R^{(1)}+\frac{4R^{(1)}}{2^{n}}\quad \text{and}\quad k_n=\mu_+-\frac{\omega}{2^{m_2p+\bar m+5+ m'+5}}-\frac{\omega}{2^{m_2p
+\bar m+5+m'+5+n}}\qquad n=0,1,2,\cdots\end{equation*}
and set $Q_n=K_{R_n}\times(t^{(1)},t^{(2)}]$.
Let $\zeta_n=\zeta_n(x)$ be defined via \eqref{zeta}-\eqref{zetai} with $r$ and $r'$ replaced by $R_{n+1}$ and $R_n$.
Then the cutoff function $\zeta_n$ satisfies $0\leq\zeta_n\leq 1$ in $K_{R_n}$, $\zeta_n\equiv 1$ in $K_{R_{n+1}}$
and $|D\zeta_n|\leq 2^n(2R^{(1)})^{-1}$. Taking into account that $U(K_{R_n},t^{(1)},t^{(2)},(u-k_n)_+\zeta_n^p)=0$,
we write the energy estimate \eqref{Caccioppoli} for $(u-k_n)_+$ over the cylinder $Q_n$ and obtain
\begin{equation*}\begin{split}\esssup_{t^{(1)}<t<t^{(2)}}\int_{K_{R_n}\times\{t\}}&(u-k_n)_+^2\zeta_n^p dx
+\iint_{Q_n}|D(u-k_n)_+\zeta_n|^p dxdt
\\&\leq \gamma_1\left(\frac{\omega}{2^{m_2p+\bar m+5+ m'+4}}\right)^p\frac{2^{np}}{(2R^{(1)})^p}A_n
= \gamma_1\left(\frac{\omega}{2^{m_2p+\bar m+4}}\right)^p\frac{2^{np}}{(2R^{(0)})^p}A_n\left(\frac{2^{-5p-m'p}}{2^p}\right),
\end{split}\end{equation*}
where $A_n=\iint_{Q_n}\chi_{[(u-k_n)_+>0]}dxdt$.
Applying the parabolic Sobolev's inequality \eqref{Sobolevf}, we have
\begin{equation*}\begin{split}\iint_{Q_n}|(u-k_n)_+\zeta_n|^{p\frac{N+2}{N}}&dxdt\leq \gamma
\iint_{Q_n}|D(u-k_n)_+\zeta_n|^p dxdt\left(\esssup_{t^{(1)}<t<t^{(2)}}\int_{K_{R_n}\times\{t\}}(u-k_n)_+^2\zeta_n^p dx\right)^{\frac{p}{N}}
\\&\leq \gamma\gamma_1^{1+\frac{p}{N}} \frac{2^{np(1+\frac{p}{N})}}{L_2^{p(1+\frac{p}{N})}} \left(\frac{\omega^{(3-p)p}}{2^{2pm_2}}\right)^{1+\frac{p}{N}}R^{-p(1+\frac{p}{N})}\left(\frac{1}{2^{\bar m+3}}\right)^{p(1+\frac{p}{N})}A_n^{1+\frac{p}{N}}\left(\frac{2^{-5p-m'p}}{2^p}\right)^{1+\frac{p}{N}}.\end{split}\end{equation*}
To estimate below for the left-hand side, we observe that
\begin{equation*}k_{n+1}-k_n=\frac{\omega}{2^{m_2p+\bar m+5+m'+n+6}}\end{equation*}
and therefore
\begin{equation*}\begin{split}\iint_{Q_n}|(u-k_n)_+\zeta_n|^{p\frac{N+2}{N}}dxdt\geq (k_{n+1}-k_n)^{p\frac{N+2}{N}}A_{n+1}
=\left(\frac{\omega}{2^{m_2p+\bar m+n+6}}\right)^{p\frac{N+2}{N}}A_{n+1}\left(\frac{1}{2^{ m'
+5}}\right)^{p\frac{N+2}{N}}.\end{split}\end{equation*}
Then we deduce that
\begin{equation*}\begin{split}A_{n+1}\leq \gamma\gamma_1^{1+\frac{p}{N}}
2^{n(2p+\frac{2p}{N}+\frac{p^2}{N})}\omega^{-\frac{p^3}{N}+(\frac{3}{N}-1)p^2+2(1-\frac{1}{N})p}&2^{m_2p(p-2)}R^{-p(1+\frac{p}{N})}
\frac{2^{p(1+\frac{p}{N})}4^{p\frac{N+2}{N}}}{L_2^{p(1+\frac{p}{N})}}\left(\frac{1}{2^{\bar m+4}}\right)^{p(\frac{p-2}{N})}
\\&\times A_n^{1+\frac{p}{N}}2^{(m'+5)p\frac{2-p}{N}}2^{-p(1+\frac{p}{N})}.\end{split}\end{equation*}
Taking into account that
\begin{equation*}\begin{split}|Q_n|=\left(2R^{(1)}+\frac{4R^{(1)}}{2^{n}}\right)^N(t^{(2)}-t^{(1)})=(2+2^{2-n})^N\frac{1}{4^N}
\left(\frac{\omega}{2^{m_2}}\right)^{(p-2)N}\delta_1\nu_5
\omega^{(1-p)(2-p)}R^{N+p}(2^NM_1),\end{split}\end{equation*}
then there holds
\begin{equation*}\begin{split}(2^NM_1)\frac{1}{2^N}
\left(\frac{\omega}{2^{m_2}}\right)^{(p-2)N}\delta_1\nu_5
\omega^{(1-p)(2-p)}R^{N+p}\leq|Q_n|\leq (2^NM_1)\frac{3^N}{2^N}
\left(\frac{\omega}{2^{m_2}}\right)^{(p-2)N}\delta_1\nu_5
\omega^{(1-p)(2-p)}R^{N+p}\end{split}\end{equation*}
and
\begin{equation*}\frac{|Q_{n+1}|}{|Q_n|}\geq\frac{1}{3^N}.\end{equation*}
Set $Y_n=A_n/|Q_n|$.
Recalling from the definition of $M_1$ and noting that $0<\nu_5<1$, we deduce
\begin{equation*}\begin{split}Y_{n+1}\leq \frac{3^{N+p}}{2^p}\gamma\gamma_1^{1+\frac{p}{N}}(\delta_1\nu_5)^{\frac{p}{N}}2^{n(2p+\frac{2p}{N}+\frac{p^2}{N})}
&\frac{2^{p(1+\frac{p}{N})}4^{p\frac{N+2}{N}}}{L_2^{p(1+\frac{p}{N})}}\left(\frac{1}{2^{\bar m+4}}\right)^{p(\frac{p-2}{N})}
Y_n^{1+\frac{p}{N}}2^{(m'+5)p\frac{2-p}{N}}2^{-p(1+\frac{p}{N})}(2^NM_1)^{\frac{p}{N}}\\&
\leq \delta_2b^{n}Y_n^{1+\frac{p}{N}},\end{split}\end{equation*}
since
\begin{equation}\label{important relation}2^{(m'+5)p\frac{2-p}{N}}2^{-p(1+\frac{p}{N})}(2^NM_1)^{\frac{p}{N}}
=2^{(m'+5)p\frac{2-p}{N}}2^{-p(1+\frac{p}{N})}2^p\left(\frac{2^p}{2^{(m'-5)(2-p)}}\right)^{\frac{p}{N}}=1.\end{equation}
Then we conclude that if
$Y_0\leq \nu_5=\delta_2^{-\frac{N}{p}}b^{-\frac{N^2}{p^2}}$
then $Y_n\to 0$ as $n\to\infty$. To this end, we choose $\nu=\nu_5$ in \eqref{I2 expand time} and the claim follows.
\end{proof}
\subsection{Iterative arguments: time propagation of positivity from $t^{(k)}$ to $t^{(k+1)}$ ($k\geq2$)}\label{tktk+1}
Starting from \eqref{I2 DeGiorgi}, we can repeat the argument of \S\ref{t1t2} at any time to obtain a sequence of time levels
$\{t^{(k)}\}_{k=1}^\infty$. Here and subsequently,
set
\begin{equation*}R^{(k)}=2^kR^{(0)},\quad s_k=m_2p+\bar m+5+(m'+5)(k-1)\quad\text{and}\quad t^{(k)}=t^{(k-1)}+\delta_1\nu_5\omega^{(1-p)(2-p)}M_1
^{k-1}R^p.\end{equation*}
We now claim that
\begin{equation}\label{Ik DeGiorgi}u(x,t)<\mu_+-\frac{\omega}{2^{s_k}}
\qquad\text{a.e.}\quad\text{in}\quad K_{R^{(k)}}\times(t^{(k-1)},t^{(k)}]\end{equation}
for any $k\geq2$.
\begin{proof}[Proof of \eqref{Ik DeGiorgi}] We first observe from \eqref{I2 DeGiorgi} that \eqref{Ik DeGiorgi} holds with $k=2$.
Assuming \eqref{Ik DeGiorgi} to hold for $k$, we will prove it for $k+1$. We proceed similarly as in \S\ref{t1t2} and divide the proof
into four steps.

{\emph Step 1:}
Let $(x',t')$ be the new variables defined by
\begin{equation*}\begin{split}x'=\frac{x}{20R^{(k)}}\qquad\text{and}\qquad t'=\frac{t-t^{(k)}}{t^{(k)}-t^{(k-1)}}.\end{split}\end{equation*}
This transformation maps $K_{R^{(k)}}\times (t^{(k-1)},t^{(k)})\to K_{\frac{1}{20}}\times (-1,0)$ and
$K_{20R^{(k)}}\times (t^{(k-1)},t^{(k)})\to Q_1$.
We introduce the new functions
\begin{equation*}\begin{split}\tilde u(x',t')=u(x,t)\qquad\text{and}\qquad\bar u(x',t')=\left(\tilde u(x',t')-\mu_+\right)
\left(\frac{2^{s_k}}{\omega}\right). \end{split}\end{equation*}
By induction hypothesis, we see that
$\big|\{x'\in K_1:\bar u(x',t')<-1\}\big|\geq 20^{-N}$
for all $t'\in(-1,0)$.
Set $\tilde w(x',t')=w(x,t)$ and $v(x,t)$ in \eqref{frequently use weak form two phase stefan} can be written in the new variable as
\begin{equation*}
	\tilde v(x',t')=\begin{cases}
	\bar\nu,&\quad \text{on}\quad\{\bar u< -\mu_+\frac{2^{s_k}}{\omega}\},\\
	-\tilde w(x',t'),&\quad \text{on}\quad\{\bar u=-\mu_+\frac{2^{s_k}}{\omega}\}.
	\end{cases}
\end{equation*}
We rewrite the weak form \eqref{frequently use weak form two phase stefan} in terms of the new variables and new functions
\begin{equation}\begin{split}\label{dimensonal less weak form two phase stefan I2}-\frac{2^{s_k}}{\omega}&\int_{K_1} \tilde v(\cdot,t')\chi_{\left\{\bar u\leq -\mu_+\frac{2^{s_k}}{\omega}\right\}}\varphi(\cdot,t')dx'\bigg|_{t'=t_1}^{t_2}+
\frac{2^{s_k}}{\omega}\int_{t_1}^{t_2}\int_{K_1} \tilde v\chi_{\left\{\bar u\leq -\mu_+\frac{2^{s_k
}}{\omega}\right\}}\frac{\partial \varphi}{\partial t'}dx'dt'
\\&\quad+\int_{t_1}^{t_2}\int_{K_1}
\left(\varphi\frac{\partial \bar u}{\partial t'}+\bar A(x',t',\bar u,D\bar u)\cdot D\varphi\right)dx'dt'=0 \end{split}\end{equation}
for any $\varphi\in W_p(Q_1)$ and $[t_1,t_2]\subset (-1,0]$. We check at once that
$\bar A$ satisfies the structure condition
\begin{equation}\label{bar A Ik}
	\begin{cases}
	\bar A(x',t',\bar u,D \bar u)\cdot D \bar u\geq C_0\Lambda_k|D \bar u|^{p},\\
	|\bar A(x',t',\bar u,D \bar u)|\leq C_1\Lambda_k|D \bar u|^{p-1},
	\end{cases}
\end{equation}
where
\begin{equation*}\Lambda_k=\left(\frac{\omega}{2^{s_k}}\right)^{p-2}\frac{t^{(k)}-t^{(k-1)}}{(20R^{(k)})^p}.\end{equation*}
Taking into account the definitions of $s_k$, $R^{(k)}$, $t^{(k)}$, $M_1$ and $\Lambda_2$,
we infer that
\begin{equation}\begin{split}\label{Lambdak}
\Lambda_k&=\left(\frac{\omega}{2^{m_2p+\bar m+5+(m'+5)(k-1)}}\right)^{p-2}
\frac{\delta_1\nu_5\omega^{(1-p)(2-p)}M_1^{k-1}R^p}{20^p2^{kp}4^{-p}L_2^p\left(\frac{\omega}{2^{m_2}}\right)^{(p-2)p}R^p}
\\&=\left(\frac{1}{2^{\bar m+5+(m'+5)(k-1)}}\right)^{p-2}\frac{\delta_1\nu_5}{20^p2^{kp}4^{-p}L_2^p}
\left(\frac{2^p}{2^{(m'+5)(2-p)}}\right)^{k-1}
\\&=\frac{2^{5(2-p)}\delta_1\nu_5}{10^p2^{\bar m(p-2)}L_2^p}=\Lambda_2.
\end{split}\end{equation}
At this point,
the proof now is exactly the same as in the step 1 in \S\ref{t1t2}.
For any $\nu\in(0,1)$, we choose $m'$ as in \eqref{mprime}
and conclude that
\begin{equation*}\big|\{x'\in K_{\frac{3}{4}}: \bar u(x',t')>-2^{-m'}\}\big|\leq \nu|K_{\frac{3}{4}}|\quad\text{for\ \ all}\quad t'\in
\left(-\left(\tfrac{3}{4}\right)^{p},0\right].\end{equation*}

{\emph Step 2:} With the previous result at hand, we proceed similarly as in the step 1 in \S\ref{t1t2}
to conclude that
there exists a time level $t_\omega^{(k)}\in (t^{(k-1)}, t^{(k)})$, such that
\begin{equation}\label{Ik space expanding I2}u(x,t)<\mu_+-\frac{\omega}{2^{s_k+m'
+1}}\qquad\text{a.e.}\quad\text{in}\quad(0,t^{(k)})+ Q\left( \frac{t^{(k)}-t_\omega^{(k)}}{2^p},10R^{(k)}\right),\end{equation}
where $m'$ is the same constant as in \eqref{prime m I2}.

{\emph Step 3:} For any $\nu\in(0,1)$, we claim that there exists a time level
$$t^{(k+1)}=t^{(k)}+\delta_1\nu\omega^{(1-p)(2-p)}M_1
^{k}R^p$$
such that
\begin{equation}\label{Ik I2 expand time}\big|\{x\in K_{6R^{(k)}}:u(x,t)>\mu_+-\frac{\omega}{2^{s_k+ m'+4}}\}\big|
<\nu|K_{6R^{(k)}}|\end{equation}
for any $t\in (t^{(k)}, t^{(k+1)}]$.
\begin{proof}[Proof of the claim] Let $k'=\mu_+-\frac{\omega}{2^{s_k+ m'+1}}$ and $c=\frac{\omega}{2^{s_k+ m'+4}}$.
We consider the logarithmic function
\begin{equation*}\begin{split}\psi^+=\ln^+\left(\frac{\frac{\omega}{2^{s_k+ m'+1}}}{\frac{\omega}{2^{s_k
+ m'+1}}-(u-k')_++c}\right).
\end{split}\end{equation*}
Taking into account the definition of $s_k$, we have $\psi^+\leq 8\ln2$ and
\begin{equation*}\begin{split}[(\psi^+)']^{2-p}\leq \left(\frac{1}{c}\right)^{2-p}
\leq \left(\frac{\omega}{2^{m_2p}}\right)^{p-2}2^{(\bar m+5+(m'+5)(k-1)+m'+4)(2-p)}.\end{split}\end{equation*}
Let $\zeta=\zeta(x)$ be defined via \eqref{zeta}-\eqref{zetai} with $r$ and $r'$ replaced by $6R^{(k)}$ and $10R^{(k)}$.
It is easily seen that $0\leq \zeta\leq1$ in $K_{10R^{(k)}}$,
$\zeta\equiv 1$ in $K_{6R^{(k)}}$ and $|D\zeta|\leq (4R^{(k)})^{-1}$. It follows from \eqref{Ik space expanding I2} that
$\int_{K_{10R^{(k)}}\times\{t^{(k)}\}}\left(\psi^{+}\right)^2\zeta^p dx=0$.
Taking into account that $U(K_{10R^{(k)}},t^{(k)},t^{(k+1)},2\psi^{+}\left(\psi^{+}\right)^\prime\zeta^p)=0$,
we obtain from \eqref{logarithmic} the logarithmic estimate
\begin{equation*}\begin{split}\esssup_{t^{(k)}<t<t^{(k+1)}}&\int_{K_{10R^{(k)}}\times\{t\}}\left(\psi^{+}\right)^2\zeta^p dx\leq
\gamma_2\int_{t^{(k)}}^{t^{(k+1)}}\int_{
K_{10R^{(k)}}}\psi^{+}|\left(\psi^{+}\right)^\prime|^{2-p}
|D\zeta|^pdxdt.\end{split}\end{equation*}
Keeping in mind the definitions of $\delta_1$ and $M_1$, we obtain
\begin{equation*}\begin{split}\esssup_{t^{(k)}<t<t^{(k+1)}}&\int_{K_{10R^{(k)}}\times\{t\}}\left(\psi^{+}\right)^2\zeta^p dx
\\&\leq
8\gamma_2\left(\frac{\omega}{2^{m_2p}}\right)^{p-2}2^{(\bar m+4)(2-p)}(4R^{(0)})^{-p}(t^{(k+1)}-t^{(k)})
(\ln2)|K_{10R^{(k)}}|\left(\frac{2^{(5+m')(2-p)k}}{2^{kp}}\right)
\\&=\nu M_1^k\left(\frac{2^{(5+m')(2-p)k}}{2^{kp}}\right)|K_{6R^{(k)}}|
=\nu |K_{6R^{(k)}}|.\end{split}\end{equation*}
On the other hand, we introduce the smaller set
$$\{x\in K_{6R^{(k)}}:u(x,t)>\mu_+-\frac{\omega}{2^{s_k+m'+4}}\}\subset K_{6R^{(k)}}.$$
In this set, $\zeta\equiv 1$ and $\psi^+\geq 2\ln2>1$. Then we conclude that
\begin{equation*}\begin{split}\int_{K_{10R^{(1)}}\times\{t\}}\left(\psi^{+}\right)^2\zeta^p dx
\geq
\big|\{x\in K_{6R^{(1)}}:u(x,t)>\mu_+-\frac{\omega}{2^{s_k+m'+4}}\}\big|
\end{split}\end{equation*}
for all $t\in (t^{(k)},t^{(k+1)})$ and the estimate \eqref{Ik I2 expand time} follows.
\end{proof}

{\emph Step 4:} Proof of the estimate \eqref{Ik DeGiorgi} for $k+1$.
For $n=0,1,2,\cdots$, set
\begin{equation*}R_n=2R^{(k)}+\frac{4R^{(k)}}{2^{n}},\quad k_n=\mu_+-\frac{\omega}{2^{s_k+ m'+5}}-\frac{\omega}{2^{s_k+m'+5+n}}\quad
\text{and}\quad Q_n=K_{R_n}\times(t^{(k)},t^{(k+1)}].\end{equation*}
Let $\zeta_n=\zeta_n(x)$ be defined via \eqref{zeta}-\eqref{zetai} with $r$ and $r'$ replaced by $R_{n+1}$ and $R_n$.
Then the cutoff function $\zeta_n$ satisfies $0\leq\zeta_n\leq 1$ in $K_{R_n}$, $\zeta_n\equiv 1$ in $K_{R_{n+1}}$
and $|D\zeta_n|\leq 2^n(2R^{(k)})^{-1}$. Taking into account that $U(K_{R_n},t^{(k)},t^{(k+1)},(u-k_n)_+\zeta_n^p)=0$,
we write the energy estimate \eqref{Caccioppoli} for $(u-k_n)_+$ over the cylinder $Q_n$ and obtain
\begin{equation*}\begin{split}\esssup_{t^{(k)}<t<t^{(k+1)}}\int_{K_{R_n}\times\{t\}}&(u-k_n)_+^2\zeta_n^p dx
+\iint_{Q_n}|D(u-k_n)_+\zeta_n|^p dxdt
\\&\leq \gamma_1\left(\frac{\omega}{2^{s_k+ m'+4}}\right)^p\frac{2^{np}}{(2R^{(k)})^p}A_n
= \gamma_1\left(\frac{\omega}{2^{m_2p+\bar m+4}}\right)^p\frac{2^{np}}{(2R^{(0)})^p}A_n\left(\frac{2^{-(5+m')kp}}{2^{kp}}\right),
\end{split}\end{equation*}
where $A_n=\iint_{Q_n}\chi_{[(u-k_n)_+>0]}dxdt$.
Applying the parabolic Sobolev's inequality \eqref{Sobolevf}, we have
\begin{equation*}\begin{split}\iint_{Q_n}&|(u-k_n)_+\zeta_n|^{p\frac{N+2}{N}}dxdt\leq \gamma
\iint_{Q_n}|D(u-k_n)_+\zeta_n|^p dxdt\left(\esssup_{t^{(k)}<t<t^{(k+1)}}\int_{K_{R_n}\times\{t\}}(u-k_n)_+^2\zeta_n^p dx\right)^{\frac{p}{N}}
\\&\quad\leq \gamma\gamma_1^{1+\frac{p}{N}} \frac{2^{np(1+\frac{p}{N})}}{L_2^{p(1+\frac{p}{N})}} \left(\frac{\omega^{(3-p)p}}{2^{2pm_2}}\right)^{1+\frac{p}{N}}R^{-p(1+\frac{p}{N})}\left(\frac{1}{2^{\bar m+3}}\right)^{p(1+\frac{p}{N})}A_n^{1+\frac{p}{N}}\left(\frac{2^{-(5+m')kp}}{2^{kp}}
\right)^{1+\frac{p}{N}}.\end{split}\end{equation*}
To estimate below for the left-hand side, we observe that
\begin{equation*}k_{n+1}-k_n=\frac{\omega}{2^{s_k+m'+n+6}}=
\frac{\omega}{2^{m_2p+\bar m+n+6}}\frac{1}{2^{(m'+5)k}}\end{equation*}
and therefore
\begin{equation*}\begin{split}\iint_{Q_n}|(u-k_n)_+\zeta_n|^{p\frac{N+2}{N}}dxdt\geq (k_{n+1}-k_n)^{p\frac{N+2}{N}}A_{n+1}
=\left(\frac{\omega}{2^{m_2p+\bar m+n+6}}\right)^{p\frac{N+2}{N}}A_{n+1}\left(\frac{1}{2^{ (m'
+5)k}}\right)^{p\frac{N+2}{N}}.\end{split}\end{equation*}
Then we deduce that
\begin{equation*}\begin{split}A_{n+1}\leq \gamma\gamma_1^{1+\frac{p}{N}}
2^{n(2p+\frac{2p}{N}+\frac{p^2}{N})}\omega^{-\frac{p^3}{N}+(\frac{3}{N}-1)p^2+2(1-\frac{1}{N})p}&2^{m_2p(p-2)}R^{-p(1+\frac{p}{N})}
\frac{2^{p(1+\frac{p}{N})}4^{p\frac{N+2}{N}}}{L_2^{p(1+\frac{p}{N})}}\left(\frac{1}{2^{\bar m+4}}\right)^{p(\frac{p-2}{N})}
\\&\times A_n^{1+\frac{p}{N}}2^{k(m'+5)p\frac{2-p}{N}}2^{-kp(1+\frac{p}{N})}.\end{split}\end{equation*}
Taking into account that
\begin{equation*}\begin{split}|Q_n|=\left(2R^{(k)}+\frac{4R^{(k)}}{2^{n}}\right)^N(t^{(k+1)}-t^{(k)})=(2+2^{2-n})^N\frac{1}{4^N}
\left(\frac{\omega}{2^{m_2}}\right)^{(p-2)N}\delta_1\nu_5
\omega^{(1-p)(2-p)}R^{N+p}(2^NM_1)^k,\end{split}\end{equation*}
then there holds
\begin{equation*}\begin{split}(2^NM_1)^k\frac{1}{2^N}
\left(\frac{\omega}{2^{m_2}}\right)^{(p-2)N}\delta_1\nu_5
\omega^{(1-p)(2-p)}R^{N+p}\leq|Q_n|\leq (2^NM_1)^k\frac{3^N}{2^N}
\left(\frac{\omega}{2^{m_2}}\right)^{(p-2)N}\delta_1\nu_5
\omega^{(1-p)(2-p)}R^{N+p}\end{split}\end{equation*}
and also
\begin{equation*}\frac{|Q_{n+1}|}{|Q_n|}\geq\frac{1}{3^N}.\end{equation*}
Set $Y_n=A_n/|Q_n|$.
Recalling from the definition of $M_1$ and noting that $\nu_5\in(0,1)$, we deduce from \eqref{important relation} that
\begin{equation*}\begin{split}Y_{n+1}&\leq \frac{3^{N+p}}{2^p}\gamma\gamma_1^{1+\frac{p}{N}}(\delta_1\nu_5)^{\frac{p}{N}}2^{n(2p+\frac{2p}{N}+\frac{p^2}{N})}
\frac{2^{p(1+\frac{p}{N})}4^{p\frac{N+2}{N}}}{L_2^{p(1+\frac{p}{N})}}\left(\frac{1}{2^{\bar m+4}}\right)^{p(\frac{p-2}{N})}
Y_n^{1+\frac{p}{N}}
\\&\qquad\qquad\times \left(2^{(m'+5)p\frac{2-p}{N}}2^{-p(1+\frac{p}{N})}(2^NM_1)^{\frac{p}{N}}\right)^k\\&
\leq \delta_2b^{n}Y_n^{1+\frac{p}{N}},\end{split}\end{equation*}
Then we conclude that if
$Y_0\leq \nu_5=\delta_2^{-\frac{N}{p}}b^{-\frac{N^2}{p^2}}$
then $Y_n\to 0$ as $n\to\infty$. To this end, we choose $\nu=\nu_5$ in \eqref{Ik I2 expand time} and the estimate \eqref{Ik DeGiorgi} follows
 for $k+1$.
\end{proof}
\subsection{Analysis of the second alternative concluded}
Now we are ready to proceed to obtain a DeGiorgi-type result similar to Proposition \ref{1st main result}.
We conclude from \eqref{DeGiorgi5}, \eqref{DeGiorgiI17} and \eqref{Ik DeGiorgi} that for any $j\geq 0$ there holds
\begin{equation}\label{DeGiorgi final}u(x,t)<\mu_+-\frac{\omega}{2^{s_j}}\qquad\text{a.e.}\quad\text{in}\quad
K_{\frac{1}{4}c_2R}\times(\tilde t,t^{(j)}].\end{equation}
Note that the time level $t^{(j)}$ should be taken so large that $\bar t\in (\tilde t,t^{(j)}]$.
For this purpose, we have to impose a condition that $M_1\geq1$. Taking into account that $m'$ is uniformly bounded in $p$,
we set $m''=\sup_{\frac{3}{2}<p\leq2}m'$ and impose a condition for $p$ as follows
\begin{equation}\label{artificial p}\frac{2(m''+5)}{m''+6}\leq p<2.\end{equation}
In this case, we see that $t^{(j)}\to +\infty$ as $j\to +\infty$ and we can find an integer $j_*$ with the expression
\begin{equation*}j_*=\left[\frac{2^{m_1(p-1)(2-p)}}{\delta_1\nu_5 }\right]+1
\end{equation*}
such that $\bar t\in (\tilde t,t^{(j_*)}]$.

In order to perform the iteration, we observe that any subcylinder
should be contained in the reference cylinder $\overline Q=Q(R,R^{\frac{1}{p}})\subset\Omega_T$. To this end,
we conclude that the estimate \eqref{DeGiorgi final} holds for $j=j_*$
under the assumption that $2^{j_*}\leq 20^{-1}R^{-\kappa}$ and $2^{-m_2}\omega \geq R$, where $\kappa=\frac{1}{2}(p-1-\frac{1}{p})$.
However, if $2^{j_*}> 20^{-1}R^{-\kappa}$, then there exists a constant $C'$ depending only upon the data and $p$ such that
\begin{equation*} \exp\left\{\exp (C'm_2)\right\}\geq
 R^{-\kappa}.\end{equation*}
According to \eqref{m2f}, there exist positive constants $\bar C$ and $\bar\alpha$ depending only upon the data and $p$ such that
\begin{equation*}\omega\leq \left(\frac{\bar C}{\ln\ln\ln\ln R^{-\kappa}}\right)^{\bar \alpha}.\end{equation*}
We observe that $\omega\downarrow 0$ as $R\downarrow 0$.

To proceed further, we now turn our attention to the case $K_{20R^{(j_*)}}\subset K_{R^{\frac{1}{p}}}$. If there exists $j\in\mathbb{N}^+$
such that $\bar t=t^{(j)}$, then \eqref{DeGiorgi final} is the desired estimate for the second alternative. If $\bar t\neq t^{(j)}$
for any $j\geq0$, then there exists $\bar j\in\mathbb{N}^+$
such that $t^{(\bar j)}<\bar t<t^{(\bar j+1)}$ and $\bar j+1<j_*$.
The iteration scheme only need a slight modification. We perform the iterative
arguments $\bar j$ times to obtain \eqref{DeGiorgi final} for $j=\bar j$. Starting from this estimate, we repeat the arguments from
step 1 to step 2 in \S\ref{tktk+1} to obtain the estimate \eqref{Ik I2 expand time} with $k=\bar j$. Furthermore, we can repeat the
the arguments in
step 3 and step 4 in \S\ref{tktk+1} with $t^{(k+1)}=\bar t$, since the proof still works with the condition
$t^{(k+1)}=t^{(k)}+\delta_1\nu_5\omega^{(1-p)(2-p)}M_1
^{k}R^p$ replaced by
$$t^{(k)}<t^{(k+1)}< t^{(k)}+\delta_1\nu_5\omega^{(1-p)(2-p)}M_1
^{k}R^p.$$
Then we conclude that the DeGiorgi-type
estimate \eqref{DeGiorgi final} holds in the cylinder $K_{\frac{1}{4}c_2R}\times(\tilde t,\bar t]$,
with $s_j$ replaced by $s_{\bar j+1}$.

In conclusion, we have proved the following proposition.
\begin{proposition}\label{2nd alt pro}
Suppose that the assumptions \eqref{R} and \eqref{artificial p} are in force. Then there exists a constant $s$ depending only upon the data and $\omega$, such that, either
\begin{equation*}u(x,t)<\mu_+-\frac{\omega}{2^{s}}\qquad\text{a.e.}\quad\text{in}\quad
Q\left(d\left(\frac{R}{2}\right)^p,d_*\left(\frac{R}{2}\right)\right),\quad \text{or}\qquad\omega\leq \left(\frac{\bar C}{\ln\ln\ln\ln R^{-\kappa}}\right)^{\bar \alpha}\end{equation*}
for some constants $\kappa$, $\bar C$ and $\bar\alpha$ depending only upon the data and $p$. Moreover, $s\uparrow +\infty$ as $\omega\downarrow 0$.
\end{proposition}
\section{Proof of the main result concluded}
As we have already discussed in \S\ref{The intrinsic geometry},
the strategy of the proof is to study oscillation of the weak solution
in a sequence of nested and shrinking cylinders with common vertex and
prove that the essential oscillation converges to zero. We follow the notation used in \S\ref{The intrinsic geometry} and assume that
the common vertex coincides with $(0,0)$. To start with, we
set $R_1=R$,
\begin{equation*}Q_1=Q\left(\left(\frac{\omega}{2^{m_1}}\right)^{(1-p)(2-p)}\left(\frac{R_1}{2}\right)^p,
\left(\frac{\omega}{2^{m_2}}\right)^{p-2}\left(\frac{R_1}{2}\right)\right)\qquad\text{and}
\qquad \overline Q_1=Q(R_1,R_1^{\frac{1}{p}}).\end{equation*}
From Proposition \ref{1st main result} and \ref{2nd alt pro}, we conclude that
there exists $\sigma(\omega)$ such that either
\begin{equation*}\essosc_{Q_1}u\leq\sigma(\omega)\omega,\quad Q_1\subset \overline Q_1,\qquad\text{or}\qquad
\omega\leq \left(\frac{\bar C}{\ln\ln\ln\ln R_1^{-\kappa}}\right)^{\bar \alpha}\end{equation*}
and $\sigma(\omega)$ satisfies $\sigma(\omega)<1$ and $\sigma(\omega)\uparrow 1$ as $\omega\downarrow 0$.
Next, we set $\omega_1=\sigma(\omega)\omega$, $R_2=2^{-p}R_1^p$ and construct the reference parabolic cylinder $\overline Q_2=Q(R_2,R_2^{\frac{1}{p}})$.
We see that $\overline Q_2\subset Q_1\cap \overline Q_1$ and $\essosc_{\overline Q_2}u\leq \omega$. Moreover, we choose the cylinder
$Q_2$ by
\begin{equation*}Q_2
=\begin{cases}
Q\left(\left(\frac{\omega_1}{2^{m_1(\omega_1)}}
\right)^{(1-p)(2-p)}\left(\frac{R_2}{2}\right)^p,\left(\frac{\omega_1}{2^{m_2(\omega_1)}}\right)^{p-2}\left(\frac{R_2}{2}\right)
\right),&\quad\text{if}\quad \essosc_{\overline Q_2}u\leq \omega_1,\\
Q\left(\left(\frac{\omega}{2^{m_1(\omega)}}
\right)^{(1-p)(2-p)}\left(\frac{R_2}{2}\right)^p,\left(\frac{\omega}{2^{m_2(\omega)}}\right)^{p-2}\left(\frac{R_2}{2}\right)
\right),&\quad\text{if}\quad \essosc_{\overline Q_2}u> \omega_1.
	\end{cases}
\end{equation*}
At this point, we apply Proposition \ref{1st main result} and \ref{2nd alt pro} again, with $Q_1$ and $\overline Q_1$ replaced by
$Q_2$ and $\overline Q_2$, to obtain
either
\begin{equation*}\essosc_{Q_2}u\leq\max\{\sigma(\omega_1)\omega_1,\omega_1\}=\omega_1,\quad Q_2\subset\overline Q_2,\end{equation*}
\begin{equation*}\text{or}\quad\omega_1\leq \left(\frac{\bar C}{\ln\ln\ln\ln R_2^{-\kappa}}\right)^{\bar \alpha}
\quad\text{and}\quad \essosc_{\overline Q_2}u\leq \omega_1,
\quad\text{or}\quad\omega\leq \left(\frac{\bar C}{\ln\ln\ln\ln R_2^{-\kappa}}\right)^{\bar \alpha}.\end{equation*}
We continue this process to find two sequences $\{R_n\}_{n=1}^\infty$ and $\{\omega_n\}_{n=1}^\infty$ such that
\begin{equation*}R_{n+2}=2^{-p}R_{n+1}^p\qquad\text{and}\qquad\omega_{n+1}=\sigma(\omega_n)\omega_n,\qquad n=1,2,\cdots.\end{equation*}
Then we have $R_n\downarrow 0$ as $n\to\infty$.
From \cite[Page 222]{U1997}, we see that $\omega_n\downarrow 0$ as $n\to\infty$.
With these choices, we define $\overline Q_{n+1}=Q(R_{n+1},R_{n+1}^{\frac{1}{p}})$ and determine the cylinder $Q_{n+1}$ by
\\
\begin{equation*}Q_{n+1}
=\begin{cases}
Q\left(\left(\frac{\omega_n}{2^{m_1(\omega_n)}}
\right)^{(1-p)(2-p)}\left(\frac{R_{n+1}}{2}\right)^p,\left(\frac{\omega_n}{2^{m_2(\omega_n)}}\right)^{p-2}\left(\frac{R_{n+1}}{2}\right)
\right),&\text{if}\quad \essosc_{\overline Q_{n+1}}u\leq \omega_n,\\
Q\left(\left(\frac{\omega_{n-1}}{2^{m_1(\omega_{n-1})}}
\right)^{(1-p)(2-p)}\left(\frac{R_{n+1}}{2}\right)^p,\left(\frac{\omega_{n-1}}{2^{m_2(\omega_{n-1})}}\right)^{p-2}
\left(\frac{R_{n+1}}{2}\right)\right),&\text{if}\quad \omega_n<\essosc_{\overline Q_{n+1}}u\leq \omega_{n-1},\\
\qquad\qquad\qquad\qquad\qquad\vdots \\
Q\left(\left(\frac{\omega}{2^{m_1(\omega)}}
\right)^{(1-p)(2-p)}\left(\frac{R_{n+1}}{2}\right)^p,\left(\frac{\omega}{2^{m_2(\omega)}}\right)^{p-2}\left(\frac{R_{n+1}}{2}\right)
\right),&\text{if}\quad \omega_1<\essosc_{\overline Q_{n+1}}u\leq\omega,
	\end{cases}
\end{equation*}
\\
where $n=2,3,\cdots.$ These cylinders are nested with common vertex $(0,0)$ and shrinking to this point as $n\to\infty$.
Next, we remark that there is no any subsequence $\{n_k\}_{k=1}^\infty$ such that $n_k\to\infty$ and
\begin{equation*}\omega \leq
\left(\frac{\bar C}{\ln\ln\ln\ln R_{n_k}^{-\kappa}}\right)^{\bar \alpha},\qquad k=1,2,\cdots.\end{equation*}
This is because if this estimate holds true for any $k=1,2,\cdots$, then $\omega=0$, which contradicts to the assumption $\omega>0$.
Repeated application of Proposition \ref{1st main result} and \ref{2nd alt pro}, we conclude that
there exist subsequences $\{n_k\}_{k=1}^\infty$ and $\{n_k'\}_{k=1}^\infty$ such that either
\begin{equation*}\essosc_{Q_{n_k}}u\leq \omega_{n_k'},\quad\text{or}\quad
\essosc_{\overline Q_{n_k}}u\leq \omega_{n_k'}\leq
\left(\frac{\bar C}{\ln\ln\ln\ln R_{n_k}^{-\kappa}}\right)^{\bar \alpha}.\end{equation*}
Therefore, we conclude that either $\essosc_{Q_{n_k}}u\downarrow 0$ or $\essosc_{\overline Q_{n_k'}}u\downarrow 0$ as $k\to\infty$. This proves the theorem.
\section*{Acknowledgement}
The author wishes to thank  Eurica Henriques, Peter Lindqvist, Irina Markina and Jos\'e Miguel Urbano for the valuable discussions.

\bibliographystyle{abbrv}

\end{document}